\documentclass[a4paper,12pt]{amsart}
\usepackage{amssymb}
\usepackage{amsmath,amssymb,amsfonts,amsthm}
\usepackage[abbrev,nobysame]{amsrefs}
\usepackage{mathtools}\usepackage{cases}

\usepackage{enumerate}

\newtheorem{thm}{Theorem}[section]
\newtheorem{lem}[thm]{Lemma}
\newtheorem{prop}[thm]{Proposition}
\newtheorem{cor}[thm]{Corollary}

\newtheorem{ques}[thm]{Question}

\theoremstyle{definition}
\newtheorem{defn}{Definition}[section]

\theoremstyle{remark}
\newtheorem{rem}{Remark}[section]

\makeatletter \@addtoreset{equation}{section}

\newcommand{\thmref}[1]{Theorem~\ref{#1}}

\newcommand{\lemref}[1]{Lemma~\ref{#1}}
\newcommand{\corref}[1]{Corollary~\ref{#1}}
\newcommand{\propref}[1]{Proposition~\ref{#1}}

\def\a{\alpha}
\def\b{\beta}

\def\ep{\epsilon}

\def\la{\lambda}

\def\p{\partial}
\def\vphi{\varphi}

\def\cD{{\cal D}}

\def\cD{\mathcal D}

\def\Int{{\rm Int}\,}

\def\and{\quad{\rm and}\quad}

\def\eps{\epsilon}

\let\lra=\longrightarrow

\def\mapright{\xrightarrow}
\def\mapright\#1{\,\smash{\mathop{\lra}\limits^{\#1}}\,}

\def\om{\omega}
\def\Om{\Omega}

\def\tri{\triangle}
\def\ul{\underline}

\newcommand{\norm}[1]{\left\Vert#1\right\Vert}
\newcommand{\abs}[1]{\left\vert#1\right\vert}
\newcommand{\set}[1]{\left\{#1\right\}}

\newcommand{\pfrac}[2]{\frac{\partial #1}{\partial #2}}

\numberwithin{equation}{section}

\DeclareMathOperator{\osc}{Osc}
\DeclareMathOperator{\Tr}{Tr}

\let\norm=\enVert
\makeatletter

\newcommand{\Rmnum}[1]{\expandafter\@slowromancap\romannumeral #1@}

\allowdisplaybreaks[1]     
\makeatletter
\newcommand{\pushright}[1]{\ifmeasuring@#1\else\omit\hfill$\displaystyle#1$\fi\ignorespaces}  
\newcommand{\pushleft}[1]{\ifmeasuring@#1\else\omit$\displaystyle#1$\hfill\fi\ignorespaces}
\makeatother

\title[Uniqueness of cscK cone metrics]{Geodesics in the space of K\"ahler cone metrics, II. Uniqueness of constant scalar curvature K\"ahler cone metrics}

\author [Kai Zheng]{Kai Zheng}
  \address{Mathematics Institute, University of Warwick, Coventry CV4 7AL, UK}
  \email{K.Zheng@warwick.ac.uk}

\begin{document}
\maketitle
\begin{abstract} This is a continuation of the previous articles on K\"ahler cone metrics. In this article, we introduce weighted function spaces and provide a self-contained treatment on cone angles in the whole interval $(0,1]$. We first construct geodesics in the space of K\"ahler cone metrics (cone geodesics). We next determine the very detailed asymptotic behaviour of constant scalar curvature K\"ahler (cscK) cone metrics, which leads to the reductivity of the automorphism group. Then we establish the linear theory for the Lichnerowicz operator, which immediately implies the openness of the path deforming the cone angles of cscK cone metrics. Finally, we address the problem on the uniqueness of cscK cone metrics and show that the cscK cone metric is unique up to automorphisms.
\end{abstract}
\tableofcontents
\section{Introduction}

Geodesics in the space of K\"ahler metrics, first established in Chen \cite{MR1863016}, are fundamental geometric objects in K\"ahler geometry and link constant scalar curvature K\"ahler (cscK) metrics in differential geometry to geometric invariant theory, see the foundational articles by Calabi-Chen \cite{MR1969662}, and also Donaldson \cite{MR1736211,MR2103718}.
One main aim of this article is to construct geodesics in the space of K\"ahler metrics with cone singularities (cone geodesics), \thmref{geo existence} below. To a large extend, there is a continuously growing interest in studying the constant scalar curvature K\"ahler metrics with cone singularities, see Chen \cite{arXiv:1506.06423}, due to the up-to-date achievement in K\"ahler geometry. One another aim of this article is to cooperate the cone geodesics, roughly speaking, with the constant scalar curvature K\"ahler metrics with cone singularities (cscK cone metrics, Definition \ref{csckconemetricdefn}), \thmref{Uniqueness}, in the general conjectural picture. We refer to the excellent review articles by Chen \cite{arXiv:1506.06423} and Donaldson \cite{MR3381498} for many more references and a more detailed updated account of the development of the cscK metric problem.


The approach we take in the present article is to solve a Dirichlet problem for homogeneous complex Monge-Amp\`ere (HCMA) equation with cone boundary values. Let $D$ be a smooth divisor in a smooth closed $n$-dimensional K\"ahler manifold $X$ and $\om_0$ be a smooth K\"ahler metric in $X$.
In this article, the cone angle is assumed to be $$0 <\b\leq 1.$$ Let $\mathcal H_{\b}$ be the space of K\"ahler cone metrics of cone angle $\b$, which are K\"ahler metrics in $M:=X\setminus D$ and have cone singularities of angle $\b$ along the divisor $D$. The geodesic segment $\{\vphi(t);0\leq t\leq 1\}$ connecting two K\"ahler cone metrics $\{\vphi_i\in\mathcal H_\beta, i=0,1\}$,
turns out to be a HCMA on the product manifold $M\times R  $,
\begin{align}\label{geo smooth}
{(\Om_{0}+i\p\bar\p\Psi)^{n+1}=0}.
\end{align} 
Here $R=S^1\times [0,1]$ is a cylinder with boundary, and $\Om_0, \Psi$ is the pull-back metric of $\om_0$, $\vphi$ respectively
under the natural projection. 

There is extensive literature on Dirichlet problem for homogeneous complex Monge-Amp\`ere equation. We will extend the method developed in Chen \cite{MR1863016} and He \cite{MR3298665} to our cone geodesics.
In our case, the source of difficulties are twofold: not only the right hand side is degenerate, but also the boundary values are singular. They are the main technical difficulties, in particular for the regularity theory. We first construct the generalised cone geodesic, which has bounded spatial Laplacian, see Definition \ref{ctribw}.

\begin{thm}[Cone geodesic]\label{geo existence}
	Suppose that $0<\b\leq1$ and $\om_1,\om_2$ are two K\"ahler cone metrics in $\mathcal H_{\b}$. Then there exists a unique generalised cone geodesic connecting them. 
	
\end{thm}

We then consider more regularity of the cone geodesic, i.e. the full Laplacian.
In \cite{MR3405866}, we restricted the cone angle less than $\frac{1}{2}$ and constructed the $C^{1,1,\b}$ cone geodesic between K\"ahler cone metrics with bounded geometry, i.e. in the space $\mathcal H_C$. The boundary values in $\mathcal H_C$ were used to construct the background metric with bounded Christoffel symbols and Ricci upper bound. The background metric had two functions in the construction of $C^{1,1,\b}$ cone geodesic, it served as initial solution for the continuity method, and the a priori estimates required the geometric conditions of the background metric. We show a different approach in this paper. In Section \ref{Cone geodesic}, we use approximation method and choose different background metrics, as a result, we improve the a priori estimates to relax these geometric conditions \eqref{Om1backgroundmetricChristoffelsymbols}, \eqref{backgroundmetricChristoffelsymbols} and \eqref{backgroundmetriccurvature}; and we construct the $C^{1,1,\b}_{\mathbf w}$ cone geodesics. The precise statements are \thmref{geodesicclosedness} and \thmref{weakgeodesicclosedness}.

\bigskip

In Section \ref{Asymptotic behaviours}, we describe the regularity of the constant scalar curvature K\"ahler cone metrics. 
The following regularity theorem is a loose statement of this result, a full description of the higher order asymptotic of cscK cone metrics are presented in Section \ref{Asymptotic behaviours}, we refer to \thmref{thm:interior} for complete statement.

\begin{thm}[Regularity of cscK cone metrics]\label{csck geo existence}
	Suppose that $0<\b\leq1$ and the H\"older exponent $\alpha$ satisfies 
\begin{align*} 
\a\b<1-\b.
\end{align*} Let $\omega$ be a constant scalar curvature K\"ahler cone metric. Then $\vphi$ is $C^{3,\a,\b}_{\bf w}$. 
\end{thm}
\begin{rem}
The letter '$\bf w$' in $C^{3,\a,\b}_{\bf w}$ stands for "weak". It will become clear in Section \ref{3rdHolderspaces} why we chose this terminology.
\end{rem}
\begin{rem}
\thmref{csck geo existence} is a simple corollary of our main regularity \thmref{thm:interior}. 
\end{rem}  
The strategy pursued in this work is to scale the K\"ahler potential $\vphi$, and then to derive
the estimates by using techniques inspired from our previous paper \cite{arxiv:1609.03111} on asymptotic analysis of complex Monge-Amp\`ere equation with cone singularities. In that article, more result is proved, that is the expansion formula obtained simply as a consequence of the asymptotic analysis. Related difficult studies for K\"ahler-Einstein cone metrics could be found in the references therein, e.g. \cite{MR3432582,MR3668765,MR3488129,MR3431580}, etc, however, we have a more explicit description of the asymptotic, though.

\bigskip

In Section \ref{Linear theory for Lichnerowicz operator}, we investigate the linearisation operator of the cscK cone metric, i.e. the Lichnerowicz operator. 
We introduce a new space $C_\mathbf{w}^{4,\a,\b}(\om)$ (Definition \ref{cw4ab}). We show that this is the right space to have Fredholm alternative theorem for the Lichnerowicz operator. 
\begin{thm}[Linear theory]\label{Fredholm}
	Suppose that $0<\b\leq1$ and $\a\b<1-\b$. Let $\omega$ be a constant scalar curvature K\"ahler cone metric.
    Assume that $f\in C^{0,\a,\b}$ with normalisation condition $\int_X f\om^n =0$. Then one of the following holds:
	\begin{itemize}
		\item Either the Lichnerowicz equation ${{\mathbb{L}\mathrm{ic}}}_{\om}(u)=f$ has a unique $C_\mathbf{w}^{4,\a,\b}(\om)$ solution.
		\item Or the kernel of ${{\mathbb{L}\mathrm{ic}}}_{\om}(u)$ generates a holomorphic vector field tangent to $D$.
	\end{itemize}
\end{thm}

\begin{rem}
The linear theory implies the openness of the path, which deforms the cone angle of cscK cone metrics, by following the implicit function theorem argument in Section 4.4 in Donaldson's paper \cite{MR2975584}, 
\end{rem}

The proof of the linear theory updates the method developed in our paper \cite{arxiv:1703.06312}, where this theorem was proved under angle restriction $0<\b<\frac{1}{2}$, but in the space $C^{4,\a,\b}(\om)$ which has better regularities. The key steps of the proof of \thmref{Fredholm} can be summarised as follows:  
\begin{itemize}
	\item Solving the perturbed bi-Laplacian equation \eqref{KbiLaplacian equation} in $C_\mathbf{w}^{4,\a,\b}(\om)$ by developing a $L^2$-theory for the 4th order elliptic PDEs with cone leading coefficients.
	\item Building the continuity path connecting the perturbed Lichnerowicz equation with the perturbed bi-Laplacian equation and showing that along the continuity path, the $C_\mathbf{w}^{4,\a,\b}(\om)$ estimates hold uniformly.
\end{itemize}

\bigskip

In Section \ref{Reductivity of automorphism group}, Calabi's reductivity theorem of the automorphism group to cscK metric is generalised to cscK metric with cone singularities with full angle $0<\b\leq1$.
In our previous paper \cite{arxiv:1603.01743}, we proved this theorem under the half angle condition $0<\b<\frac{1}{2}$, after proving the regularity of the half angle cscK cone metric. 
New ingredients in this part are the asymptotic of Christoffel symbols of cscK cone metric in Section \ref{Christoffel symbols of cscK cone metric} and asymptotic of functions in $C_\mathbf{w}^{4,\a,\b}(\om)$ in Section \ref{Asymptotics of functions}.

Our precise result is the content of \thmref{preceisereductivity}, let us state it loosely as follows. 
Let $\mathfrak{h}(X;D)$ be the space of all holomorphic vector fields tangential to the divisor and $\mathfrak{h}'(X;D)$ be the complexification of the Lie algebra consisting of Killing vector fields of $X$ tangential to $D$.
\begin{thm}[Reductivity]Suppose $\om$ is a constant scalar curvature K\"ahler cone metric. 
	Then there exists a one-to-one correspondence between $\mathfrak{h}'(X;D)$ and the kernel of ${\mathbb{L}\mathrm{ic}}_{\om}$. Moreover, $\mathfrak{h}(X;D)$ is reductive.
\end{thm}

\bigskip

In Section \ref{bifurcation}, we prove the bifurcation of the $J$-twisted path at a cscK cone metric ($t=1$),
	\begin{align}\label{csckpath}
\Phi(t,\vphi(t))=S(\om_{\vphi(t)})-\underline S_\b-(1-t)(\frac{\om^n}{\om_{\vphi(t)}^n}-1).
\end{align}
 The bifurcation argument of Bando-Mabuchi \cite{MR946233} for Aubin-Yau path at a K\"ahler-Einstein metric was extended to cone path in our previous paper \cite{arxiv:1511.02410}, and also to Chen's path \cite{arXiv:1506.06423} at the extremal metric in \cite{arxiv:1506.01290}. We will prove the bifurcation at a cscK cone metric, by utilising the linear theory in Section \ref{Linear theory for Lichnerowicz operator} and the reductivity theorem in Section \ref{Reductivity of automorphism group}.

\begin{thm}[Bifurcation]\label{Bifurcation}Suppose $\om$ is a constant scalar curvature K\"ahler cone metric with K\"ahler potential $\vphi$. 
	Then there exists a parameter $\tau>0$ such that $\vphi(t)$ is extended uniquely to a smooth one-parameter family of solutions of the $J$-twisted path \eqref{csckpath} on $(1-\tau, 1]$.
\end{thm}

\bigskip

To conclude this introductory section, let us mention the uniqueness of the constant scalar curvature K\"ahler cone metric.
This article addresses the following fundamental question: 
\textit{is the constant scalar curvature K\"ahler cone metric in $\mathcal H_\b$ unique up to automorphisms?} 


\begin{thm}[Uniqueness of cscK cone metrics]\label{Uniqueness}
	The constant scalar curvature K\"ahler cone metric along $D$ is unique up to automorphisms.
\end{thm}

\begin{rem}
A byproduct of the uniqueness theorem is the log-$K$-energy is bounded from below over the space of K\"ahler cone metric $\mathcal H_\beta$, assuming the existence of a constant scalar curvature K\"ahler cone metric in $\mathcal H_\beta$. That is interpreted as semistability in the language of geometric invariant theory.
\end{rem}

\begin{rem}
On Riemannian surface, the uniqueness of singular Gauss curvature metric is a classical outstanding problem in conformal geometry and geometric topology, see \cite{MR1333510,MR3556430,MR1773558} for references and related uniqueness results. 
About non-uniqueness result, there are examples of constant conical Gauss curvature metrics with angle larger than one, but they are not unique. One may ask whether the similar phenomena occurs for Calabi's extremal metrics (a generalisation notion of cscK metrics), when the cone angle $\beta>1$? However, in general, the uniqueness of such critical metrics fails either, e.g. Chen's examples for HCMU (the Hessian of the curvature of the metric is umbilical) metrics in \cite{MR1753319}. 
\end{rem}

\begin{rem}
In a recent article \cite{MR3323577}, the uniqueness of K\"ahler-Einstein cone metrics is proved. In \cite{arxiv:1511.02410}, we provide an alternative proof of this result.  The uniqueness is also true for more general singularities, i.e. the klt pairs in \cite{arXiv:1111.7158}. 
\end{rem}

As shown in a series of previous paper \cite{arxiv:1603.01743,arxiv:1703.06312}, we observe the loss of regularity phenomena along the singular direction for the cscK cone metrics, when angle is larger than half. This is a main difficulty, which causes the technical complexity. Approaches to address this problem will be presented in Section \ref{Uniqueness of cscK cone metrics}, as well as, Section \ref{Asymptotic behaviours}, \ref{Linear theory for Lichnerowicz operator}, \ref{Reductivity of automorphism group}, \ref{bifurcation}.

The structure of the proof of \thmref{Uniqueness} is organised as follows.
A key ingredient
of the proof of the uniqueness of the cscK, as well as the extremal K\"ahler metric, is the convexity of the $K$-energy along Chen's geodesic in the recent papers by Berman-Berndtsson \cite{MR3671939}, Chen-Li-P\u aun \cite{MR3582114}, Chen-P\u aun-Zeng \cite{arxiv:1506.01290}. It is also shown in Li \cite{arxiv:1511.00178} that the log-$K$-energy is also convex along the $C^{1,1,\b}$ geodesic, which is constructed in our previous paper \cite{MR3405866}.
In Section \ref{Convexity along the generalised cone geodesic}, we will extend their method to prove the convexity of the log-$K$-energy along the $C^{1,1,\b}_{\bf w}$ cone geodesic, that is used to prove the uniqueness of the $J$-twisted cscK cone metrics. The new difficulty is again the loss of regularity of the cone geodesic, which is different to the non-singular case.

Utilising Bando-Mabuchi's bifurcation method, the proof of the uniqueness of cscK cone metrics is now reduced to prove the existence of the deformation from a cscK cone metric to a $J$-twisted cscK cone metric. In order to realise such deformation, we perturb the cscK cone metric along the $J$-twisted path and assemble all the parts of the proof including the following ingredients, which are organised in this article as following
\begin{itemize}
\item the regularity of cscK cone metrics (Section \ref{Asymptotic behaviours}), 
\item the Fredholm alternative theorem for the Lichnerowicz operator (Section \ref{Linear theory for Lichnerowicz operator}), 
\item the reductivity of automorphism group (Section \ref{Reductivity of automorphism group}),
\item the bifurcation argument (Section \ref{bifurcation}).
\end{itemize}

\bigskip

The method in this paper has naturally potential application in other type singularities and we raise the following question.
\begin{ques} {What kind of singularities, the constant scalar curvature (or extremal) K\"ahler metric prescribed with, is generically unique? I.e. such singular constant scalar curvature (or extremal) K\"ahler metric is unique up to automorphisms.} 
\end{ques}

\bigskip

\noindent {\bf Acknowledgments:}
The work of K. Zheng has received funding from the European Union's Horizon 2020 research and innovation programme under the Marie Sk{\l}odowska-Curie grant agreement No. 703949, and was also partially supported by EPSRC grant number EP/K00865X/1.

\section{Cone geodesics}\label{Cone geodesic}
In this section, we solve the boundary value problem for geodesic in the space of K\"ahler cone metrics. 
Recall that $X$ is a smooth compact $n$-dimensional K\"ahler manifold without boundary and $\om_0$ is a smooth K\"ahler metric.
We are also given a smooth hyper-surface $D$ in $X$ and a cone angle $0 <\b\leq 1$.
Let {$s$} is a global section of the associated bundle $L_D$ of $D$ and $h$ is an Hermitian metric on $L_D$.

A \emph{K\"ahler cone metric} $\om$ of cone angle $\b$ along $D$,
is a smooth K\"ahler metric on the regular part $M:=X\setminus D$, and quasi-isometric to the cone flat metric, 
\begin{align}\label{flat cone}
\om_{cone}:=\b^2|z^1|^{2(\b-1)} i dz^1\wedge dz^{\bar 1}+\sum_{2\leq j\leq n} i dz^j\wedge dz^{\bar j},
\end{align}
under the cone chart $\{U_p;z^1,\cdots,z^n\}$ near $p\in D$, where $z^1$ is the local defining function of $D$.

We let $\mathcal H_\beta$ be the space of all K\"ahler potentials such that the associated K\"ahler metrics are K\"ahler cone metrics in $[\om_0]$, i.e. 
\begin{align*}
\mathcal H_\beta=\{\vphi\vert \om_{\vphi}=\om_0+i\p\bar\p\vphi \text{ is a K\"ahler cone metric in }[\om_0]\}.
\end{align*}
The space $\mathcal H_\beta$ contains the 
\emph{model metric}
\begin{align}\label{model cone}
\om_D=\om_0+\delta i \p\bar\p |s|^{2\b}_{h},
\end{align}
provided with a sufficiently small constant $\delta>0$. We would also use $\om$ without specification to denote $\om_D$ for short in this paper.


We denote the cylinder $R=[0,1]\times S^1$ with
the coordinate $$z^{n+1}=x^{n+1}+\sqrt{-1}y^{n+1}.$$
We also denote the $(n+1)$-dimensional product manifold with boundary and its regular part by $$\mathfrak{X} = X \times R;\quad \mathfrak M=M \times R.$$ We extend the segment $\{\vphi(z^1,\cdots,z^n,t);0\leq t\leq 1\}$ to the product manifold $\mathfrak{X}$ as $$\vphi(z',z^{n+1}):=\vphi(z^1,\cdots,z^n,x^{n+1})=\vphi(z^1,\cdots,z^n,t).$$ We let $\pi$ be the natural projection from $\mathfrak{X} $ to $X$ and also denote, after lifting-up, the new coordinates, K\"ahler metric $\om_0$ and K\"ahler potential $\vphi$ on $\mathfrak{X}$ by
\begin{align}\label{conventions}
\left\{
\begin{array}{ll}
z&=(z',z^{n+1})=(z^1,\cdots,z^n,z^{n+1}) \; ,\\
\Om_0&=(\pi^{-1})^\ast \om_0+dz^{n+1} \wedge d\bar z^{n+1} \; ,\\
\Om_{cone}&=(\pi^{-1})^\ast \om_{cone} \; ,\\
\Om&=(\pi^{-1})^\ast\om_D=\Om_0+i \p\bar\p \Psi_D\;,\\
\Psi(z)&=\vphi(z)-|z^{n+1}|^2 \; ,\\
\Om_{\Psi}&=\Om_0+i\p\bar\p\Psi.
   \end{array}
\right.
\end{align}

With the notations above, a \emph{cone geodesic} $\{\vphi(t);0\leq t\leq 1\}$
connecting $\vphi_0 $, $\vphi_1$ is a curve segment $\{\vphi(t);0\leq t\leq 1\}\subset \mathcal H_\b$ and satisfies a degenerate complex Monge-Amp\`{e}re  equation
\begin{align}\label{geo ma}
[\vphi''-(\p\vphi',\p\vphi')_{g_{\vphi}}]
\cdot \om^n_\vphi=\Om_{\Psi}^{n+1}=0\text{ in }\mathfrak M \; .
\end{align}
In which, $$\vphi'=\frac{\p\vphi}{\p t}\text{ and }(\p\vphi',\p\vphi')_{g_{\vphi}}=\sum_{1\leq i,j\leq n}g_\vphi^{i\bar j}\frac{\p\vphi'}{\p z^i}\frac{\p\vphi'}{\p z^{\bar j}}.$$
The cone geodesic segment requires that $\om_{\vphi(t)}$ is a possibly degenerate K\"ahler cone metric for any $0\leq t\leq1$.

The Dirichlet boundary conditions $\Psi_0$ are on two disjoint copies of $X$,
\begin{align}\label{geo ma boundary}
  \left\{
   \begin{array}{ll}
\Psi_0(z',0)&\doteq\Psi(z',\sqrt{-1}y^{n+1})\\
&={ \vphi_0(z')}-(y^{n+1})^2 \text{ on }
X\times \{0\}\times S^1,\\
\Psi_0(z',1)&\doteq\Psi(z',1+\sqrt{-1}y^{n+1})\\
&={ \vphi_1(z')}-1-(y^{n+1})^2 \text{ on }
X\times \{1\}\times S^1.
   \end{array}
  \right.
\end{align}

 We use the notation $\Om_{\Psi_0}:=\Om_0+i\p\bar\p\Psi_0$, which coincides with the boundary values $\om_{\vphi_0}$ and $\om_{\vphi_1}$ on $X\times \{0\}\times S^1$ and $X\times \{1\}\times S^1$, respectively.
We let the H\"older exponent satisfy $$\a\b < 1-\b.$$ The H\"older spaces $C^{0,\a,\b}$, $C^{2,\a,\b}$ on $\mathfrak X$ are defined in Section 2 in \cite{MR3405866}.









Assume that $\Om_{\mathbf b}$ is a K\"ahler cone metric in the product manifold $\mathfrak X$. We will show in the following sections, how a priori estimates depend on $\Om_{\mathbf b}$, and leave the construction of $\Om_{\mathbf b}$ in Section \ref{Geometric conditions on the background metric}.

We approximate the geodesic equation by a family of equation,
\begin{align}\label{per equ a}
\left\{
\begin{array}{ll}
\Om_{\Psi}^{n+1}
=\tau \cdot \Om_{\mathbf b}^{n+1} &\text{ in }\mathfrak{M}\; ,\\
\Psi=\Psi_0 & \text{ on }\p\mathfrak{X} \;.
\end{array}
\right.
\end{align}

We let $\Psi_{0,\eps}$ be the smooth approximation of the boundary value by Richberg's regularisation (see \cite{demaillybook,MR2299485}), since our boundary K\"ahler potentials are H\"older continuous. When the boundary values are cscK cone metrics, we will present alternative approximation of them in Section \ref{Approximation}. We also need the background metric $\Om_{\mathbf b}$ has smooth approximation $\Om_{\mathbf b_\eps}$, see Section \ref{Geometric conditions on the background metric}, that is
$\Om_{\mathbf b_\eps}^{n+1}=\frac{e^{-f_0}\cdot \Om_{0}^{n+1}}{(|s|^2+\eps^2)^{1-\b}}.$
Then we consider the smooth approximation equation
\begin{align}\label{smooth per equ a}
\left\{
\begin{array}{ll}
\Om_{\Psi_\eps}^{n+1}
=\tau \cdot  \Om_{\mathbf b_\eps}^{n+1}&\text{ in }\mathfrak{M}\; ,\\
\Psi_\eps=\Psi_{0,\eps} & \text{ on }\p\mathfrak{X} \;.
\end{array}
\right.
\end{align}
The smooth approximation equation \eqref{smooth per equ a} has a unique smooth solution $\Psi_\eps$. Moreover, according to Chen \cite{MR1863016}, the $C^{1,1}$-norm of $\Psi_\eps$ is uniformly bounded, independent of $\tau$, i.e.
	\begin{align*}
	||\Psi_\eps||_{C^{1,1}}=\sup_{\mathfrak X}\{|\Psi_\eps|+|\p\Psi_\eps|_{\Om_0}+|\p\bar{\p}\Psi_\eps|_{\Om_0}\}.
	\end{align*}
\subsubsection{Geometric conditions on the background metric $\om_{\mathbf b}$ and $\Om_{\mathbf b}$}\label{Geometric conditions on the background metric}

In \cite{MR3405866}, we reduce $C^{1,1,\b}$ estimate to geometric conditions of the background metric $\Om_{1}$, i.e.
\begin{itemize}
\item the cone angle less than $\frac{1}{2}$,
\item bounded Christoffel symbols,
\item bounded Ricci upper bound. 
\end{itemize}
We recall the construction of the background metric $\Om_{1}$. It is constructed by using the boundary values. Let $\tilde{\Psi}_0$ be the line segment between the boundary K\"ahler cone potentials $\tilde{\Psi}_0=t{  \vphi_1} +(1-t){  \vphi_0} ,$
	and  $\Phi$  be a convex function on $z^{n+1}$ and vanish on the boundary. Letting $\Psi_1 : =\tilde{\Psi}_0+m\Phi,$ we knew from Proposition 2.5 in \cite{MR3405866} that,
$\Om_{1}:=\Om_0+\sum_{1\leq i,j\leq n+1}i\p_i\p_{\bar j}\Psi_1$ is a K\"ahler cone metric on $(\mathfrak{X}, \mathfrak{D})$. The $\Psi_1$ inherits regularities from the boundary values. When $\vphi_0,\vphi_1\in C_{\mathbf w}^{3,\a,\b}$, we have $\Psi_1\in C_{\mathbf w}^{3,\a,\b}$, by using the computation from Corollary 4.2 in \cite{MR3405866}. We denote the weight 
\begin{align*}
\kappa=\b-\a\b.
\end{align*}
This weight comes from the computation of the Christoffel symbols of cscK cone metrics, see Section \ref{Christoffel symbols of cscK cone metric}.
We obtain that under the normal coordinate, the Christoffel symbols satisfy
\begin{align}\label{Om1backgroundmetricChristoffelsymbols}
\sup_{\mathfrak X} (\sum_{2\leq i\leq n+1} |\nabla_i^{cone}(\Om_1)_{a\bar b}|_{\Om}+|z^1|^{\kappa}\cdot |\nabla_1^{cone}(\Om_1)_{a\bar b}|_{\Om})\leq C,
\end{align}
according to Corollary \ref{psibcondition} and \ref{pwpsibcondition}.

We will use another background metrics $\Om_{\mathbf b}$ and $\om_{\mathbf b}$, whose geometry could be unbounded. The background metric $\om_{\mathbf b}$ stays on $X$ and $\Om_{\mathbf b}$ is the one on the product manifold $\mathfrak X$, i.e. $\Om_{\mathbf b}=(\pi^{-1})^\ast \om_{\mathbf b}$. We present the construction of the background metrics $\Om_{\mathbf b}$ and $\om_{\mathbf b}$, in order to fulfil the requirement in a priori estimates.


The following geometric conditions of $\om_{\mathbf b}$ (same for $\Om_{\mathbf b}$) will be used in obtaining the a priori estimates. 
\begin{enumerate}
\item Christoffel symbols:
\begin{align}\label{backgroundmetricChristoffelsymbols}
\sup_{\mathfrak X} (\sum_{2\leq i\leq n+1} |\nabla_i^{cone}(\Om_{\mathbf b})_{a\bar b}|_{\Om_{{\mathbf b}}}+|z^1|^{\kappa}\cdot |\nabla_1^{cone}(\Om_{\mathbf b})_{a\bar b}|_{\Om_{{\mathbf b}}})\leq C_1.
\end{align}

\item Riemannian curvature:	
for some fixed constant $C_2,C_3$,
\begin{align}\label{backgroundmetriccurvature}
\left\{
\begin{array}{ll}
&R_{i\bar j k\bar l}(\om_{\bf b} )\geq -(\tilde g_{\mathbf b})_{i\bar j} \cdot (g_{\bf b} )_{k\bar l}\\
&\tilde \om_{\mathbf b}=C_2\cdot \om_{\bf b} +i\p\bar\p \Phi\geq 0,\\
&|\Phi|_\infty\leq C_3.
\end{array}
\right.
\end{align}
\end{enumerate}

A simple way to construct $\om_{\bf b}$ and its approximation is to use $\om_{\bf b}=\om_D$ and $\om_{\bf b_\eps}=\om_0+\delta i \p\bar\p( |s|^{2}+\eps)^{\gamma}$ with small $\delta$ and $\tilde\om_{\bf b_\eps}=\om_0+\delta i \p\bar\p( |s|^{2}+\eps)^{\gamma}$ with $0<\gamma<\min\{\beta,1-\beta\}$. The $\Phi=( |s|^{2}+\eps)^{\gamma}$ is bounded. 

Another way of construction of $\om_{\mathbf b_\eps}$ uses the weight function $\chi_\gamma(\epsilon^2+t)=\gamma^{-1}  \int_0^t \frac{ (\epsilon^2+r)^\gamma- \epsilon^{2\gamma} }{r} dr$ in \cite{MR3488129}; and then
\begin{align*}
&\om_{\bf b_\eps}=\om_0+\delta i\p\bar\p\chi_\beta(|s|^2+\epsilon),\\
&\Phi=\chi_\gamma(|s|^2+\epsilon),\\
&\tilde\om_{\bf b_\eps}=\om_{\bf b_\eps}+i\p\bar\p \Phi.
\end{align*} 
In this construction, $\Phi$ is also bounded. Meanwhile, both $\Om_{\bf b}$ satisfy property \eqref{backgroundmetriccurvature}.



\subsubsection{Generalised cone geodesic}
After taking $\eps\rightarrow 0$, the solution of the smooth approximation equation \eqref{smooth per equ a} provides a solution to the approximation equation \eqref{per equ a}. Then taking $\tau\rightarrow 0$, we obtain a solution of the geodesic equation \eqref{geo ma}.
\begin{defn}\label{ctribw}
	We say $\{\vphi(t), 0\leq t\leq 1\}$ is a generalised cone geodesic, if it is the limit of solutions to the approximation equation \eqref{per equ a} as $\tau\rightarrow 0$ under the following norm,
	\begin{align*}
	||\vphi||_{C_{\tri}^{\b}}=\sup_{(z,t)\in \mathfrak X}\{|\vphi|+|\p_t\vphi|+|\p_z\vphi|_{\om}+|\p_z{\p_{\bar z}}\vphi|_{\om}\}.
	\end{align*}
\end{defn}

The estimate of $|\vphi|+|\p_t\vphi|+|\p_z{\p_{\bar z}}\vphi|_{\om}$ follows from the \lemref{Linfty estimate}, \lemref{boundary gradient estimate}
 and the interior spatial Laplacian estimate (Proposition \ref{bad prop: interior Laplacian estimate}). The interior spatial gradient estimate $|\p_z\vphi|_{\om}$ follows from Proposition \ref{Prop Interior spatial gradient estimate}.

\begin{thm}\label{geodesicclosedness}Assume $\{\vphi_{i}, i=0,1\}$ are two K\"ahler cone metrics.  Assume that the background metric $\Om_{\mathbf b}$ satisfies curvature condition  \eqref{backgroundmetriccurvature}. Then there exists a unique $C_{\tri}^{\b}$ generalised cone geodesic connecting $\vphi_i$ and there is a constant $C$ independent of $\tau$ such that $$||\Psi||_{C_{\tri}^{\b}}\leq C.$$ The constant $C$ depends on constants in \eqref{backgroundmetriccurvature} and 
	\begin{align*}
&\sup_{X} \Tr_{\om}\om_{\vphi_i},\quad\sup_X|\vphi_i|.
	\end{align*}
\end{thm}	


\subsubsection{Cone geodesic}

\begin{defn}\label{c11bw}
	We say $\{\vphi(t), 0\leq t\leq 1\}$ is a cone geodesic, if it is the limit of solutions to the approximation equation \eqref{per equ a} as $\tau\rightarrow 0$ under the following norm,
	\begin{align*}
	||\vphi||_{C_{\bf w}^{1,1,\b}}=||\vphi||_{C_{\tri}^{\b}}+\sup_{ \mathfrak X}\{\sum_{2\leq i\leq n}|\frac{\p^2\vphi}{\p z^{i}\p t} |_{\Om}+|s|^{\kappa}|\frac{\p^2\vphi}{\p z^{1}\p t} |_{\Om}+|s|^{2\kappa} |\frac{\p^2\vphi}{\p t^2} |\}.
	\end{align*}
\end{defn}

The second order estimate follows from Proposition \ref{roughnormalnormal}, which combines the rough boundary Hessian estimates (Proposition \ref{prop: boundary hessian estimate}), the rough interior Laplacian estimate (Proposition \ref{bad prop: interior Laplacian estimate}) and the interior spatial Laplacian estimate (Proposition \ref{prop: interior spatial Laplacian estimate}).

\begin{thm}\label{weakgeodesicclosedness}
Assume $\{\vphi_{i}, i=0,1\}$ are two $C_{\mathbf w}^{3,\a,\b}$ K\"ahler cone metrics. Assume that the background metric $\Om_{\mathbf b}$ satisfies assumption \eqref{backgroundmetricChristoffelsymbols} and \eqref{backgroundmetriccurvature}. Then there exists a unique $C_{\bf w}^{1,1,\b}$ cone geodesic satisfying \eqref{per equ a} connecting $\vphi_i$ and there is a constant $C$ independent of $\tau$ such that $$||\vphi||_{C_{\bf w}^{1,1,\b}}\leq C.$$ The constant $C$ depends on the constants in
\eqref{Om1backgroundmetricChristoffelsymbols}, \eqref{backgroundmetricChristoffelsymbols} and \eqref{backgroundmetriccurvature}, and 
\begin{align*}
&\sup_{\mathfrak{X}}|\Psi|, \quad
\sup_{\mathfrak X}|\Psi_{\mathbf b}|, \quad 
\sup_{\mathfrak X}|\p\Psi_{\mathbf b}|_{\Om},\quad 
\sup_{\mathfrak X} |\p\Psi_1|^2_\Om,\\
 & \sup_{\mathfrak X}|\Om_{\mathbf b}|_\Om, \quad
 \sup_{\p\mathfrak{X}} \Tr_{\om}\om_\vphi,\quad
 \sup_{\p \mathfrak X}[|s|^{\kappa}(|\frac{\p^2\Psi_1}{\p z^{1}\p z^{\overline{n+1}}} |_{\Om}+ |\frac{\p^2 \Psi_1}{\p z^{1}\p z^{{n+1}}}|_{\Om})]\\
&\sup_{\p \mathfrak X}[ |\frac{\p^2\Psi_1}{\p z^{ i}\p z^{n+1}} |_{\Om}+|\frac{\p^2\Psi_1}{\p z^{\bar i}\p z^{n+1}} |_{\Om}], \quad 2\leq i\leq n.
\end{align*} 

\end{thm}
The $C_{\mathbf w}^{3,\a,\b}$-norm of $\vphi_i$ control the terms involving $\Psi$ and $\Psi_1$.
The ideas and techniques developed in \cite{arxiv:1609.03111} could be applied to cone geodesics, we leave all these discussion for future development.

\begin{rem}
	The construction of bounded geodesics (weak solutions) has been elucidated in the well-known literature, see for example \cite{MR2884031,MR3323577,MR2932441,MR3617346}, but it is necessary to establish a self-contained, down-to-earth, asymptotic analysis theory to precisely understand smoothness/regularity of the cone geodesics. 
\end{rem}
\subsubsection{Estimates}
We collect basic estimates of \eqref{per equ a} with sufficient small $\tau$, then
\begin{align}
\left\{
\begin{array}{ll}
\Om_{\Psi}^{n+1}
=\tau \cdot  \frac{\Om_{\mathbf b}^{n+1}}{\Om_{1}^{n+1}} \cdot \Om_{1}^{n+1}\leq  \Om_{1}^{n+1} &\text{ in }\mathfrak{M}\; ,\\
\Psi=\Psi_1=\Psi_0 & \text{ on }\p\mathfrak{X} \;.
\end{array}
\right.
\end{align}
Let $h$ satisfy the linear equation 
\begin{align*}
\left\{
\begin{array}{ll}
\tri h =-n-1 & \text{ in }\mathfrak{M}\; ,\\
h=\Psi_0 & \text{ on }\p\mathfrak{X} \;.
\end{array}
\right.
\end{align*}

For all $0<\b\leq 1$, the maximum principle implies
the $L^\infty$ estimate.
\begin{lem}[$L^\infty$-estimate]\label{Linfty estimate} For any point $x\in \mathfrak{X}$, 
	\begin{align}
	\Psi_{1}(x)\leq \Psi(x)\leq h(x) .
	\end{align}
\end{lem}

The $L^\infty$-estimate yields the boundary gradient estimate.
\begin{lem}[Boundary gradient estimate]\label{boundary gradient estimate}
	\begin{align}
	\sup_{ \p \mathfrak X}|\nabla\Psi|_{\Om}\leq \sup_{\mathfrak{X}}
	|\nabla\Psi_{1}|_{\Om}+\sup_{\mathfrak{X}}|\nabla h|_{\Om}\; .
	\end{align}
\end{lem}

It also holds that $\frac{\p \vphi}{\p t}$ is uniformly bounded, see \cite{MR3323577}.
Boundary gradient estimate tells us at the boundary, $\frac{\p \Psi}{\p x}$ is uniformly (on $\tau$) bounded and so is $\frac{\p \vphi}{\p t}$.
By \eqref{per equ a}, $\frac{\p^2\vphi}{\p t^2}\geq 0$, thus $\frac{\p \vphi}{\p t}$ is bounded uniformly along the whole path $0\leq t\leq 1$. 

\subsection{Boundary Hessian estimates}\label{Boundary Hessian estimate, I}
In this section, we improve the boundary Hessian estimates, Section 3.3 in \cite{MR3405866}. 

As shown in \cite{MR3405866}, there is an obstruction to directly obtain the mixed singular tangent-normal estimates on the boundary \begin{align*}
| \frac{\p^2\Psi}{\p z^{1}\p z^{\overline{n+1}}}|_{\Om},\quad | \frac{\p^2\Psi}{\p z^{1}\p z^{n+1}}|_{\Om},
\end{align*} since the term $|\nabla_1^{cone}(\Om_{\mathbf b})_{a\bar b}|_{\Om_{cone}}$ is generally not bounded for large angles.
For example, we know from in \cite{MR3405866} that the model metric $\om_D$ has bounded $|\nabla_1^{cone}(\om_D)_{a\bar b}|_{\Om_{cone}}$, when angle is less than $\frac{2}{3}$. KE cone metric \cite{arxiv:1609.03111} or more general cscK cone metric (see \corref{bgdmetric}) has bounded $|\nabla_1^{cone}(\om_{cscK})_{a\bar b}|_{\Om_{cone}}$, when angle is less than $\frac{1}{2}$, moreover, it has growth rate $|z^1|^{-\kappa}$ with $\kappa=\b-\a\b$. In this section, we derive more information of this difficulty.

The Hessian estimates are divided into three parts.
The tangent-tangent estimates on the boundary follow from the boundary values directly,
	\begin{align*}
 \frac{\p^2 (\Psi -\Psi_1)}{\p z^{i}\p z^{\bar j}}= \frac{\p^2 (\Psi -\Psi_1)}{\p z^{\bar i}\p z^{ j}}=0,\quad \forall 1\leq i,j \leq n.
\end{align*}
\begin{enumerate}
	\item the mixed regular tangent-normal estimates on the boundary (Proposition \ref{tangent normal estimates}),
	\begin{align*}
	\sup_{\p \mathfrak{X}} |\frac{\p^2 \Psi  }{\p z^{i}\p z^{\overline {n+1}}} |_{\Om},\quad 
	\sup_{\p \mathfrak{X}} |\frac{\p^2\Psi }{\p z^{i}\p z^{ n+1}}\Psi|_{\Om},\quad \forall 2\leq i \leq n;
	\end{align*}

	\item the rough mixed singular tangent-normal estimates on the boundary (Proposition \ref{bad singular tangent-normal estimates on the boundary}),
	\begin{align*}
	\sup_{\p\mathfrak X} |s|^{\kappa}|\frac{\p^2 \Psi }{\p z^{1}\p z^{\overline{ n+1}}}\Psi|_{\Om},\quad
	\sup_{\p\mathfrak X}|s|^{\kappa} |\frac{\p^2 \Psi}{\p z^{1}\p z^{ n+1}}\Psi|_{\Om}.
	\end{align*}
	
	\item the rough normal-normal estimates on the boundary (Proposition \ref{bad Boundary normal-normal estimate}), 
	\begin{align*}
	\sup_{\p \mathfrak{X}} |s|^{2\kappa}|\frac{\p^2 \Psi }{\p z^{n+1}\p z^{\overline{ n+1}}}\Psi|_{\Om},\quad\sup_{\p \mathfrak{X}} |s|^{2\kappa}|\frac{\p^2 \Psi }{\p z^{n+1}\p z^{ n+1}}\Psi|_{\Om};
	\end{align*}
\end{enumerate}

\begin{prop}[Rough boundary Hessian estimates]\label{prop: boundary hessian estimate} Assume that the Christoffel symbols of the background metrics $\Om_{1}$ and $\Om_{\mathbf b}$ satisfy conditions
\eqref{Om1backgroundmetricChristoffelsymbols} and \eqref{backgroundmetricChristoffelsymbols}, respectively.
	Then the terms in $(1)(2)(3)$ are bounded by 
	\begin{align*}
	C(\sup_{\mathfrak X} |\p\Psi|^2_\Om+1).
	\end{align*}
	The constant $C$ depends on the constants in 
	\eqref{Om1backgroundmetricChristoffelsymbols} and \eqref{backgroundmetricChristoffelsymbols}, and 
\begin{align*}
	&\sup |\Om_{\mathbf b}|_{\Om_{ cone}},\quad \sup_{\mathfrak X} |\p\Psi_1|^2_\Om,\quad
	\sup_{\p \mathfrak X}[|s|^{\kappa}(|\frac{\p^2\Psi_1}{\p z^{1}\p z^{\overline{n+1}}} |_{\Om}+ |\frac{\p^2 \Psi_1}{\p z^{1}\p z^{{n+1}}}|_{\Om})]\\
	&\sup_{\p \mathfrak X}[ |\frac{\p^2\Psi_1}{\p z^{ i}\p z^{n+1}} |_{\Om}+|\frac{\p^2\Psi_1}{\p z^{\bar i}\p z^{n+1}} |_{\Om}], \quad 2\leq i\leq n.
	\end{align*} 
\end{prop}
\begin{rem}
	In general, we could apply the argument with different weight to replace $|z^1|^{\kappa}$, which depends on the boundary values.
\end{rem}


\subsubsection{Boundary mixed regular tangent-normal estimates}\label{Boundary Hessian estimates}

 We are given a point $p$ on the boundary of $X$, i.e.  $ \p\mathfrak X$. 
We assume the point $p$ is on the divisor, otherwise, the estimates are simpler.
We use the cone chart $\{U;z^1,\cdots z^{n+1}\}$ centred at $p$ as before, i.e. $p=0$, $z^1$ is the normal direction to $D$ and $z^{n+1} := x+ \sqrt{-1}y$ parametrises the cylinder $R$. We fix a half ball of $p$ i.e. 
$B^+_{\rho_0} $ is the interior of the set $ \mathfrak X\cap B_{\rho_0}(p){\subset U}$ such that
the {radius} $\rho_0<1$ and $B_{5\rho_0}\subset U$. 

We consider in $B^+_{\rho_0}$,
\begin{align}\label{eq_defn_function_v}
 v:= (\Psi -\Psi_1) + sx - N x^2\, ,
\end{align}
where $N$ is determined in \eqref{N} and $s=2 N$.

\begin{lem}\label{eq_subclaim_inequalities_on_v}Let $\tri_{\Psi} $ be the Laplacian with respect to $g_\Psi$. Let $\eps_0$ be the constant such that $\Om_1\geq \eps_0\cdot\Om_{cone}$ in $B^+_{\rho_0}$. Then there exists a constant $N$ depending on $\eps_0$ and $\inf\frac{\Om^{n+1}_{cone}}{\Om^{n+1}_{\bf b}}$ such that the following inequalities hold
\begin{align*}
\left\{
\begin{array}{ll}
\tri_{\Psi}v \leq - \frac{\eps_0}{2}(1 + {\Tr}_{\Om_\Psi}\Om_{cone}) &  \mbox{ in } B^+_{\rho_0},
 \\
  v\geq 0 \qquad & \mbox{ on } \partial B^+_{\rho_0}, \\
  v=0 &\text{ on }\p B^+_{\rho_0}\cap\p \mathfrak X.
\end{array}
\right.
\end{align*}

\end{lem}
\begin{proof}
Recall that $\p_{z^i}=\frac{1}{2}(\p_{x^i}-i\p_{y^i})$. Then direct computation shows that
\begin{align*}
\tri_{\Psi} v&=n+1- \Tr_{\Om_\Psi}\Om_1-N\cdot g_{\Psi}^{n+1,\overline{n+1}}\\
&\leq n+1- \frac{\eps_0}{2}\Tr_{\Om_\Psi}\Om_{cone}- [\frac{\eps_0}{2}\Tr_{\Om_\Psi}\Om_{cone}+N\cdot g_{\Psi}^{n+1,\overline{n+1}}].
\end{align*}
We then have, by the inequality of arithmetic and geometric means,
\begin{align*}
\frac{\eps_0}{2}\Tr_{\Om_\Psi}\Om_{cone}+N\cdot g_{\Psi}^{n+1,\overline{n+1}}
\geq (n+1)[(N+\frac{\eps_0}{2})\frac{1}{2^n}\frac{\Om_{cone}^{n+1}}{\Om_\Psi^{n+1}}]^{\frac{1}{n+1}}.
\end{align*}
Since $\Om^{n+1}_\Psi=\tau\cdot\Om^{n+1}_{\bf b}$ and $\inf\frac{\Om^{n+1}_{cone}}{\Om^{n+1}_{\bf b}}>0$, we choose large $N$ such that 
\begin{align}\label{N}
(n+1)-(n+1)[(N+\frac{\eps_0}{2})\frac{1}{2^n}\inf\frac{\Om^{n+1}_{cone}}{\Om^{n+1}_{\bf b}}]^{\frac{1}{n+1}}\leq - \frac{\eps_0}{2}.
\end{align} So the first inequality is proved. Moreover, the boundary inequality follows from $\Psi_1\leq \Psi$ and the choice of $s$.

\end{proof}

\begin{lem}\label{triu}
	There exits a function $u$ such that
	\begin{itemize}
		\item $u$ vanishes on $ \p B^+_{\rho_0}\cap\p \mathfrak X$;
		\item for $q\in \Gamma^+_{\rho_0}:=\partial B^+_{\rho_0}\cap \Int(\mathfrak X)$, $u(q)\geq \eps_1(\rho_0)$ for some constant $\eps_1(\rho_0)$ depending only on the fixed $\rho_0$;
		\item  For some fixed constant $\eps_2\geq 1$, it holds in $B^+_{\rho_0}$, \begin{align*}
		\tri_\Psi u\leq \eps_2\cdot \Tr_{\Om_\Psi}\Om_{cone}.
		\end{align*}
	\end{itemize}
\end{lem}
\begin{proof}
	We could choose nonnegative $u$ to be the auxiliary global bounded function in $M\times R$ constructed on page 1173 in \cite{MR3405866} or the local function $|z^1|^{2\b}+\sum_{2\leq j\leq n+1}|z^j|^2$. In the latter case, \begin{align}\eps_1(\rho_0)=\rho_0\text{ and }\eps_2=1.\end{align}
\end{proof}

We define the real operator 
\begin{align*}
D_i := \p_{z^i}+\p_{z^{\bar i}} \text{ or }
\sqrt{-1}(\p_{z^i}-\p_{z^{\bar i}}),\quad \forall 2\leq i \leq n+1.
\end{align*}
And we consider $D_i(\Psi-\Psi_1)$.

\begin{lem}\label{realtildepsipsi1}Suppose $2\leq i \leq n+1$. Then there exists a constant \begin{align}\eps_3=\sup \Tr_{\Om_{{\mathbf b}}}\Om_{ cone},
\end{align} such that on $\mathfrak M$,
	\begin{align*}
	\tri_\Psi[D_i(\Psi-\Psi_1)]
	\leq &\eps_3\cdot F_i
	\cdot(1 + {\Tr}_{\Om_\Psi}\Om_{cone}).
	\end{align*}
	In which, $F_i=|\nabla_i^{cone}(\Om_{\mathbf b})_{a\bar b}|_{\Om_{cone}}+|\nabla_i^{cone}(\Om_{1})_{a\bar b}|_{\Om_{cone}}$.
\end{lem}
\begin{proof}
We denote $g$ the local K\"ahler potential of $\om$ and note that $\p_{z^a}\p_{z^{\bar b}}g=g_{a\bar b}$, $1\leq a,b\leq n+1$.
From
$\tri_\Psi=g_\Psi^{a\bar  b}\p_{z^a}\p_{z^{\bar b}}$, 
we have
\begin{align*}
\tri_\Psi \p_{z^i}(\Psi-\Psi_1)=\p_{z^i}\log\Om_{\Psi}^{n+1}-g^{a\bar b}_{\Psi} \p_{z^i}(g_{\Psi_{1}})_{a\bar b}.
\end{align*}
Using the approximation equation \eqref{per equ a}, we have $\p_{z^i}\log\Om_{\Psi}^{n+1}=\p_{z^i}\log\Om_{{\mathbf b}}^{n+1}$, and then
\begin{align}\label{psipsi1}
\tri_\Psi \p_{z^i}(\Psi-\Psi_1)=\p_{z^i}\log\Om_{{\mathbf b}}^{n+1}- g^{a\bar b}_{\Psi} \p_{z^i} (g_{\Psi_1})_{a\bar b}.
\end{align}
Further computing \eqref{psipsi1}, we get
	\begin{align*}
	&\tri_\Psi[\p_{z^i}(\Psi-\Psi_1)]\\
	&=\sum_{1\leq a,b\leq n+1}[g^{a\bar b}_{\Psi_{\mathbf b}}\nabla_i^{cone}(\Om_{\mathbf b})_{a\bar b}-g^{a\bar b}_{\Psi}\nabla_i^{cone}(\Om_{1})_{a\bar b}]\\
	&\leq \eps_3\cdot(|\nabla_i^{cone}(\Om_{\mathbf b})_{a\bar b}|_{\Om_{cone}}+|\nabla_i^{cone}(\Om_{1})_{a\bar b}|_{\Om_{cone}})\cdot(1 + {\Tr}_{\Om_\Psi}\Om_{cone}) .
	\end{align*}
	The lemma follows from the definition of $D_i$.
\end{proof}


The mixed regular tangent-normal estimates on the boundary are already in the proof of Proposition 3.7 in \cite{MR3405866}. We recall the proof for later use.

\begin{prop}[Boundary mixed regular tangent-normal estimates]\label{tangent normal estimates}For $2\leq i \leq n$, there exits a constant $C$ such that
\begin{align}\label{end_section_hessian_bdry_estimate}
&\sup_{\p \mathfrak X}[ |\frac{\p^2(\Psi-\Psi_1)}{\p z^{ i}\p z^{n+1}} |_{\Om}+|\frac{\p^2(\Psi-\Psi_1)}{\p z^{\bar i}\p z^{n+1}} |_{\Om}]\\
&\leq C [1+\sup_{\p \mathfrak X} |\p_{n+1}(\Psi - \Psi_1)|_{\Om} \cdot \sup_{\mathfrak X} |\p_{i}(\Psi - \Psi_1)|_{\Om}]\;.\nonumber
\end{align}
The constant $C$ depends on 
	\begin{align*}
&\sup |\Om_{\mathbf b}|_{\Om_{ cone}},\quad
\sup|\nabla_i^{cone}(\Om_{\mathbf b})_{a\bar b}|_{\Om_{cone}},\\ &\sup|\nabla_i^{cone}(\Om_{1})_{a\bar b}|_{\Om_{cone}},\quad 1\leq a,b\leq n+1.
\end{align*} 
\end{prop}
\begin{proof}

We consider 
\begin{align*}
h:= \lambda_1  v + \lambda_2 u+ D_i ( \Psi - \Psi_1),
\end{align*}
with two constants
\begin{align}
\label{l2}\lambda_2&=\frac{1}{\eps_1}\cdot [1+\sup_{\mathfrak X} |\p_i (\Psi - \Psi_1)|_{\Om}],\\
\label{l1}\lambda_1&=\frac{4}{\eps_0}\cdot[\lambda_2 \cdot\eps_2+\eps_3\cdot\sup F_i].
\end{align} 

We check the boundary value of $h$. On $ \p{B^+_{\rho_0}}\cap\p \mathfrak X$, $h=0$. Using \lemref{eq_subclaim_inequalities_on_v} and \eqref{l2}, we have the estimates of $h$ on the upper boundary, that is, for $q\in \partial B^+_{\rho_0} \cap {\Int}(\mathfrak X)$, 
\begin{align*}
h(q)&\geq \lambda_2  u(q)- | D_i(\Psi -  \Psi_1)(q)|\\
&\geq \lambda_2 \eps_1-\sup_{\mathfrak X} |\p_i(\Psi - \Psi_1)|_{\Om}\\
&\geq 0  .
\end{align*}
Applying \lemref{eq_subclaim_inequalities_on_v}, \lemref{triu}, \lemref{realtildepsipsi1} and \eqref{l1}, we have the differential inequality which $h$ obeys,
\begin{align*}
\tri_\Psi h
\leq& [- \frac{\eps_0}{2}\lambda_1+\lambda_2  \eps_2+\eps_3(|\nabla_i^{cone}(\Om_{\mathbf b})_{a\bar b}|_{\Om_{cone}}+|\nabla_i^{cone}(\Om_{1})_{a\bar b}|_{\Om_{cone}})]\\
&\cdot(1 + {\Tr}_{\Om_\Psi}\Om_{cone})
< 0.
\end{align*} 
Then using the maximum principle and $h(p)=0$,
we have {(recall $z^{n+1}= x+\sqrt{-1}\, y$)}
\[
 \frac{\partial h}{\partial x}(p) \geq 0  .
\]
Thus we have
\begin{align*}
 \frac{\p }{\p x} D_i (\Psi - \Psi_1)(p) \geq 
 -\lambda_1 [\frac{\partial(\Psi - \Psi_1)}{\partial x}+s
 -2Nx ]-\lambda_2 \frac{\p u}{\partial x}(p):=f \;.
\end{align*}
The same argument holds for $ -D_i $,
\begin{align*}
-\frac{\p }{\p x} D_i (\Psi - \Psi_1)(p) \geq f \;.
\end{align*} 
But along the tangent direction $$\frac{\p}{\p y} D_i (\Psi- \Psi_1) =0.$$
Recall that $D_i=\frac{\p}{\p x^i},\frac{\p}{\p y^i}$.
Combining them together, we get that for any $1\leq i \leq n$,
\begin{align*}
|\frac{\p^2(\Psi-\Psi_1)}{\p z^{n+1}\p z^{i}}  |(p)
\leq
|f|\;,\quad |\frac{\p^2(\Psi-\Psi_1)}{\p z^{n+1}\p z^{\bar i}}  |(p)
\leq
|f|\;.
\end{align*} 
From the choice of $u$, we know that $\sup_{\p \mathfrak X} |\p_{n+1} u|_{\Om}\leq 1$.
Therefore, 
\begin{align*}
|f|&\leq \lambda_1[\sup_{\p \mathfrak X} |\p_{n+1}(\Psi - \Psi_1)|_{\Om}+4N]+\lambda_2.
\end{align*}
The conclusion is proved.
\end{proof}

\subsubsection{Rough boundary mixed singular tangent-normal estimates}
\label{Rough boundary mixed singular tangent-normal estimate}

When we directly estimate the mixed singular tangent-normal derivative on the boundary, that is $i=1$, the term $|\nabla_1^{cone}(\Om_1)_{a\bar b}|_{\Om_{cone}}$ is not finite in general when angle is large than $\frac{2}{3}$, as shown in Section 3.3 in \cite{MR3405866}. But it has the growth rate $O(|z^1|^{\a\b-\b})$, when the boundary values are cscK cone metrics, according to \lemref{psibcondition}. We present a direct adaption of previous proof in this section, which gives us a rough estimate of the boundary mixed singular tangent-normal estimates.

\begin{prop}[Rough boundary mixed singular tangent-normal estimates]\label{bad singular tangent-normal estimates on the boundary}These exits a constant $C$ such that
	\begin{align*}
	 &\sup_{\p \mathfrak X}[|s|^{\kappa}(|\frac{\p^2(\Psi- \Psi_1)}{\p z^{1}\p z^{\overline{n+1}}} |_{\Om}+ |\frac{\p^2(\Psi-  \Psi_1)}{\p z^{1}\p z^{{n+1}}}|_{\Om})]\nonumber\\
	 &\leq C [1+\sup_{\p \mathfrak X} |\p_{n+1}(\Psi - \Psi_1)|_{\Om} \cdot \sup_{\mathfrak X} |\p_{1}(\Psi - \Psi_1)|_{\Om}]\;.
	\end{align*}
	The constant $C$ depends on
	\begin{align*}
	&\sup |\Om_{\mathbf b}|_{\Om_{ cone}},\quad
\sup|z^1|^{\kappa}|\nabla_1^{cone}(\Om_{\mathbf b})_{a\bar b}|_{\Om_{cone}},\\ &\sup|z^1|^{\kappa}|\nabla_1^{cone}(\Om_{1})_{a\bar b}|_{\Om_{cone}}\quad 1\leq a,b\leq n+1.
	\end{align*} 
\end{prop}
\begin{proof}
We choose $\kappa_s$ to be a sequence of rational numbers decreasing to $\kappa=\b-\a\b$. We then apply the orbit coordinate and use $\tilde{\cdot}$ to denote the pull-back in one branched covering.
From \lemref{realtildepsipsi1}, we have
	\begin{align}\label{1realtildepsipsi1}
	\tilde \tri_\Psi[\tilde D_1(\Psi-\Psi_1)]
	\leq & |z^1|^{\kappa_s}\cdot  \tri_\Psi[ D_1(\Psi-\Psi_1)]\\
	\leq &\eps_3\cdot|z^1|^{\kappa_s}\cdot F_1 \cdot(1 + {\Tr}_{\Om_\Psi}\Om_{cone})\nonumber.
	\end{align}
We consider
\begin{align*}
\tilde h:= \lambda^1_1 \tilde v + \lambda^1_2 \tilde u+\tilde D_1 (\tilde \Psi -  \tilde\Psi_1)\;.
\end{align*}
with two constants determined by
\begin{align}
\label{1l2}\lambda^1_2&=\eps_1^{-1}\cdot [1+|\p (\Psi - \Psi_1)|_{\Om_{cone}}],\\\label{1l1}\lambda^1_1&=\frac{4}{\eps_0}\cdot(\lambda^1_2 \cdot\eps_2+\eps_3\cdot|z^1|^{\kappa_s}\cdot\sup F_1).
\end{align}
	Since $\kappa_s\geq \b-\a\b\geq0$, $|\tilde D_1(\tilde \Psi -  \tilde\Psi_1)|\leq |z^1|^{\kappa_s}\cdot |\p_1(\Psi - \Psi_1)|$ is bounded. The boundary value of $h$ on $\partial \tilde B_{\delta_0} \cap \tilde {\Int}(M\times R)$ satisfies, by \lemref{eq_subclaim_inequalities_on_v} and \eqref{l2},
	\begin{align*}
	\tilde h(q)&\geq \lambda^1_2  \tilde u(q)- |\tilde D_1( \tilde\Psi - \tilde \Psi_1)(q)|\\
	&\geq \lambda^1_2 \eps_1-|z^1|^{\kappa_s}|\p_1(\Psi - \Psi_1)(q)|\\
	&\geq 0 \, .
	\end{align*}
While, on $ \partial \tilde B_{\delta_0} \cap \tilde {\p}(M\times R)$, $\tilde h=0$. By \lemref{eq_subclaim_inequalities_on_v}, \lemref{triu}, \eqref{1realtildepsipsi1} and \eqref{1l1}, $\tilde h$ satisfies the differential inequality
	\begin{align*}
	\tilde\tri_\Psi \tilde h\leq &[- \frac{\eps_0}{2}\lambda^1_1+\lambda^1_2\cdot  \eps_2
	+\eps_3\cdot|z^1|^{\kappa_s}\cdot(|\nabla_1^{cone}(\Om_{\mathbf b})_{a\bar b}|_{\Om_{cone}}+\nabla_1^{cone}(\Om_{1})_{a\bar b}|_{\Om_{cone}})]\\
	&\cdot(1 + {\Tr}_{\Om_\Psi}\Om_{cone})< 0.
	\end{align*}
Since $\tilde h(p)=0$, we have by the strong maximum principle
\[
\frac{\partial \tilde h}{\partial x}(p) \geq 0  ,
\]
That is
\begin{align*}
\frac{\p }{\p x}\tilde D_1 (\tilde\Psi - \tilde\Psi_1)(p) \geq -\lambda^1_1 (\frac{\partial(\tilde\Psi - \tilde\Psi_1)}{\partial x}+s
-2Nx) -\lambda^1_2 \frac{\p\tilde u}{\partial x}(p):=f \;.
\end{align*}
The same inequality holds for $-\tilde D_1$. After transforming back to the $z$ coordinate, we have the mixed singular tangent-normal
derivative is bounded by $|f|_{\infty}$, that is
\begin{align*}
|z^1|^{\kappa_s}\frac{\p^2(\Psi - \Psi_1)}{\p x\p x^1}\leq C\;.
\end{align*} 
where $C$ depends on $N$, $\lambda^1_1\cdot|\p_{n+1}(\Psi - \Psi_1)|_{\Om}$, $\lambda^1_2$  and $|\p u|_{\Om}$.
Using the same argument for $\pm \p_{y^1}(\Psi-\Psi_1)$, combining with the tangent direction $\frac{\p}{\p y} D_i (\Psi- \Psi_1) =0$ and letting $s\rightarrow \infty$, we completes the proof of the proposition.
\end{proof}

Applying the approximation geodesic equation \eqref{per equ a}, we obtain the normal-normal estimates on the boundary immediately.
\begin{prop}[Rough boundary normal-normal estimates]\label{bad Boundary normal-normal estimate}
These exits a constant $C$ such that
\begin{align}\label{badend_normal_normal_bdry_estimate}
&\sup_{\p \mathfrak X}[|s|^{2\kappa} (|\frac{\p^2\Psi}{\p z^{n+1}\p z^{\overline{n+1}}} |_{\Om}+\sup_{\p \mathfrak X} |\frac{\p^2\Psi}{\p z^{n+1}\p z^{{n+1}}}|_{\Om})]\nonumber\\
&\leq  C(\sup_{\mathfrak X} |\p\Psi|^2_\Om+1)\;.
\end{align}
The constant $C$ depends on the constants in Proposition \ref{tangent normal estimates},  Proposition \ref{bad singular tangent-normal estimates on the boundary} and
\begin{align*}
&\sup_{\mathfrak X} |\p\Psi_1|^2_\Om,\quad
\sup_{\p \mathfrak X}[|s|^{\kappa}(|\frac{\p^2\Psi_1}{\p z^{1}\p z^{\overline{n+1}}} |_{\Om}+ |\frac{\p^2 \Psi_1}{\p z^{1}\p z^{{n+1}}}|_{\Om})]\\
&\sup_{\p \mathfrak X}[ |\frac{\p^2\Psi_1}{\p z^{n+1}\p z^{ i}} |_{\Om}+|\frac{\p^2\Psi_1}{\p z^{n+1}\p z^{\bar i}} |_{\Om}], \quad 2\leq i\leq n.
\end{align*} 
\end{prop}
\subsection{Interior Laplacian estimate}\label{Interior Laplacian estimate}
This section improves Proposition 3.3 in \cite{MR3405866}, where we get three estimates $C_1,C_2,C_3$. It is sufficient to consider the case when $\frac{1}{2}\leq \b\leq 1$. When angle is large, $\inf Riem(\Om)$ in $C_1$ and $|Riem(\Om_1)|_\infty$ in $C_2$ could be unbounded. 
The constant $C_3$ relies on the boundary estimates $\sup_\mathfrak{\p X} \Tr_{\Om_{\Psi}}\Om$, which may not be bounded, either. However, according to Proposition \ref{prop: boundary hessian estimate} in Section \ref{Boundary Hessian estimate, I}, we have a rough bound on the boundary $\sup_{\p\mathfrak{X}} (|s|^{2\kappa}\cdot \Tr_{\Om}\Om_\Psi(z))$. As a result, we could extend boundary estimates to the interior under a curvature condition of the background metric $\Om_{\mathbf b}$.

\begin{prop}[Rough interior Laplacian estimate]\label{bad prop: interior Laplacian estimate}
Assume that the background metric $\Om_{\mathbf b}$ satisfies curvature condition \eqref{backgroundmetriccurvature}. 
	Then there exists a constant $C$ such that 
	\begin{align}\label{bab claim of subsection interior laplacian estimate}
	\sup_{\mathfrak{X}} (|s|^{2\kappa}\cdot \Tr_{\Om}\Om_\Psi)\leq C .
	\end{align}
	The constant $C$ depends on the constants in \eqref{backgroundmetriccurvature} and 
	\begin{align*}
\sup_{\p\mathfrak{X}} (|s|^{2\kappa}\Tr_{\Om}\Om_\Psi), \quad
 \sup_{\mathfrak X}|\Om_{\mathbf b}|_\Om, \quad \sup_{\mathfrak X}|\Psi_{\mathbf b}|, \quad
 \sup_{\mathfrak X}|\Psi_{\mathbf b}|.
	\end{align*}
\end{prop}
\begin{proof}

We applying the inequality from Yau's second order estimate
\begin{align*}
\tri_{\Psi} \log\Tr_{\Om_{\mathbf b}}\Om_{\Psi}\geq \frac{- g_{\mathbf b}^{i\bar j}R_{i\bar j}(\Om_{\Psi})+g_\Psi^{k\bar l}{R^{i\bar j}}_{k\bar l}(\Om_{{\mathbf b}}){g_\Psi}_{i\bar j}}{\Tr_{\Om_{{\mathbf b}}}\Om_{\Psi}}.
\end{align*}
From approximation geodesic equation \eqref{per equ a}, we have
\begin{align*}
Ric(\Om_{\Psi})=Ric(\Om_{{\mathbf b}})
\end{align*}
and
\begin{align}\label{RHS2nd}
RHS=\frac{-S(\Om_{{\mathbf b}})+g_\Psi^{k\bar l}{R^{i\bar j}}_{k\bar l}(\Om_{{\mathbf b}}){g_\Psi}_{i\bar j}}{\Tr_{\Om_{{\mathbf b}}}\Om_{\Psi}}.
\end{align}
We use an inequality in Proposition 2.1 in \cite{MR3488129}.
Under the local normal coordinate of $ \Om_{{\mathbf b}}$ (understood in the approximation sense), we denote by $\lambda_a, 1\leq a\leq n+1$ the eigenvalue of $\Om_{\Psi}$. So \begin{align*}
RHS\geq \frac{\Sigma_{a\leq b}(\frac{\lambda_a}{\lambda_b}+\frac{\lambda_b}{\lambda_a}-2)R_{a\bar ab\bar b}(\Om_{{\mathbf b}})}{\Sigma \lambda_c}.
\end{align*}
We use the curvature condition \eqref{backgroundmetriccurvature}, $R_{a\bar a b\bar b}(\Om_{\bf b} )\geq -(\tilde g_{\mathbf b})_{a\bar a} \cdot (g_{\bf b} )_{b\bar b}$. Note that $\frac{\lambda_a}{\lambda_b}+\frac{\lambda_b}{\lambda_a}-2\geq 0.$ Then
\begin{align}\label{RHS2ndinequality}
RHS&\geq -C\frac{\Sigma_{a\leq b}(\frac{\lambda_a}{\lambda_b}(\tilde\Om_{{\mathbf b}})_{b\bar b}+\frac{\lambda_b}{\lambda_a}(\tilde\Om_{{\mathbf b}})_{a\bar a})}{\Sigma \lambda_c}\nonumber\\
&\geq -C(\Tr_{\Om_{\Psi}}\Om_{{\mathbf b}}+\tri_{\Psi}\Phi).
\end{align}

Recall that $L$ is the line bundle associated to the divisor $D$ and $s$ is the holomorphic section.
We choose $h$ to be a smooth Hermitian metric on $L$ such that for some constant $C$, 
\begin{align*}
i\p\bar\p \log |s|_h\geq -C\cdot  \Om_{\mathbf b}.
\end{align*}
Combining with \eqref{RHS2ndinequality}, we have
\begin{align*}
\tri_{\Psi} \log(|s|_h^{2\kappa}\Tr_{\Om_{\mathbf b}}\Om_{\Psi})
\geq  -C(\Tr_{\Om_{\Psi}}\Om_{{\mathbf b}}+\tri_{\Psi}\Phi).
\end{align*}

Since  $\tri_{\Psi}(\Psi-\Psi_{\mathbf b})=n+1-\Tr_{\Om_{\Psi}}\Om_{\mathbf b}$, we obtain the differential inequality of $$Z=\log(|s|_h^{2\kappa}\Tr_{\Om_{\mathbf b}}\Om_{\Psi})+C\Phi-(C+1) (\Psi-\Psi_{\mathbf b}),$$ which reads
\begin{align*}
\tri_{\Psi} Z
\geq \Tr_{\Om_{\Psi}}\Om_{{\mathbf b}}-(n+1)(C+1).
\end{align*}
If the maximum achieves on the boundary, the proof is finished. When the maximal point $p$ is an interior point, 
$
\Tr_{\Om_{\Psi}}\Om_{{\mathbf b}}(p)\leq(n+1)(C+1)
$ and the inequality of arithmetic and geometric means implies at $p$, $$\Tr_{\Om_{{\mathbf b}}  }\Om_\Psi
	\leq \frac{n+1}{n}(\Tr_{\Om_\Psi}\Om_{{\mathbf b}})^n \cdot \frac{\Om_\Psi^{n+1}}{\Om_{{\mathbf b}}^{n+1}}=\tau\frac{n+1}{n}(\Tr_{\Om_\Psi}\Om_{{\mathbf b}})^n .$$ Letting $L=C\Phi-(C+1) (\Psi-\Psi_{\mathbf b})$, we have for any $z\in \mathfrak X$, $$\log(|s|_h^{2\kappa}\Tr_{\Om_{\mathbf b}}\Om_{\Psi})(z)\leq Z(p)-L(x).$$
Thus the proof is complete, since $L$ is bounded.
\end{proof}
\subsection{Interior spatial Laplacian estimate}\label{Interior spatial Laplacian estimate}


The interior spatial Laplacian estimate \cite{MR3298665} was considered in singular case in \cite{arxiv:1306.1867} to obtain a rough interior spatial Laplacian estimate. This method, as a generalisation of Yau's second order estimate, relies on the lower bound of the bisectional curvature and upper bound of the scalar curvature, which is bounded for model metric when cone angle is less than $\frac{1}{2}$. However, we will see that the estimate is still true when the curvature is not bounded.
We use the approximate geodesic equation \eqref{per equ a},
\begin{align*}
\Om_\Psi^{n+1}=\vphi''-(\p\vphi',\p\vphi')_{g_{\vphi}}\cdot\om^n_\vphi=\tau \cdot \Om^{n+1}_{\mathbf b}\text{ in }\mathfrak M.
\end{align*}
Recall the notations $\Psi(z)=\vphi(z)-|z^{n+1}|^2 $.
\begin{prop}\label{prop: interior spatial Laplacian estimate}
Assume that the background metric $\Om_{\mathbf b}$ satisfies curvature condition \eqref{backgroundmetriccurvature}. 
	Then there exists a constant $C$ such that 
	\begin{align}\label{bab claim of subsection interior spatial laplacian estimate}
	\sup_{\mathfrak{X}} \Tr_{\om}\om_\vphi\leq C .
	\end{align}
	The constant $C$ depends on the constants in \eqref{backgroundmetriccurvature} and 
	\begin{align*}
&\sup_{\p\mathfrak{X}} \Tr_{\om}\om_\vphi,\quad \osc_\mathfrak{X}\vphi.
	\end{align*}
\end{prop}
\begin{proof}
Letting $G=\vphi''-(\p\vphi',\p\vphi')_{g_{\vphi}}$, its linearised operator with variation $\delta \vphi=u$ at $\vphi$ is
\begin{align*}
L(u)=\frac{u''+g_\vphi^{i\bar l}g_\vphi^{k\bar j}u_{i\bar j}\vphi'_i\vphi'_{\bar j}-(\p u',\p\vphi')_{g_{\vphi}}-(\p\vphi',\p u')_{g_{\vphi}}}{G}+\tri_\vphi u.
\end{align*}
We denote $h=\Tr_{\om_{{\mathbf b}}}\om_\vphi$.
Using the computation in \cite{MR3298665}, we get
\begin{align*}
L(h)=|\p \log G|_{g_{{\mathbf b}}}^2+I+\frac{II+III}{G}+IV.
\end{align*} 
In which, the third term 
\begin{align*}
III\geq A G h,
\end{align*} with the notation \begin{align*}
A=\frac{[h'-(\p\vphi',\p h)_{g_\vphi}][h'-(\p\vphi',\p h)_{g_\vphi}]}{h^2 G};
\end{align*}  and
the fourth term 
\begin{align}
IV=g_\vphi^{i\bar j}g_\vphi^{p\bar q}\vphi_{i\bar q k}\vphi_{p\bar j k}\geq h|\p \log h|^2_{g_\vphi}.
\end{align}

We need to handle the terms $I$ and $II$, which contain the curvature.
We use the curvature condition \eqref{backgroundmetriccurvature}.
Similarly to the proof of Proposition \ref{bad prop: interior Laplacian estimate}, i.e. \eqref{RHS2ndinequality} and \eqref{RHS2nd},
we control the first term by
\begin{align*}
I&=-S(\om_{{\mathbf b}})+g_\vphi^{k\bar l}{R^{i\bar j}}_{k\bar l}(\om_{{\mathbf b}}){g_\vphi}_{i\bar j}\\
&\geq -C(\Tr_{\om_{\vphi}}\om_{{\mathbf b}}+\tri_{\vphi}\Phi)\Tr_{\om_{{\mathbf b}}}\om_{\vphi}.
\end{align*} 

The second term 
\begin{align*}
II&=g^{i\bar q}_\vphi g^{p\bar j}_\vphi R_{p\bar qk\bar l}g_{\vphi l\bar k}\vphi'_i\vphi'_{\bar j}\\
&\geq-2 C h \sum_i \frac{(1+\Phi_{i\bar i})\vphi'_i \vphi'_{\bar i} }{(1+\vphi_{i\bar i})^2}.
\end{align*} 

In conclusion, $L(h)$ satisfies
\begin{align*}
L(h)\geq -C(\Tr_{\om_{\vphi}}\om_{{\mathbf b}}+\tri_{\vphi}\Phi)h- \frac{2 C h}{G}\sum_i \frac{(1+\Phi_{i\bar i})\vphi'_i \vphi'_{\bar i} }{(1+\vphi_{i\bar i})^2}+Ah+h|\p \log h|^2_{g_\vphi}.
\end{align*} 

Letting $Z=\log\Tr_{\om_{{\mathbf b}}}\om_\vphi-B_1\vphi+\frac{t^2}{2}+B_2\Phi$, we have the identity
\begin{align*}
L(Z)&=\frac{L(h)}{h}-|\p \log h|_{g_\vphi}^2-A-(n+1)B_1+B_1\Tr_{\om_\vphi}\om_{{\mathbf b}}+\frac{B_1}{G} \sum_i \frac{\vphi'_i \vphi'_{\bar i} }{(1+\vphi_{i\bar i})^2}+\frac{1}{G}\\
&+\frac{B_2}{G}\sum_i \frac{\Phi_{i\bar i}\vphi'_i \vphi'_{\bar i} }{(1+\vphi_{i\bar i})^2}+B_2 \tri_\vphi \Phi.
\end{align*} 
Inserting $L(h)$ into the formula of $L(Z)$, we have
\begin{align*}
L(Z)&\geq-(n+1)B_1+(B_1-C)\Tr_{\om_\vphi}\om_{{\mathbf b}}+\frac{B_1-2C}{G} \sum_i \frac{\vphi'_i \vphi'_{\bar i} }{(1+\vphi_{i\bar i})^2}+\frac{1}{G}\\
&+\frac{B_2-2C}{G}\sum_i \frac{\Phi_{i\bar i}\vphi'_i \vphi'_{\bar i} }{(1+\vphi_{i\bar i})^2}+(B_2-C) \tri_\vphi \Phi.
\end{align*} 
We deal with the terms involved with $\Phi$, using $\Tr_{\om_\vphi}\tilde \om_{{\mathbf b}}=\Tr_{\om_\vphi}\om_{{\mathbf b}}+\tri_\vphi \Phi$,
\begin{align*}
&(B_1-C)\Tr_{\om_\vphi}\om_{{\mathbf b}}+(B_2-C) \tri_\vphi \Phi\\
&=(B_1-B_2)\Tr_{\om_\vphi}\om_{{\mathbf b}}+(B_2-C)(\Tr_{\om_\vphi}\tilde \om_{{\mathbf b}}).
\end{align*} 
and 
\begin{align*}
&\frac{B_1-2C}{G} \sum_i \frac{\vphi'_i \vphi'_{\bar i} }{(1+\vphi_{i\bar i})^2}+\frac{B_2-2C}{G}\sum_i \frac{\Phi_{i\bar i}\vphi'_i \vphi'_{\bar i} }{(1+\vphi_{i\bar i})^2}\\
&=\frac{B_1-B_2}{G} \sum_i \frac{\vphi'_i \vphi'_{\bar i} }{(1+\vphi_{i\bar i})^2}+\frac{B_2-2C}{G}\sum_i \frac{(1+\Phi_{i\bar i})\vphi'_i \vphi'_{\bar i} }{(1+\vphi_{i\bar i})^2}.
\end{align*} 
We choose $B_1=B_2+1=2C+2$, then
\begin{align*}
L(Z)&\geq-(n+1)B_1+\Tr_{\om_\vphi}\om_{{\mathbf b}}+\frac{1}{G}.
\end{align*} 
We denote the addendum $E=-B_1\vphi+\frac{t^2}{2}+B_2\Phi$, which is bounded.
When the maximal point is on the boundary, the proof is finished. 
If the maximal point $p$ appeared in the interior of $\mathfrak X$, $\Tr_{\om_\vphi}\om_{{\mathbf b}}+\frac{1}{G}$ at $p$ is bounded. We apply the inequality between two positive matrices $M$ and $N$, i.e. $\Tr_M N\leq \frac{n+1}{n}(\Tr_N M)^n\frac{\det N}{\det M}$, to see at $p$,
\begin{align*}
\Tr_{\om_{{\mathbf b}}}\om_\vphi+G\leq \frac{n+1}{n}(\Tr_{\om_\vphi}{\om_{{\mathbf b}}}+\frac{1}{G})^n\frac{\Om_\Psi^{n+1}}{\Om_{{\mathbf b}}^{n+1}}.
\end{align*} 
Thus at any point $Z$, 
\begin{align*}
\log\Tr_{\om_{{\mathbf b}}}\om_\vphi(z)\leq \log\Tr_{\om_{{\mathbf b}}}\om_\vphi(p)+E(p)-E(z).
\end{align*}
The proof is complete.
\end{proof}

Combining the rough Laplacian estimate (Proposition \ref{bad prop: interior Laplacian estimate}) with the spatial Laplacian estimate (Proposition \ref{prop: interior spatial Laplacian estimate}), we get the following global Laplacian estimate.
\begin{prop}\label{roughnormalnormal}
Assume that the background metric $\Om_{\mathbf b}$ satisfies curvature condition \eqref{backgroundmetriccurvature}. 
	Then there exists a constant $C$ (independent of $\tau$), which depends on the constants in Proposition \ref{bad prop: interior Laplacian estimate} and Proposition \ref{prop: interior spatial Laplacian estimate}, such that 
	\begin{itemize}
	\item The rough mixed singular tangent-normal estimates:
 \begin{align*}
	 \sup_{ \mathfrak X}[|s|^{\kappa}(|\frac{\p^2\Psi}{\p z^{1}\p z^{\overline{n+1}}} |_{\Om}+ |\frac{\p^2\Psi}{\p z^{1}\p z^{{n+1}}}|_{\Om})]\leq C .
	\end{align*}
	\item The rough normal-normal estimates: 
	\begin{align*}
\sup_{ \mathfrak X}[|s|^{2\kappa} (|\frac{\p^2\Psi}{\p z^{n+1}\p z^{\overline{n+1}}} |_{\Om}+ |\frac{\p^2\Psi}{\p z^{n+1}\p z^{{n+1}}}|_{\Om})]\leq  C.
\end{align*}
	\end{itemize}
\end{prop}


\subsection{Interior gradient estimate}
\subsubsection{Interior spatial gradient estimate}\label{Interior spatial gradient estimate}
We use the interior spatial Laplacian estimate (Proposition \ref{prop: interior spatial Laplacian estimate}) to prove the interior spatial gradient estimate.
\begin{prop}\label{Prop Interior spatial gradient estimate}
There is a constant $C$ depending on the constants in \lemref{Linfty estimate} and Proposition \ref{prop: interior spatial Laplacian estimate}, such that
\begin{align}
\sup_{\mathfrak X}|\p_z\vphi|_{\om} \leq C.
\end{align}
\end{prop}
\begin{proof}
According to the interior spatial Laplacian estimate (Proposition \ref{prop: interior spatial Laplacian estimate}), we have already known that $\tri_\om\vphi$ is bounded. The potential $\vphi$ is also bounded, due to \lemref{Linfty estimate}.

In order to get $|\p_z\vphi|_{\om}$, we apply the global $L^p$-estimate for cone metric to $\tri_\om\vphi$, which is glued from the local $L^p$-estimate for cone metric.
The local $L^p$-estimate is proved by applying cone Green function (see \cite{MR3668765}), there exists a constant $C$ depending on $n,p,\b$ such that
	\begin{align}\label{c1alphalocal}
	&\Vert u\Vert_{C^{1,\a,\b}(\om_{cone};B_r)}\\
	&\leq C(\Vert \tri_{\om_{cone}} u\Vert_{L^p(\om_{cone};B_{2r})}+\Vert u\Vert_{W^{1,2}(\om_{cone};B_{2r})})\nonumber.
	\end{align}

We will also need an interpolation inequality. 

\begin{lem}[Interpolation inequality]
\label{1alphainterpolation}
Suppose that $\eps>0$ and $1<p<\infty$. There exists a constant $C$ such that for all $u\in C^{1,\a,\b}(\om)$, we have
\begin{align*}
\Vert u\Vert_{W^{1,p}(\om)}\leq \eps \Vert u\Vert_{C^{1,\a,\b}(\om)}+ C\Vert u\Vert_{L^{p}(\om)}.
\end{align*}
\end{lem}

The proof of this lemma is similar to the proof of the Proposition 3.1 in \cite{arxiv:1703.06312}. We could replace $W^{2,p,\b}_{\bf s}$ with $C^{1,\a,\b}$ and use the compactness argument of $C^{1,\a,\b}$.

We continue back to our proof and patch the local estimates together. We let the manifold $B$ be covered by a finite number of coordinates charts $\{U_i,\psi;1\leq i\leq N\}$. We let $\rho_i$ be the smooth partition of unity subordinate to $\{U_i\}$ and be supported in $B_r\subset B_{3r}\subset U_i$ for each $i$. From
	\begin{align*}
	\tri_{\om_{cone}} (\rho_i \vphi)
	=\tri_{\om_{cone}}  \rho_i \vphi
	+ \rho_i \tri_{\om_{cone}}  \vphi
	+2(\p \rho_i,\p \vphi)_{\om_{cone}} ,
	\end{align*}
using the equation and putting all estimates in each $U_i$ together, we obtain from \lemref{c1alphalocal},
	\begin{align*}
	\Vert \vphi\Vert_{C^{1,\a,\b}(\om)}&=\left\Vert \textstyle\sum_{i}\rho_i \vphi\right\Vert_{C^{1,\a,\b}(\om)}\leq C\left\Vert \textstyle\sum_{i}\rho_i \vphi\right\Vert_{C^{1,\a,\b}(\om_{cone})}\\
	&\leq C(\Vert \textstyle\sum_{i}\rho_i \vphi\Vert_{W^{1,2}(\om_{cone})}
	+\Vert \tri_{\om_{cone}} (\rho_i \vphi)\Vert_{L^{p}({\om_{cone}} )}).
	\end{align*}
	Since are bounded 
		\begin{align*}
	RHS.\leq C(\Vert  \vphi\Vert_{L^{\infty}}+\Vert  \tri_{\om_{cone}}  \vphi\Vert_{L^{\infty}}+\Vert \vphi\Vert_{W^{1,p}(\om )}).
	\end{align*}
	Applying the interpolation inequality in \lemref{1alphainterpolation}, we obtain that
			\begin{align*}
	\Vert \vphi\Vert_{C^{1,\a,\b}(\om)}\leq C(\Vert  \vphi\Vert_{L^{\infty}}+\Vert  \tri_{\om_{cone}}  \vphi\Vert_{L^{\infty}}).
	\end{align*}
	Thus the interior spatial gradient estimate is proved.

\end{proof}
\subsubsection{Contradiction method}
We already have the gradient estimate from the bound of $\p_t\vphi$ and the spatial estimates in Section \ref{Interior spatial gradient estimate}. We would present an alternative proof, under the condition that the Laplacian estimate holds.
This section improves the interior gradient estimate, Section 3.4 in \cite{MR3405866}.
\begin{prop}\label{prop: interior gradient estimate}	
	Assume that
	\begin{itemize}
		\item Interior Laplacian estimate: $\sup_{ \mathfrak X}\Tr_{\Om}\Om_{\Psi}\leq C_1$;
		\item Boundary Laplacian estimate: $\sup_{\p \mathfrak X}\Tr_{\Om}\Om_{\Psi}\leq  C_2(\sup_{\mathfrak X} |\p\Psi|^2_\Om+1)\;.$
	\end{itemize}
Then there exists a constant $C$ depending on the constants in both $C_1$ and $C_2$ such that
\begin{align*}
\sup_{\mathfrak X}|\p\Psi|_{\Om} \leq C \; .
\end{align*}
\end{prop}
\begin{proof}
We prove it with an argument by contradiction. Assume that we have a sequence of $\{\Psi_s\}$ and points $\{p_s\}$ such that
$$
|\p\Psi_s|_{\Om}(p_s)=\sup_{\mathfrak X}|\p\Psi_s|_{\Om} \rightarrow \infty \text{ as }s\rightarrow \infty.
$$

We take the limit of this sequence in the underlying topology on $\mathfrak X$ 
(not in the topology induced by the cone metrics which are incomplete), 
and we denote the limit point $p$. 
We need to consider where the limit point $p$ locates.	

When $p$ locates in $\mathfrak X\setminus \mathfrak D$. The argument follows directly from \cite{MR1863016}.
The new situation is $p\in \mathfrak D$.
We will use different scaling method.

There are two sub-cases, 
\begin{itemize}
	\item $p$ is in $\mathfrak D$, but not in the boundary $\p \mathfrak X$,
	\item $p$ lies on the boundary $ \mathfrak D\cap\p \mathfrak X$.
\end{itemize}

Choose a small ball $B_{\rho_0}(p)$ centred at $p$ with fixed radius $\rho_0$ and holomorphic normal coordinate chart $z$ in $B_{\rho_0}(p)$ and $p=0$.
Because this is a local argument, we choose $\Om_{cone}=\b^2|z^1|^{2(\b-1)} i dz^1\wedge dz^{\bar 1}+\sum_{2\leq j\leq n+1} i dz^j\wedge dz^{\bar j}$ as background metric. When $p\in \mathfrak D\cap \p\mathfrak X$, we choose half ball $B_{\rho_0}^+(p)$ instead.

We assume that $\Psi_s$ achieves supremum at points $p_s\in \mathfrak X$ with
\begin{align}\label{gradientzero}
|\p_z\Psi_s|_{\Om_{cone}}(p_s)=\sup_{\mathfrak X}|\p_z\Psi|_{\Om_{cone}}=\frac{1}{m_s}\; 
\end{align} such that $m_s\rightarrow 0$ as $s\rightarrow \infty$.

\textbf{Step 1:}
For any $q\in B_{\rho_0}(p)$, 
\begin{align}\label{gradientest1st} |\p_{z}\Psi_s|_{\Om_{cone}}(q)\leq \frac{1}{m_s}.
\end{align}
Then the interior Laplacian estimate and boundary Laplacian estimate imply that 
\begin{align}
&\label{gradientest2ndhessian}
(g_{cone})_{ j\bar j}(q)+\p_{z_j}\p_{z_{\bar j}}\Psi_s(q)\leq \frac{C}{m_s^2}(g_{cone})_{ j\bar j}(q); \\
&\label{gradientest2nd}
|\p_z\p_{\bar z}\Psi_s|_{\Om_{cone}}(q)\leq \frac{C}{m_s^2}.
\end{align}
 The constant $C$ depends on the constants in both Proposition \ref{prop: interior Laplacian estimate} and Proposition \ref{prop: boundary hessian estimate}.

\textbf{Step 2:}
We define $\lambda=(\lambda_1,\cdots,\lambda_j,\cdots)$, 
\begin{align*}
\lambda_1=m_s^\frac{1}{\b},\quad\lambda_j=m_s,\quad \forall 2\leq j\leq n+1.
\end{align*}
We use the transform $T: B_{\rho_0}(p)\rightarrow B_{\frac{\rho_0}{m_s}}(0)\subset \mathbb C^n$, $z\mapsto \tilde z$, 
\begin{align*}
\tilde z^j=\frac{z^j}{\lambda_j},\quad \forall 1\leq j\leq n+1.
\end{align*}
Now we define the following rescaled sequences of functions 
\begin{align*}
&\tilde\Psi_s(\tilde z)=\Psi_s(\lambda\cdot \tilde z)\, ,\quad
\tilde\Psi_{\mathbf b}(\tilde z)=\Psi_{\mathbf b}(\lambda\cdot \tilde z)\, ,\quad
\tilde h(\tilde z)=h(\lambda\cdot \tilde z)\, ,
\end{align*}
for any $\tilde z\in B_{\frac{\rho_0}{m_s}(0)} $.
We also denote $\tilde \Om(\tilde z)=\Om(\lambda\cdot \tilde z)$ and $$\tilde p_s=(\frac{z^1(p_s)}{\lambda_1},\cdots,\frac{z^j(p_s)}{\lambda_j},\cdots,\frac{z^{n+1}(p_s)}{\lambda_1}).$$
It is direct to compute that $\tilde\Om_{cone}=m_s^2\Om_{cone}$.

Then \eqref{gradientzero}\eqref{gradientest1st}\eqref{gradientest2nd} are rescaled to be
\begin{align}
\label{gradientzeroback}
|\p_{\tilde z}\tilde \Psi_s|_{\tilde\Om_{cone}}(\tilde p_s)&=1;\\
\label{gradientest1stback}
\sup_{B_{\frac{\rho_0}{m_s}}(0)}|\p_{\tilde z}\tilde\Psi_s|_{\tilde\Om_{cone}}
&=
m_s\sup_{B_{\rho_0}(p)}|\p_z\Psi_s|_{\Om_{cone}}\leq 1 ;\\
\label{gradientest2ndback}
\sup_{B_{\frac{\rho_0}{m_s}}(0)}|\p_{\tilde z}\p_{\bar{\tilde z}}\tilde\Psi_s|_{\tilde\Om_{cone}}
&=
m_s^2\sup_{B_{\rho_0}(p)}|\p_z\p_{\bar z}\Psi_s|_{\Om_{cone}}\leq C\; .
\end{align}

\textbf{Step 3:}
Any closed set $K\subset \mathbb C^{n+1}$ stays in $B_{\frac{\rho_0}{m_s}}(0)$ for sufficient large $i$.
According to \eqref{gradientest1stback} and \eqref{gradientest2ndback}, after taking the standard diagonal sequence and the Arzela-Ascoli theorem, 
we could extract a subsequence of $\tilde\Psi_s$ converges to $\tilde\Psi_\infty$ 
in $C_{loc}^{1,\a,\b}(\mathbb{C}^n,  \Om_{cone})$ and from \eqref{gradientzeroback}, the limiting function $\tilde\Psi_\infty$ is not a constant.

\textbf{Step 4:}
When $p$ lies on the boundary $\p \mathfrak X$, we always choose half balls in the argument above.
From the $L^\infty$ estimate \eqref{Linfty estimate} i.e. $\Psi_{\mathbf b}(z)\leq\Psi_s(z)\leq h(z),  \forall s $, we have 
\begin{align*}
\tilde \Psi_{\mathbf b}(\tilde z)\leq\tilde \Psi_s(\tilde z)\leq \tilde h(\tilde z), \quad \forall s \; .
\end{align*}
So after taking $s\rightarrow\infty$, the rescaled sequence converges as
\begin{align*}
\Psi_{\mathbf b}(0)\leq\tilde\Psi_\infty(\tilde z)\leq h(0) \; .
\end{align*}
On the boundary $\p \mathfrak X$, $\Psi_{\mathbf b}(0)=h(0)$, 
thus $\tilde\Psi_\infty$ has to be constant; contradiction!

\textbf{Step 5:}
When $p$ is in $\mathfrak D$, but not in the boundary $\p \mathfrak X$, we apply \eqref{gradientest2ndhessian} on the complex plane $\mathbb C$ expended by $\{\p_{z^j},\p_{z^{\bar j}}\}$ in $B_{\rho_0}(p_s)$. It suffices to consider $j=1$, otherwise we just take $\b=1$ for $2\leq j\leq n+1$. We compute 
\begin{align*}
0<|z^1|^{2\b-2}+\p_{z_1}\p_{{z_{\bar 1}}}\Psi_s\leq \frac{C}{m_s^2}|z^1|^{2\b-2}\; 
\end{align*}
then scaling, we have on $B_{\frac{\rho_0}{m_s}}(0)$,
\begin{align*}
0<|\lambda_1\tilde z^1|^{2\b-2}+\lambda_1^{-2}\p_{\tilde z_1}\p_{\tilde z_{\bar 1}}\tilde\Psi_s\leq Cm_s^{-2}|\lambda_1\tilde z^1|^{2\b-2}\; .
\end{align*}
Then 
\begin{align*}
0<\lambda_1^{2\b}|\tilde z^1|^{2\b-2}+\p_{\tilde z_1}\p_{\tilde z_{\bar 1}}\tilde\Psi_s\leq C|\tilde z^1|^{2\b-2}\; .
\end{align*}
After taking $s\rightarrow\infty$, we have on $\mathbb C$,
\begin{align*}
0<\p_{\tilde z_1}\p_{{\tilde z_{\bar 1}}}\tilde\Psi_\infty\leq C |\tilde z^1|^{2\b-2}\; \text{ in } W^{2,p}
\end{align*}
Due to the following Liouville theorem (\propref{Liouville theorem}), $\Psi_\infty$ is a constant. This is a contradiction. Since $\tilde \Psi_s\rightarrow \tilde \Psi_\infty$ in $C_{loc}^{1,\a,\b}(\mathbb{C}^n,  \Om_{cone})$, we obtain, from \eqref{gradientzeroback}, that $|\p_{\tilde z}\tilde \Psi_s|_{\tilde\Om_{cone}}(\tilde p_s)>0.5$, when $s$ is sufficiently large.
\end{proof}

\begin{prop}[Liouville theorem]\label{Liouville theorem}
We denote the cone metric $\tilde\om=|z|^{2-2\b}idz\wedge d \bar z$ in $\mathbb C$.
Suppose $u\in C^{1,\a,\b}\cap W^{2,p}(\mathbb C)$ has bounded $C^1(\tilde{ \om})$-norm, i.e. $$|u|_{C^{1}(\tilde{ \om})}=\sup_{z\in \mathbb C}\{|u(z)|+|z|^{1-\b}|\p u(z)|\}<\infty.$$
Suppose $u$ satisfies in the distribution sense,
\begin{align*}
|z|^{2-2\b}\p_{ z}\p_{\overline { z}}u\geq0.
\end{align*} Then $u$ must be a constant.
\end{prop}
\begin{proof}
Since $u$ is bounded, we could assume that $u$ is positive. We denote $\tri$ be the associated Laplacian. 
We see that for $p\geq 2$
\begin{align*}
\tri u^p=p(p-1)|\p u|^2 u^{p-2}+p u^{p-1}\tri u
\geq p(p-1)|\p u|^2 u^{p-2} \; .
\end{align*}
Choose $0<r_0<s<r$ and a cut-off function $\eta$ with $\eta=1$ in $B_s$ and $\eta=0$ outside $B_r$. 
Let $ \chi_\eps$ be the be the smooth cut off function supported outside the $r_0$-tubular neighbourhood of the divisor with the properties that such that 
\begin{align*}
|\nabla\chi_{\ep}|=\eps \cdot O(r^{-1}).
\end{align*} Then we have
\begin{align*}
\int_{B_r}  \chi_\eps div(\eta^2\p u^p) \tilde\om
=\int_{B_r} -(\p \chi_\eps ,\p u^p)_{\tilde g} \eta^2 \tilde\om \; .
\end{align*}
Since $u$ and $|\p u|_{\tilde g}$ are both bounded we have, as $\eps\rightarrow0$,
\begin{align*}
\int_{B_r}  div_{\tilde g}(\eta^2\p u^p) \tilde\om
=0 \; .
\end{align*}
That is
\begin{align*}
0\geq\int_{B_r} 2p \eta(\p\eta,\p u)_{\tilde g} u^{p-1} \tilde\om
+\int_{B_r} p(p-1)\eta^2|\p u|_{\tilde g}^2 u^{p-2} \tilde\om \; .
\end{align*}
Applying the Schwartz inequality, we have
\begin{align*}
[\int_{B_r} \eta^2|\p u|_{\tilde g}^2 u^{p-2} \tilde\om]^2
\leq \frac{2}{p-1}\int_{B_r\setminus B_s} |\p \eta|_{\tilde g}^2 u^{p} \tilde\om
\int_{B_r\setminus B_s} \eta^2|\p u|_{\tilde g}^2 u^{p-2}  \tilde\om \; .
\end{align*}
Since $u$ is bounded, we have from the assumption that the real dimension is $2$,
\begin{align*}
\int_{B_s} \eta^2|\p u|_{\tilde g}^2 u^{p-2} \tilde\om\leq\int_{B_r} \eta^2|\p u|_{\tilde g}^2 u^{p-2} \tilde\om\leq C\frac{1}{r-s} \; .
\end{align*}
Fix $s$ and take $r\rightarrow\infty$, we obtain that $u$ is a constant.
\end{proof}







\section{Uniqueness of cscK cone metrics}\label{Uniqueness of cscK cone metrics}
We prove \thmref{Uniqueness} in this sections.\begin{proof}(proof of \thmref{Uniqueness})
We are given two cscK cone metrics $\om$. We assume that there are two different orbits $\mathcal O_1$ and $\mathcal O_2$ of cscK cone metrics. We minimise the functional $J$ (see \eqref{energyJ}) in each orbit, i.e. in $\mathcal O_1$ to get a cscK cone metric $\theta_1$, meanwhile, in $\mathcal O_2$ to get $\theta_2$ (Proposition \ref{minimiseJ}).

According to the bifurcation theorem (\thmref{bifrucationattheta} in Section \ref{bifurcation}), at each cscK cone metric $\theta_i, i=1,2$, we are able to perturb $\theta_i$ to two continuity paths of $J$-twisted cscK cone metrics $\theta_i(t)$ for $1-\tau<t\leq 1$. The bifurcation construction requires the linear theory for Lichnerowicz operator (Section \ref{Linear theory for Lichnerowicz operator}) and the reductivity of the automorphism group (Section \ref{Reductivity of automorphism group}).

The regularity theorem of the $J$-twisted cscK cone metric (Corollary \ref{dmetric} in Section \ref{Asymptotic behaviours})
tells us that its K\"ahler potential is $C^{3,\a,\b}_{\bf w}$. Then fixing $t\in(1-\tau,1)$, we could connect $\theta_1(t),\theta_2(t)$ with the cone geodesic, according to \thmref{weakgeodesicclosedness} in Section \ref{Cone geodesic}.
The $J$-twisted cscK cone metric is the critical point of the twisted log-$K$-energy $\mathcal E$.
While, the twisted log-$K$-energy is strictly convex along the cone geodesic (Proposition \ref{Econvex}). Thus the  $J$-twisted cscK cone metric is unique.

The uniqueness of the $J$-twisted cscK cone metrics implies two paths have to coincide with each other when $1-\tau<t< 1$. As a consequence, $\mathcal O_1=\mathcal O_2$. Therefore the proof of the main theorem (\thmref{Uniqueness}) is complete.
\end{proof}

\subsection{CscK cone metrics and $J$-twisted cscK cone metrics}\label{cscKcone}
We would explain in this section how to define cscK cone metric. We need to start with a reference metric. 
\subsubsection{References metrics}
There are two ways to construct reference metric $\om_\theta$.

\textbf{Solving cone Calabi's conjecture}
	In any K\"ahler class $[\om_0]$, given a smooth $(1,1)$-form $$\theta\in C_1(X)-(1-\b)C_1(L_D),$$ there exists a unique $\om_\theta:=\om_{\vphi_\theta}=\om_0+i\p\bar\p \vphi_\theta$ with $\vphi_\theta\in C^{2,\a,\b}$ such that
	\begin{align}\label{Rictheta}
	Ric(\om_\theta)=\theta+2\pi (1-\b)[D].
	\end{align}
The potential $\vphi_\theta$ satisfies the equation
\begin{align}\label{theta}
\frac{\om^n_\theta}{\om^n_0}=\frac{e^{h_0}}{|s|_h^{2-2\b}}.
\end{align}
In which, $h_0$ is a smooth function and $h$ is a smooth hermitian metric on $L_D$. The construction of weak solution could be found in \cite{MR3368100}, and actually $C^{2,\a,\b}$ in \cite{MR3488129}. According to our expansion formula of the complex Monge-Amp\`ere equation in the authors' previous article \cite{arxiv:1609.03111}, $\vphi_\theta$ has higher order estimates and the expansion formula. 

\textbf{Perturbation method}	
We could perturb the model metric $\om_D$ a little bit to have bounded Ricci curvature. This method is presented in \cite{arxiv:1612.01866}. The function $$f:=\log\frac{|s|^{2\b-2}\om^{n}_0}{\om_D^{n}}$$ is $C^{0,\a,\b}$ and could be approximated by smooth function $f_0$ in ${C^{0,\a,\b}}$ with smaller $\alpha'<\alpha$. We use the same $\alpha$, but keep in mind that $\alpha$ could be adjusted to be smaller. By the implicit function theorem and resolvability of the linear equation with cone coefficients, we could solve $\om_\theta:=\om+i\p\bar{\p} \vphi_\theta$ with $\vphi_\theta$ in $C^{2,\a,\b}$ such that
\begin{align}\label{approximationback}
\frac{\om^{n}_{\theta}}{\om_D^{n}}=e^{f-f_0} \text{ in } M.
\end{align} Then direct computation shows that 
\begin{align*}
Ric(\om_{\theta})=Ric(\om_0)+i\p\bar{\p} f_0+i\p\bar{\p }\log |s|^{2-2\b}_h:=\theta+2\pi (1-\b)[D].
\end{align*}


\subsubsection{Constant scalar curvature K\"ahler cone metrics}
Motivated from the definition of the K\"ahler-Einstein cone metrics, we consider the constant scalar curvature K\"ahler cone metrics. Some progress in this direction has already been made in very recent papers \cite{MR3496771,arXiv:1506.06423,arxiv:1511.00178,arxiv:1703.06312,arxiv:1508.02640}.
\begin{defn}\label{csckconemetricdefn}
	We say $\om_{cscK}=\om_0+i\p\bar\p\vphi$ is a constant scalar curvature K\"ahler cone metric, if 
	\begin{enumerate}
		\item $\om_{cscK}$ is a K\"ahler cone metric with angle $\b$;
		\item $\om_{cscK}$ satisfies the equations
                 \begin{numcases}{}
		\label{2nd equ}
		\frac{\om_{cscK}^n}{\om_\theta^n}=e^P,\\
		\label{P}
		\tri_{\om_{cscK}} P=\Tr_{\om_{cscK}}\theta-\underline S_\b;
		\end{numcases} 
		Here $\theta$ is a smooth $(1,1)$-form  satisfies \eqref{Rictheta}.

		\item its K\"ahler potential $\vphi\in C^{2,\a,\b}$ and the H\"older exponent $\alpha$ satisfies	
		\begin{align}\label{anglealpha}
		\a\b<1-\b.
		\end{align}
	\end{enumerate} 
\end{defn}
The condition $(3)$ could be removed. When $ C^{-1}\om_D\leq \om_{cscK}\leq C\om_D$, for a positive constant $C$, by condition $(1)$, we have by applying Moser iteration (see \cite{arxiv:1511.02410}) to \eqref{P} that $P\in C^{0,\a,\b}$, and then from Schauder estimate $\vphi\in C^{2,\a,\b}$. 

Since the cscK cone potential $\vphi_{cscK}$ is in $C^{2,\a,\b}$, we are able to solve the second linear equation to obtain that $P$ is $C^{2,\a,\b}$, according to the 2nd order linear Schauder theory for K\"ahler cone metrics.



Using the equation \eqref{theta} of $\om_\theta$ to \eqref{2nd equ} in the Definition \ref{csckconemetricdefn}, we have the potential equation,
\begin{align}\label{2nd equ smooth}
\frac{\om_{cscK}^n}{\om_0^n}=\frac{e^{P+h_0}}{|s|_h^{2-2\b}},
\end{align}

We denote the tensor
\begin{align*}
T=-i\p\bar \p P+\theta.
\end{align*}
\begin{lem}\label{Ttensor}
	The tensor $T$ is $C^{0,\a,\b}$, and $Ric(\om_{cscK})=T+2\pi(1-\b)[D]$. Moreover, the volume an the averaged scalar curvature $\underline S_\b$ are well-defined topological invariant.
\end{lem}
\begin{proof}
	The first conclusion holds, since $P\in C^{2,\a,\b}$ and $\theta$ is smooth. The second conclusion follows from \eqref{2nd equ},
	\begin{align*}
	Ric(\om_{cscK})=Ric(\om_\theta) - i\p\bar\p P=\theta+2\pi (1-\b)[D]- i\p\bar\p P.
	\end{align*}
	Thus $Ric(\om_{cscK})$ has lower bound, i.e. $Ric(\om_{cscK})\geq -C\cdot \om_{cscK}$.
We will see that the scalar curvature of $\om_{cscK}$ is equal to the averaged scalar curvature $\underline S(\om_{cscK})$ over the regular part $M$, 
\begin{align*}
\underline S(\om_{cscK})=\frac{\int_M Ric(\om_{cscK})\wedge\om^{n-1}_{cscK}}{\int_M\om^{n}_{cscK}}
\end{align*}
and it is a topological invariant, denoted by $\underline S_\b$. The observation is  Proposition 2.8 in \cite{arxiv:1511.00178}, i.e. when $\om_{cscK}$ has $C^{2,\a,\b}$ K\"ahler potential, not only the K\"ahler potential is for sure globally bounded, but also its Ricci potential is also globally bounded. Precisely, when given two cscK cone metrics $\om_i$, $i=1,2$, we could choose a large constant $C$ such that they are both controlled by the model metric multiplying with the constant $C$, that is $$\om_i\leq C\om_D.$$
Thus the lower bound of the Ricci curvature implies that the twisted Ricci curvature $T_i=Ric(\om_i)+C\om_D$ is non-negative. Since $$T_1-T_2=i\p\bar\p\log\frac{\om^n_2}{\om^n_1},$$ the Ricci potential $\log\frac{\om^n_2}{\om^n_1}$ is globally bounded. We are able to apply Theorem 1.14 in \cite{MR2746347} with $u=1$ and $v=\log\frac{\om^n_2}{\om^n_1}$ to obtain that
\begin{align*}
\int_M (Ric(\om_{1})-Ric(\om_{2}))\wedge\om^{n-1}_{1}=0
\end{align*}
Similarly, we have
\begin{align*}
\int_M (\om_{1}-\om_{2})\wedge Ric(\om_i)\wedge\om^{n-1}_{i}=\int_M (\om_{1}-\om_{2})\wedge (Ric(\om_i)+C\om)\wedge\om^{n-1}_{i}=0
\end{align*}
Arranging them together, we have that
\begin{align*}
\int_M Ric(\om_{1})\wedge\om^{n-1}_{1}=
\int_M Ric(\om_{2})\wedge\om^{n-1}_{2}.
\end{align*}
The volume of $\om_i$ is also topological invariant, since their K\"ahler potential is $C^{2,\a,\b}$ and the integration by parts works.

\end{proof}

Assume that the potential function $\vphi_{cscK}$ of a cscK cone metric is $C^{2,\a,\b}$ with the angle $0<\b<\frac{1}{2}$ and H\"older exponent $\a\b < 1-2\b$. Then $\vphi_{cscK}$ is actually in $C^{4,\a,\b}$ in \cite{arxiv:1603.01743}.
We are going to prove the higher regularity of $\vphi_{cscK}$ for any $0<\beta\leq1$ in Section \ref{Higher order estimates} and the following Sections, i.e \thmref{thm:interior}. The theorem implies immediately that, $\vphi\in C_{\mathbf w}^{3,\a,\b}$, according to \corref{dmetric} and \corref{connection}.

\subsubsection{$J$-twisted constant scalar curvature K\"ahler cone metrics}
We recall that $\om_0$ is a smooth K\"ahler metric and $h$ is a smooth Hermitian metric on $L_D$.

We recall the following functionals, regarding to the smooth background metric $\om_0$,
\begin{align*}
E_{\om_\theta,\b}(\vphi)&=\frac{1}{V}\int_M\log\frac{\om^n_\vphi}{\om_0^n|s|_h^{2\b-2}e^{h_0}}\om_\vphi^{n},\\
D_{\b}(\vphi)
&=\frac{1}{V}\frac{1}{n+1}\sum_{j=0}^{n}\int_{M}\vphi\om_0^{j}\wedge\om_{\varphi}^{n-j},\\
j_{\chi,\b}(\vphi)&=-\frac{1}{V}\sum_{j=0}^{n-1}\int_{M}\vphi
\om_0^{j}\wedge
\om_\vphi^{n-1-j}\wedge \chi.
\end{align*}
Here $\chi$ is the Ricci curvature of $\om_0$ twisted by the curvature of the smooth Hermitian metric $h$, i.e. $-(1-\b)i\p\bar\p\log h$, that is by \eqref{theta},
\begin{align*}
\chi=Ric(\om_\theta)=Ric(\om_0)-i\p\bar{\p} h_0+i\p\bar{\p }\log |s|^{2-2\b}_h.
\end{align*}
Denoting the smooth function $f:=-h_0+(1-\b)i\p\bar\p\log h$ on $M$, by \eqref{Rictheta}, we have the smooth $(1,1)$-form 
\begin{align*}
\chi=\theta=Ric(\om_0)+i\p\bar{\p} f.
\end{align*}

The log-$K$-energy defined over the space of K\"ahler cone metrics $\mathcal H_\b$ is
\begin{align*}
\nu_{\b}(\vphi)
&=E_{\om_\theta,\b}(\vphi)
+\ul{S}_\b\cdot D_{\b}(\vphi)+j_{\chi,\b}(\vphi)-\frac{1}{V}\int_M f\om^n.
\end{align*}

The log-$J$-functional is defined on $\mathcal H_\b$ as following,
\begin{align}\label{energyJ}
J_{\b}(\vphi)
&=-D_{\b}(\vphi)+\frac{1}{V}\int_{M}\vphi\om_0^{n}\\
&=-\frac{1}{V}\frac{1}{n+1}\sum_{i=0}^{n}\int_{M}\vphi\om_0^{i}\wedge\om_{\varphi}^{n-i}
+\frac{1}{V}\int_{M}\vphi\om_0^{n}\nonumber.
\end{align}
\begin{lem}
The cscK cone metrics (Definition \ref{csckconemetricdefn}) are the critical points of the log-$K$-energy.
\end{lem}
\begin{proof}
We compute the first variation of the log-$K$-energy with the variation $\dot \vphi$ of $\vphi$, both in $C^{2,\a,\b}$,
\begin{align*}
\delta \nu(\dot\vphi)=\delta E(\dot{\vphi})+\frac{1}{V}\int_M\dot \vphi(\ul{S}_\b-\Tr_{\om_{\vphi}}\theta)\om^n_\vphi.
\end{align*}
and
\begin{align*}
\delta E(\dot{\vphi})=\frac{1}{V}\int_M\tri_\vphi\dot \vphi\cdot\om^n_\vphi+\frac{1}{V}\int_M\log\frac{\om^n_\vphi}{\om^n_\theta}\cdot\tri_\vphi\dot \vphi\cdot\om^n_\vphi.
\end{align*}
If $\om_\vphi=\om_{cscK}$, we insert the cscK equations \eqref{2nd equ} and \eqref{P} to the identities above, 
\begin{align*}
\delta \nu(\dot\vphi)=\frac{1}{V}\int_MP\cdot\tri_\vphi\dot \vphi\cdot\om^n_\vphi-\frac{1}{V}\int_M\dot \vphi\cdot\tri_{cscK} P\cdot\om^n_\vphi.
\end{align*}
It is equal to zero, since both $P$ and $\dot\vphi$ are $C^{2,\a,\b}$ and we could apply the integration by parts formula.
\end{proof}
\begin{rem}	
We could also use the K\"ahler cone metric $\om_\theta$ constructed in \eqref{theta} to define the functionals above.
\end{rem}

We consider the twisted log-$K$-energy over $\mathcal H_\b$,
\begin{align}\label{energyE}
\mathcal E_{\b}(\vphi)=\nu_{\b}(\vphi)+(1-t)J_{\b}(\vphi).
\end{align}
In this paper, we remove the lower index $\b$ without confusion when we use these functionals.

We call its critical points $J$-twisted cscK cone metrics. Letting \[\gamma=\underline S_\b+(1-t)(\frac{\om_0^n}{\om_\vphi^n}-1),\] we see that the $J$-twisted cscK cone metric $\om_\vphi$ solves the following equations,
\begin{align}\label{Jcsckcone}
\left\{
\begin{array}{ll}
		\frac{\om_{\vphi}^n}{\om_\theta^n}=e^P,\\
		\tri_{\om_{\vphi}} P=\Tr_{\om_{\vphi}}\theta-\gamma.
\end{array}
\right.
\end{align}
It further implies the following equation outside the divisor, 
\begin{align*}
S(\om_\vphi)-\underline S_\b-(1-t)(\frac{\om_0^n}{\om_\vphi^n}-1)=0.
\end{align*}
Thanks to the reference metric $\theta$ in \eqref{theta}, the $J$-twisted cscK cone metric $\vphi\in C^{2,\a,\b}$ satisfies the equation,
\begin{align}\label{cscKconeapproximation}
\frac{\om_\vphi^n}{\om_0^n}=\frac{e^{P+h_0}}{|s|_h^{2-2\b}}.
\end{align}

The following lemma is obvious from the formulas of the functional.
\begin{lem}
The log-$K$-energy $\nu$, the log energy functional $\mathcal E$ as well as log-$D$ and log-$j$ functionals are well-defined on the space $\mathcal H_\tri:=\{\vphi\in\mathcal H_\b\vert \sup_{\mathfrak X} \{|\vphi|+|\p_z\vphi|_\om+|\p_z{\p_{\bar z}}\vphi|_{\om}\}<\infty\}$, particularly the $C^\b_\tri$ generalised cone geodesic.
\end{lem}

\begin{lem}\label{Dweakconvex}
	Assume $\{\vphi(t),0\leq t\leq 1\}$ is a $C^{1,1,\b}_{\bf w}$ cone geodesic. Then the log-$D$-functional is convex in the distribution sense, i.e. letting $\eta$ be a smooth non-negative cut-off function supported in the interior of $[0,1]$, we have
	\begin{align*}
	\int_0^1 \p_t^2\eta \cdot Ddt=\int_0^1\eta\cdot[\vphi''-(\p\vphi',\p\vphi')_{g_{\vphi}}]dt,
	\end{align*}
	otherwise $\vphi(t)$ is a constant geodesic.
\end{lem}
\begin{proof}
Along the smooth geodesic, we know that
\begin{align*}
\p_t^2 D&=\frac{1}{V}\int_M[\vphi''-(\p\vphi',\p\vphi')_{g_{\vphi}}]\om^n_{\vphi}\\
&=\frac{1}{V}\int_M\Om^{n+1}_{\Psi}.
\end{align*}	
In order to prove this identity in the distribution sense along the cone geodesic, we use the approximation of the pluri-subharmonic function. Since $\Psi$ is H\"older continuous, there exists a smooth sequence $\Psi_s$ decreasing to $\Psi$ in any open subset of $[0,1]\times X$, (see e.g. \cite{demaillybook,MR2299485}). Thus along the smooth approximation, we have
	\begin{align*}
\int_0^1 \p_t^2\eta \cdot D(\vphi_s)dt=\int_0^1\eta\cdot\frac{1}{V}\int_M\Om^{n+1}_{\Psi_s}dt,
\end{align*}
From the $C^{1,1,\b}_{\bf w}$ estimates of $\Psi$, we know $\Om^{n+1}_{\Psi_s}=O(|z^1|^{2\b-2-2\kappa})$. It is integrable, since $2\b-2-2\kappa=-2+2\a\b>-2$. The Monge-Amp\`ere operator is continuous under deceasing sequence $\Psi_s$ by Bedford-Taylor's monotonic continuity theorem, (see e.g. \cite{MR2352488}), thus the RHS converges to \begin{align*}
\int_0^1\eta\cdot\frac{1}{V}\int_M\Om^{n+1}_{\Psi}dt,
\end{align*} by Lebesgue's dominated convergence theorem. The log-$D$-functional is well defined along $C^\b_\tri$ generalised cone geodesic, the LHS converges by the dominated convergence theorem too. In conclusion, we prove that along the $C^{1,1,\b}_{\bf w}$ cone geodesic,
	\begin{align*}
\int_0^1 \p_t^2\eta \cdot D(\vphi)dt=\int_0^1\eta\cdot\frac{1}{V}\int_M[\vphi''-(\p\vphi',\p\vphi')_{g_{\vphi}}]dt.
\end{align*}

\end{proof}

Similarly, we have
\begin{lem}\label{jweakconvex}
	Assume $\{\vphi(t),0\leq t\leq 1\}$ is a $C^{1,1,\b}_{\bf w}$ cone geodesic. Then the log-$j$- functional is convex in the distribution sense, i.e. letting the test function $\eta$ as above, we have
	\begin{align*}
	\int_0^1 \p_t^2\eta \cdot jdt=\int_0^1\eta\cdot \Om_\Psi^{n}\wedge\chi dt,
	\end{align*}
	otherwise $\vphi(t)$ is a constant geodesic.
\end{lem}
Then, we apply the lemma to $J=-D+\frac{1}{V}\int_{M}\vphi\om_0^{n}$ and carry on the same procedure for
\begin{align*}
\p_t^2 [\frac{1}{V}\int_{M}\vphi\om_\theta^{n}]=\frac{1}{V}\int_M\vphi''\om_\theta^n.
\end{align*} We prove the following strict convexity of the $J$-functional.
\begin{lem}\label{Jweakconvex}
	Assume $\{\vphi(t),0\leq t\leq 1\}$ is a $C^{1,1,\b}_{\bf w}$ cone geodesic. Then the log-$J$-functional is strictly convex in the distribution sense, i.e. letting the test function $\eta$ as above, we have
	\begin{align*}
	\int_0^1 \p_t^2\eta \cdot Jdt=\int_0^1\eta\cdot(\int_M |\p\vphi'|_\vphi^2\om_\theta^n)dt>0,
	\end{align*}
	otherwise $\vphi(t)$ is a constant geodesic.
\end{lem}

\begin{lem}\label{DJjcontinuity}
	The log-$D$-functional, log-$J$-functional and log-$j$-functional are continuous along the $C^{1,1,\b}_{\bf w}$ cone geodesic.
\end{lem}
\begin{proof}
	Given two points $\vphi_i=\vphi(t_i),i=1,2$ in the $C^{1,1,\b}_{\bf w}$ cone geodesic,
we have their decreasing approximation $\vphi_{i\eps}$ as \cite{MR2299485,MR1863016}. The cocycle condition of $D$, i.e.
\begin{align*}
D(\vphi_{1\eps})-D(\vphi_{2\eps})
&=\frac{1}{V}\frac{1}{n+1}\sum_{j=0}^{n}\int_{M}(\vphi_{1\eps}-\vphi_{2\eps})\om_{\vphi_{1\eps}}^{j}\wedge\om_{\varphi_{2\eps}}^{n-j},
\end{align*}
implies that $D$ is bounded by $|\vphi_{1\eps}-\vphi_{2\eps}|_\infty$ multiplying with a constant depending on the spatial Laplacian estimate of $\vphi_i$. The approximation $D_{i\eps}$ converges to $D_{i}$ by the continuity of the weighted Monge-Amp\'ere operator \cite{MR2352488}. Thus $D$ is continuous along the $C^{1,1,\b}_{\bf w}$ cone geodesic.

The continuity of $j$ follows in the same way, but using the cocycle condition that
\begin{align*}
j(\vphi_1)-j(\vphi_2)=-\frac{1}{V}\sum_{j=0}^{n-1}\int_{M}(\vphi_1-\vphi_2)
\om_{\vphi_1}^{j}\wedge
\om_{\vphi_2}^{n-1-j}\wedge \chi.
\end{align*}
Since $\int_{M}\vphi\om_0^{n}$ is continuous when $\vphi(t)$ is continuous on $t$, we have the continuity of $J$.
\end{proof}	
Combining the lemmas above, we have
\begin{prop}\label{Jconvex}
	Assume $\{\vphi(t),0\leq t\leq 1\}$ is a $C^{1,1,\b}_{\bf w}$ cone geodesic. Then the log-$J$-functional is strictly convex along $\vphi(t)$, otherwise $\vphi$ is a constant geodesic.
\end{prop}

We then consider the entropy on the cone geodesic.
\begin{lem}\label{lsc}
Then the log-entropy $E(\vphi)$, and also the log-$K$-energy are lower semi-continuous along the $C^{1,1,\b}_{\bf w}$ cone geodesic.
\end{lem}
\begin{proof}
Since $\vphi(t)\in C^\b_\tri$, the volume ratio $h(\vphi)=\frac{\om^n_\vphi}{\om^n}$ is uniformly bounded and non-negative. The sequence $h_i=\frac{\om^n_{\vphi_i}}{\om^n}$ is then uniformly bounded in $L^1(\om_0)$ and $L^p(\om_0)$ for some $p>1$. So it has a weakly $L^1(\om_0)$ convergent subsequence. The subsequence also converges weakly star to a limit $f$, which is $L^\infty$. Therefore the assumption of Lemma 4.7 in \cite{MR3582114} is verified and $E$ is lower semi-continuous. Then log-$K$-energy is lower semi-continuous, since $D$ and $j$ are continuous. 
\end{proof}


\subsection{Approximation of $J$-twisted cscK cone metrics}\label{Approximation}
In order to construct generalised cone geodesic between $J$-twisted cscK cone metrics, we need the smooth approximation equation \eqref{smooth per equ a} with smooth boundary values. So in this section, we construction a smooth approximation sequence the $J$-twisted cscK cone metric.

\subsubsection{Approximation of the reference metric $\om_\theta$}
We solve smooth $\vphi_{\theta_\eps}$ which solves the following approximation equation,
\begin{align}\label{thetaeps}
\frac{\om^n_{\theta_\eps}}{\om^n_0}=\frac{e^{h_0}}{(|s|^2_h+\eps)^{1-\beta}}.
\end{align}
Direct computation shows that
\begin{align}\label{thetaeps}
Ric(\om_{\theta_\eps})=Ric(\om_0)-i\p\bar\p h_0+(1-\beta)i\p\bar\p\log(|s|^2_h+\eps).
\end{align}
By \eqref{theta}, the right hand side becomes
\begin{align*}
&=Ric(\om_\theta)-(1-\beta)i\p\bar\p\log|s|^2_h+(1-\beta)i\p\bar\p\log(|s|^2_h+\eps)\\
&=\theta-i\p\bar\p\log h+(1-\beta)i\p\bar\p\log(|s|^2_h+\eps).
\end{align*}
Then we use 
\begin{align*}
i\p\bar\p\log(|s|^2_h+\eps)\geq
\frac{|s|^2_h}{|s|^2_h+\eps}i\p\bar\p\log |s|^2_h\geq i\p\bar\p\log h,
\end{align*}
to see
\begin{align}\label{thetaepstheta}
Ric(\om_{\theta_\eps})
\geq\theta.
\end{align}
\subsubsection{Approximation of $P$}
Let $\eta_\eps$ be a smooth sequence approximating the volume ratio $\frac{\om_\vphi^n}{|s|_h^{2\b-2}\om_0^n}$ in $C^{0,\a,\b}$.
Then we solve the following equation to get $\vphi_\eps$,
\begin{align}
\frac{\om^n_{\vphi_\eps}}{\om^n_0}=\frac{\eta_\eps}{(|s|^2_h+\eps)^{1-\beta}}.
\end{align}
While, we solve $P_\eps$ from solving
\begin{align}
&\tri_{\om_{\vphi_\eps}} P_\eps=\Tr_{\om_{\vphi_\eps}}\theta-\gamma.
\end{align}
Since $\vphi_\eps$ has uniform $C^{2,\a,\b}$ bound and $\theta, \gamma$ have uniform $C^{0,\a,\b}$ bound, we obtain uniform $C^{2,\a,\b}$ estimate of $P_\eps$.
\subsubsection{Approximation of $\vphi$}
We approximate $\vphi$ by the smooth sequence $\psi_\eps$ solving,
\begin{align} \label{psieps}
\frac{\om_{\psi_\eps}^n}{\om_0^n}=\frac{e^{P_\eps+h_0}}{(|s|^2_h+\eps)^{1-\beta} }.
\end{align}
Since $P_\eps$ is uniformly bounded, $\psi_\eps$ has uniform $C^\alpha$ bound and converges to $\psi_0$ in $C^{\alpha'}$ for any $\alpha'<\alpha$ as $\eps\rightarrow 0$.
Because both $\omega_{\psi_\eps}^n$ and $\omega_{\vphi_\eps}^n$ converge to $\om_\vphi^n$ in $L^p$-norm, we see that
	\begin{align*}
	\lim_{\eps\rightarrow 0}\psi_\eps=\vphi+constant.
	\end{align*}
Furthermore, the smooth sequence $\psi_\eps$ converges to $\vphi$ smoothly outside the divisor.

We compute the Ricci curvature of $\om_{\psi_\eps}$. By \eqref{psieps},
	\begin{align*}
	Ric(\om_{\psi_\eps})
	=Ric(\om_0)-i\p\bar\p h_0+(1-\beta)i\p\bar\p\log(|s|^2_h+\eps)-i\p\bar\p P_\eps.
	\end{align*}
	Then by \eqref{thetaeps}, the RHS
	\begin{align*}
	=Ric(\om_{\theta_\eps})-i\p\bar\p P_\eps.
	\end{align*}
	Using \eqref{thetaepstheta}, we have
	\begin{align*}
	Ric(\om_{\psi_\eps})\geq T_\eps:=\theta-i\p\bar\p P_\eps.
	\end{align*}
	Furthermore, we could see that the scalar curvature of $\om_{\psi_\eps}$ is 
\begin{align*}
S(\om_{\psi_\eps})=\Tr_{\om_{\psi_\eps}} (T_\eps-i\p\bar\p\log h+(1-\beta)i\p\bar\p\log(|s|^2_h+\eps)).
\end{align*}




\subsection{Convexity along cone geodesics}\label{Convexity along the generalised cone geodesic}
The convexity of the $K$-energy along the $C^{1,1}$-geodesic is proved in \cite{MR3671939,MR3582114} and extended to $C^{1,1,\b}$ cone geodesic in \cite{arxiv:1511.00178}. We adapt the convexity to $C^{1,1,\b}_{\mathbf w}$ generalised cone geodesic in this section. The delicate issue is that we lose regularity along the directions  $|\frac{\p^2\vphi}{\p z^{1}\p t} |_{\Om}$ and $|\frac{\p^2\vphi}{\p t^2} |$, but they could be bounded with proper weights, i.e. $|s|^{\kappa}|\frac{\p^2\vphi}{\p z^{1}\p t} |_{\Om}$ and $|s|^{2\kappa} |\frac{\p^2\vphi}{\p t^2} |$. We observe that that the convexity still survives.
\begin{prop}\label{convexityKenergy}
The log-$K$-energy is continuous and convex along the $C^{1,1,\b}_{\mathbf w}$ cone geodesic. 
\end{prop}
\begin{proof}
It's sufficient to prove the log-$K$-energy is convex. As long as we have convexity, it implies the upper semi-continuity. Then \lemref{lsc} provides the lower semi-continuity. So the log-$K$-energy is continuous. We follows the proof along the spirit of \cite{MR3671939,MR3582114,arxiv:1511.00178}, which contains three steps.
Let $\eps,\delta$ be small constants.

\textbf{Step 1:}  Since the potential $\Psi$ of the cone geodesic $\Om_\Psi$ is H\"older continuous, there exits a smooth approximation $\Om_\delta$ in the interior of the product manifold $R_0\times X\subset\subset \mathfrak X$ (\cite{demaillybook,MR2299485}). We use $\vphi_\delta$ to denote the potential of the restriction of $\Om_\delta$ on each fibre $\{t\}\times X$, i.e. 
\begin{align*}
\Om_\delta\vert_{\{t\}\times X} =\om_0 + i\p\bar\p\vphi_\delta = \om_{D\delta} + i\p\bar\p (\vphi_\delta-\vphi_{D\delta}),\quad
\vphi_{D\delta}=\delta_0 i \p\bar\p (|s|^{2}_{h}+\delta)^\b
\end{align*}
 and $\vphi$ is the restriction of the cone geodesic, respectively. We also denote $$\tilde \om_\delta=C\delta\om_0+\Om_\delta\vert_{\{t\}\times X} \geq 0$$ on $\{t\}\times X$ and the following estimates on each fibre $\{t\}\times X$,
\begin{itemize}
\item $\sup_{\{t\}\times X}|\vphi_\delta-\vphi|_{C^{0}}\rightarrow 0$, as $\delta\rightarrow 0$.
\item $|i\p\bar\p\vphi_\delta|_\om$ is uniformly bounded and then $g_{\vphi_\delta}$ converges to $g_\vphi$ in $L^p$-norm for some $p>1$.
\end{itemize}

\textbf{Step 2:}  On each fibre $\{t\}\times X$, $C_1(K_X+D)+\frac{\tilde\om_\delta}{\eps}$ is K\"ahler. We solve the Monge-Amp\`ere equation of $\vphi_\theta$ which depends on the parameters $\theta=(\eps,\delta,\mu)$,
\begin{align*}
(-\chi+\frac{\tilde\om_\delta}{\eps}+i\p\bar\p \vphi_\theta)^n=\eps^{-n} e^{\vphi_\theta}\frac{\om_0^n}{(|s|_h^2+\mu)^{1-\b}}.
\end{align*}
The parameter $\mu$ depends on the parameter $\delta$ and will be determined later .
The equation is equivalent to
\begin{align}\label{fibre approximation equation}
(-\eps \chi+\tilde\om_\delta+\eps i\p\bar\p \vphi_\theta)^n=e^{\vphi_\theta}\frac{\om_0^n}{(|s|_h^2+\mu)^{1-\b}}
\end{align}
and the solution satisfies the following properties on each fibre $\{t\}\times X$.
\begin{enumerate}
\item The form $-\chi+\frac{\tilde\om_\delta}{\eps}+i\p\bar\p \vphi_\theta>0$, and it could be extended to the product manifold $R_0\times X$, according to Theorem 4.1 in \cite{MR3582114}. That is 
\begin{align}\label{globalpositivity}
-\chi+\frac{\tilde\Om_\delta}{\eps}+i\p\bar\p \vphi_{\theta}\geq 0.
\end{align}

\item There exists a constant $C$ independent of $\eps,\delta$ such that 
\begin{align}\label{zeroupper}
\sup_{\mathfrak X}\vphi_{\theta}\leq C.
\end{align}
\begin{proof}
At the maximum point $q$ of $\vphi_\theta$, $i\p\bar\p \vphi_\theta(q)\leq 0$. By \eqref{fibre approximation equation},
we have $$e^{\vphi_\theta}\leq \frac{-\eps \chi+\tilde\om_\delta}{\om^n (|s|_h^s+\mu)^{\b-1}}(q).$$ Since the numerator is cone metric, i.e. $O((|s|_h^s+\delta)^{\b-1})$, we could choose $\mu$ depending on $\delta$ such that the fraction is bounded.
\end{proof}

\item There exists a constant $C$ independent of $\eps,\delta$ such that 
\begin{align}\label{zerolower}
 -\eps\inf_{\mathfrak X}\vphi_{\theta}\leq C.
\end{align}
\begin{proof}
The Monge-Amp\`ere equation \eqref{fibre approximation equation} also reads 
\begin{align}\label{fibre approximation equation re}
(\om_{\eps,\delta}+i\p\bar\p \tilde \vphi_\theta)^n=e^{\frac{\tilde \vphi_\theta-\vphi_\delta+\vphi_{D\delta}}{\eps}}\frac{\om_0^n}{(|s|_h^2+\mu)^{1-\b}}.
\end{align}
with background K\"ahler cone metric $\om_{\eps,\delta}=-\eps \chi+C\delta \om_0+\om_{D\delta}$ and $\tilde \vphi_\theta=\vphi_\delta-\vphi_{D\delta}+\eps \vphi_\theta$.
At the minimum point of $\tilde \vphi_\theta$, $i\p\bar\p \tilde\vphi_\theta\geq 0$. By \eqref{fibre approximation equation re},
we have $$e^{\frac{\tilde \vphi_\theta(q)-\vphi_\delta(q)+\vphi_{D\delta}(q)}{\eps}}\geq \frac{\om^n_{\eps,\delta}}{\om_0^n( |s|_h^s+\mu)^{\b-1}}(q).$$ Since  $\om^n_{\eps,\delta}$ is $O((|s|_h^s+\delta)^{\b-1})$, the lower bound of the right hand side is strictly positive, denoted by $C$. Thus, $$\eps \vphi_\theta\geq -\vphi_\delta+\vphi_{D\delta}+\eps\log C+\vphi_\delta(q)-\vphi_{D\delta}(q),$$ which is bounded below.
\end{proof}

\item The solution $\vphi_{\theta}(t)$ is equicontinuous. From equicontinuity of $\vphi_{\theta}=\vphi_{\eps,\delta,\mu(\delta)}$, when $\delta\rightarrow 0$, we have $\vphi_{\eps,\delta,\mu(\delta)}$ converges to a limit $\vphi_\eps$ in $C^0$-norm.
\begin{proof}
Let $\tilde\om_{\vphi_\theta(t)}=-\eps \chi+\tilde\om_\delta(t)+\eps i\p\bar\p \vphi_\theta(t)$. Then we compare two solutions at $t_1,t_2$,
\begin{align*}
(\om_{\eps,\delta}+i\p\bar\p (\tilde \vphi_\theta(t_1)-\tilde \vphi_\theta(t_2)))^n=e^{\frac{(\tilde \vphi_\theta(t_1)-\tilde \vphi_\theta(t_2))-(\vphi_\delta(t_1)-\vphi_\delta(t_2))}{\eps}}\tilde\om^n_{\vphi_\theta(t_2)}.
\end{align*}
Since $\p_t \vphi$ is bounded, we have $\p_t \vphi_\delta$ is uniformly bounded. The maximum principle implies the equicontinuity of $\vphi_{\theta}(t)$ on $t$. Similarly, using $|\p_z\vphi|_{\om_D}$ is uniformly bounded. We obtain the conclusion.
\end{proof}

\item There exists a constant $C$ independent of $\eps,\delta$ such that 
\begin{align*}
\eps\sup_{\mathfrak X}|i\p\bar{\p}\vphi_{\theta}|_{\om_{D\delta}}\leq C.
\end{align*}
\begin{proof}
It is similar to the proof of Proposition \ref{Interior Laplacian estimate}. We use $\om_{\mathbf b_\eps}$ in Section \ref{Geometric conditions on the background metric}, choose background K\"ahler cone metric $$\om=-\eps \chi+C\delta \om_0+\om_{\mathbf b_\eps}.$$ We denote \begin{align*}
\tilde\om=\om+i\p\bar\p \tilde \vphi_\theta,\quad \tilde \vphi_\theta= \vphi_\delta+\eps \vphi_\theta-\vphi_{\mathbf b_\eps}
\end{align*} and rewrite \eqref{fibre approximation equation re} as
\begin{align}\label{fibre approximation equation re re}
\tilde\om^n=e^{\frac{\tilde \vphi_\theta-\vphi_\delta+\vphi_{\mathbf b_\eps}}{\eps}+F}\om^n.
\end{align}
In which, we denote \begin{align*}F&=\log\frac{\om_0^n(|s|_h^2+\mu)^{\b-1}}{\om^n}, \\
 \tilde F&= F+\frac{\tilde \vphi_\theta-\vphi_\delta+\vphi_{\mathbf b_\eps}}{\eps}.
\end{align*}
Yau's computation shows that
\begin{align*}
\tilde\tri \log\Tr_{\om}\tilde\om\geq \frac{- g^{i\bar j}R_{i\bar j}(\tilde\om)+\tilde g^{k\bar l}{R^{i\bar j}}_{k\bar l}(\om)\tilde g _{i\bar j}}{\Tr_{{\om}}\tilde\om}.
\end{align*}
Using \eqref{fibre approximation equation re re}, we have $Ric(\tilde \om)=Ric(\om)-i\p\bar\p\tilde F$ and then
\begin{align*}
RHS.\geq \frac{\tri\tilde F- S(\om)+\tilde g^{k\bar l}{R^{i\bar j}}_{k\bar l}(\om)\tilde g _{i\bar j}}{\Tr_{\om}\tilde\om}.
\end{align*}
Then the last two terms are controlled by the same reason of \eqref{RHS2ndinequality} by adding the extra weight $\tri_{\tilde\om}\Phi$,
\begin{align*}
RHS.\geq \frac{\tri_{\om}\tilde F}{\Tr_{\om}\tilde\om} -C\Tr_{\tilde\om}\om-C\tri_{\tilde\om}\Phi.
\end{align*}
Using the computation in \cite{MR3488129}, 
\begin{align*}
i\p\bar\p F\geq - C (\om_{\mathbf b_\eps}+i\p\bar \p\Phi)\text{ and } i\p\bar\p\Phi+C\om_0>0,
\end{align*} we have 
\begin{align*}
\tri F&\geq -\Tr_\om[C\om_{\mathbf b_\eps} + C(i\p\bar\p\Phi+C\om_0)]+C^2\Tr_\om\om_0\\
&\geq- C \Tr_{\om}\tilde\om \Tr_{\tilde\om}{(i\p\bar\p\Phi+C\om_0)}- C .
\end{align*}Meanwhile, $\tri\frac{\tilde \vphi_\theta-\vphi_\delta+\vphi_{\mathbf b_\eps}}{\eps}
=\tri\vphi_\theta$.
Using $\Tr_{\tilde\om_{\vphi_\theta}}\om_{0}\leq \Tr_{\tilde\om_{\vphi_\theta}}\om_{\eps,\delta}$ and putting the inequalities above together, we have
\begin{align*}
RHS.\geq \frac{-C}{\Tr_{\om}\tilde\om} -C\Tr_{\tilde\om}\om-C\tilde\tri\Phi+\frac{\tri\vphi_\theta}{ \Tr_{\om}\tilde\om}.
\end{align*}
Assume that $\Tr_{\om}\tilde\om\geq 1$, otherwise we are done. So we have
\begin{align*}
\tilde\tri \log\Tr_\om\tilde\om\geq -C -C\Tr_{\tilde\om}\om-C\tilde\tri\Phi+\frac{\tri\vphi_\theta}{ \Tr_{\om}\tilde\om}.
\end{align*}
Then use the auxiliary function $Z=-(C+1)\tilde\vphi_\theta+C\Phi$ to cancel the middle terms. Since $\tilde\tri\tilde\vphi_\theta=n-\Tr_{\tilde\om}\om$, we have
\begin{align*}
\tilde\tri (\log\Tr_{\om}\tilde\om+Z)\geq  -C(n+1)-C+\frac{\tri\vphi_\theta}{\Tr_{\om}\tilde\om}.
\end{align*}
At the maximum point $p$ of $\log\Tr_{\om}\tilde\om+Z$, it holds
\begin{align*}
\tri\vphi_\theta\leq C\Tr_{\om}\tilde\om .
\end{align*}
Now we use
 \begin{align*}
 \tri\vphi_\theta=\tri\frac{\tilde \vphi_\theta-\vphi_\delta+\vphi_{\mathbf b_\eps}}{\eps}
 \geq\frac{\tri\tilde\vphi_\theta-C}{\eps}
 \end{align*} and $\Tr_{\om}\tilde\om=n+\tri\tilde\vphi_\theta$. Direct computation shows that 
\begin{align*}
\tri\tilde\vphi_\theta(p)\leq \frac{C}{1-\eps C}
\end{align*}
is bounded, as long as $\eps\leq (2C)^{-1}$. At arbitrary point $z$, we have the estimate of $\log\Tr_{\om}\tilde\om(z)$,
\begin{align*}
\log\Tr_{\om}\tilde\om(z)\leq\log\Tr_{\om}\tilde\om(p)+Z(p)-Z(z).
\end{align*}
According to the zero estimates \eqref{zeroupper} and \eqref{zerolower}, the auxiliary function $Z$ is bounded.
Therefore the Laplacian estimate follows.

\end{proof}

\item

When $\eps$ is fixed, recall that $\vphi_\eps$ is the limit of $\vphi_{\theta}$, as $\delta\rightarrow 0$, due to the equicontinuity of $\vphi_{\theta}$. With the help of the zero estimate and the Laplacian estimate of $\eps\vphi_\theta$, together with the corresponding estimates of $\vphi_\delta$, the limit $\vphi_\eps$ also satisfies the limit of \eqref{fibre approximation equation}, as $\delta\rightarrow 0$,
\begin{align}\label{fibre approximation equation limit}
(-\eps \chi+\om_0+i\p\bar\p (\vphi+\eps \vphi_\eps))^n=e^{\frac{\eps\vphi_\eps}{\eps}}\cdot\frac{\om_0^n}{|s|_h^{2-2\b}}.
\end{align}
Moreover, from \eqref{globalpositivity} we get
\begin{align}\label{globalpositivitylimit}
-\chi+\frac{\tilde\Om_\Psi}{\eps}+i\p\bar\p \vphi_{\eps}\geq 0\text{ in }R_0\times X.
\end{align}
Furthermore, we take $\eps\rightarrow 0$. Then the potential $\eps\vphi_\eps$ converges to a limit $\phi$ in $C^{1,\a,\b}$-norm. The volume form $(-\eps \chi+\om_0+i\p\bar\p (\vphi+\eps \vphi_\eps))^n$ converges to $(\om_\vphi+i\p\bar\p\phi)^n$ in $L^p$-norm for some $p$. Moreover, from \eqref{fibre approximation equation limit}, we see that $\phi\leq 0$.

\item The volume form $e^{\vphi_\eps}\frac{\om_0^n}{|s|_h^{2-2\b}}\rightarrow \om_{\vphi}^n$ weakly in $L^p$, as $\eps\rightarrow 0$.
\begin{proof}
According to the convergence of the volume form $(-\eps \chi+\om_0+i\p\bar\p (\vphi+\eps \vphi_\eps))^n$ above, it is sufficient to prove $\phi=0$. So we consider the domain when $\phi<-a$ for small constant $a$. According to the comparison theorem (Theorem 3 in \cite{MR2299485}), 
\begin{align*}
\int_{\phi<-a}(\om_0+i\p\bar\p\vphi)^n\leq \int_{\phi<-a}(\om_0+i\p\bar\p(\vphi+\phi))^n
\end{align*}
which has to be zero, from the approximation equation \eqref{fibre approximation equation limit}. Since $(\om_0+i\p\bar\p\vphi)^n$ is integrable, we could let $a\rightarrow 0$ and obtain $(\om_\vphi+i\p\bar\p\phi)^n=\om_{\vphi}^n$ on the closure of the domain $\phi<0$. Outside this domain, the identity is clear, since $\phi=0$. But, since both $(\om_\vphi+i\p\bar\p\phi)^n$ and $\om_{\vphi}^n$ are $L^p$, the uniqueness theorem in \cite{MR2425147} infers that $\phi$ has to vanish. Thus the proof of the convergence of the approximation volume form is finished.
\end{proof}

\end{enumerate}

\textbf{Step 3:}  
Choose three points $t_p, p=1,2,3$ on the geodesic.
With the fourth property, the Banach-Saks theorem implies that the convex combination $\frac{1}{k}\sum_{j=1}^k e^{\vphi_{\eps_j}(t_p)}\rightarrow e^{\vphi(t_p)}$ in $L^1(\{t_p\}\times X)$, as $k\rightarrow\infty$. We solve a K\"ahler metric $\om_k(t)$ from $\om_k^n(t)=\frac{1}{k}\sum_{j=1}^ke^{\vphi_{\eps_j}(t)} \om^n$ and denote the approximation $K$-energy by
\begin{align*}
\nu_\b^k(t)&=E^k_\b(\vphi)
+\ul{S}_\b\cdot D_{\b}(\vphi)+j_{\b}(\vphi)
\end{align*}
with 
\begin{align*}
E_\b(\vphi)&=\frac{1}{V}\int_M\log\frac{\om^n_k}{\om_0^n|s|_h^{2-2\b}}\om_\vphi^{n}.
\end{align*}

Thanks to \lemref{DJjcontinuity}, the approximation entropy is continuous.
So the approximation log-$K$-energy is continuous on $[0,1]$, by \lemref{DJjcontinuity}.
Noticing that the volume form of the cone geodesic is integrable, since the order along singular direction is $2\b-2-2\kappa=-2+\a\b$. Direct computation shows that \begin{align*}
i\p\bar\p\log\frac{1}{k}\sum_{j=1}^k e^{\vphi_{\eps_j}}\geq i\p\bar\p\vphi_{\eps_j}.
\end{align*}
From \eqref{globalpositivitylimit}, we have $-\chi+\frac{\Om_\Psi}{\eps_j}+i\p\bar\p \vphi_{\eps_j}\geq 0$. 
So using the geodesic equation $\Om_\Psi^{n+1}=0$, we obtain that in the distribution sense \begin{align*}
i\p\bar\p\log\frac{1}{k}\sum_{j=1}^k e^{\vphi_{\eps_j}}\wedge \Om_\Psi^n\geq \chi\wedge \Om_\Psi^n
\end{align*}
Then the approximation log-$K$-energy is convex in the distribution sense along the cone geodesic. Therefore, after taking limit, the log-$K$-energy is convex along $\vphi(t)$.

\end{proof}

Combining the convexity of the log-$J$-functional (Proposition \ref{Jconvex}) and the convexity of the log-$K$-energy (Proposition \ref{convexityKenergy}), we have proved the convexity of the log-$\mathcal E$-energy.
\begin{prop}\label{Econvex}
	Suppose $\vphi(t)$ is a $C^{1,1,\b}_{\mathbf w}$ cone geodesic. The log-$\mathcal E$-energy is strictly convex along the cone geodesic and the $J$-twisted cscK cone metric is unique.
\end{prop}




\section{Regularity of cscK cone metrics}\label{Asymptotic behaviours}
We prove the regularity of cscK cone metrics in this section. The main theorem is \thmref{Higher order estimates}, as well as, Proposition \ref{lem:goodmetric} and Proposition \ref{lem:goodtangent}. 
\subsection{H\"older spaces $C_{\mathbf {pw}}^{3,\a,\b}$ and $C_{\mathbf {w}}^{3,\a,\b}$}\label{3rdHolderspaces}
Before we prove the regularity theorem, we introduce new function spaces, which are motivated from the regularity of cscK cone metrics.

\begin{defn}\label{cpw3ab}
	A function $\vphi$ belongs to $C_{\mathbf {pw}}^{3,\a,\b}$ with the H\"older exponent $\alpha$ satisfying	
	$
	\a\b<1-\b,
	$
	if it holds
	\begin{itemize}
		\item $\vphi\in C^{2,\a,\b}$;
		\item the first derivatives of the corresponding metric $g$ satisfy for any $2\leq i,k,l\leq n$, the following items are $C^{0,\a,\b}$,	
		\begin{align*}
		\frac{\p g_{ k\bar l}}{\p z^i},\quad
		|z^1|^{1-\b}\frac{\p g_{k\bar{1}}}{\p z^{ i}},\quad
		|z^1|^{1-\b}\frac{\p g_{ 1\bar l}}{\p z^i},\quad
		|z^1|^{2-2\b}\frac{\p g_{1\bar 1}}{\p z^i}.
		\end{align*}
		
	\end{itemize}
	When the coordinate chart does not intersect the divisor, all definitions are in the classical way.
\end{defn}

We define another space $C_{\mathbf w}^{3,\a,\b}$, with information along the singular directions.

\begin{defn}\label{cw3ab}
	A function $\vphi$ is in $C_{\mathbf w}^{3,\a,\b}$ with the H\"older exponent $\alpha$ satisfying $\a\b<1-\b,$
	if $\vphi\in C_{\mathbf {pw}}^{3,\a,\b}$, and additionally in the normal cone chart, the following terms are $O(|z^1|^{\a\b-\b})$,	
	\begin{align*}
	|z^1|^{1-\b}\frac{\p g_{ k\bar l}}{\p z^1},\qquad
	|z^1|^{2-2\b}\frac{\p g_{k\bar{1}}}{\p z^{  1}},\quad
	|z^1|^{2-2\b}\frac{\p g_{1\bar l}}{\p z^1},\qquad
	|z^1|^{3-3\b}\frac{\p g_{1\bar 1}}{\p z^1} .
	\end{align*}

\end{defn}

The spaces above could be generalised to the product manifold $\mathfrak X$, adding the now direction $z^{n+1}$ as the regular direction $2\leq j\leq n$, i.e. changing $n$ to $n+1$ in the definition. 
We list the terms $|\nabla_i^{cone}(g_{\Psi_{\mathbf b}})_{a\bar b}|_{\Om_{cone}}$ computed in the next section.
\begin{cor}\label{psibcondition}
	Suppose $\Psi\in C_{\mathbf {pw}}^{3,\a,\b}$. Then, for all $2\leq i\leq n+1$ and $1\leq a,b\leq n+1$, it holds $|\nabla_i^{cone}(g_{\Psi})_{a\bar b}|_{\Om_{cone}}$ are in $C^{0,\a,\b}$.	
\end{cor}
\begin{proof}
	When $2\leq i,k,l\leq n+1$, we have the formulas
	\begin{align*}
	\left\{
	\begin{array}{ll}
	&\nabla_i^{cone}(g_{\Psi})_{k\bar l}
	=\frac{\p (g_{\Psi})_{ k\bar l}}{\p z^i},\quad
	\nabla_i^{cone}(g_{\Psi})_{1\bar l}=\frac{\p (g_{\Psi})_{ 1\bar l}}{\p z^i},\quad\\
	&\nabla_i^{cone}(g_{\Psi})_{1\bar 1}=\frac{\p (g_{\Psi})_{1\bar 1}}{\p z^i},\quad \nabla_i^{cone}(g_{\Psi})_{k\bar 1}
	=\frac{\p (g_{\Psi})_{ k\bar 1}}{\p z^i}.
	\end{array}
	\right.
	\end{align*}
	Then the conclusions follow from Corollary \ref{dmetric}.
\end{proof}

\begin{cor}\label{pwpsibcondition}
	Suppose $\Psi\in C_{\mathbf w}^{3,\a,\b}$. Then for all $2\leq k,l\leq n+1$, 
	the following terms are $O(|\rho_0|^{\a-1})$,	
	\begin{align*}
	\left\{
	\begin{array}{ll}
	&|\nabla_1^{cone}(g_{\Psi})_{k\bar l}|_{\Om_{cone}},\quad
	|\nabla_1^{cone}(g_{\Psi})_{k\bar 1}|_{\Om_{cone}},\\
	&|\nabla_1^{cone}(g_{\Psi})_{1\bar l}|_{\Om_{cone}},\quad
	|\nabla_1^{cone}(g_{\Psi})_{1\bar 1}|_{\Om_{cone}}.
	\end{array}
	\right.
	\end{align*}
\end{cor}
\begin{proof}
	The computation is direct, as Corollary \ref{dmetric}.
\end{proof}

\subsection{Christoffel symbols of cscK cone metrics}\label{Christoffel symbols of cscK cone metric}
We apply Proposition \ref{lem:goodmetric} and Proposition \ref{lem:goodtangent} to  the following asymptotic behaviours of the cscK cone metric $g$ and its Christoffel symbols for further application.


We are given a point $z$ outside the divisor and in the cone chart we denote $$\rho_0=|z^1|^\b.$$

\begin{cor}\label{dmetric}
	Let $\varphi$ be a $C^{2,\a,\b}\cap\mathcal H_{\b} $ solution of \eqref{eqn:main} of the cscK cone metric (or $J$-twisted cscK cone metric). Then $\vphi\in C_{\mathbf w}^{3,\a,\b}\cap\mathcal H_{\b} $ i.e. the first derivatives of the corresponding metric $g$ satisfy for any $2\leq i,k,l\leq n$,	
	\begin{align}\label{bgdmetric}
	\left\{
	\begin{aligned}
	&\frac{\p g_{ k\bar l}}{\p z^i}=C^{0,\a,\b},\quad
	|z^1|^{1-\b}\frac{\p g_{ k\bar 1}}{\p z^i}=C^{0,\a,\b},\quad\\
	&|z^1|^{1-\b}\frac{\p g_{ 1\bar l}}{\p z^i}=C^{0,\a,\b},\quad
	|z^1|^{2-2\b}\frac{\p g_{1\bar 1}}{\p z^i}=C^{0,\a,\b},\quad\\
	&|z^1|^{1-\b}\frac{\p g_{ k\bar l}}{\p z^1}=O(\rho_0^{\a-1}),\quad
	|z^1|^{2-2\b}\frac{\p g_{k\bar{1}}}{\p z^{  1}}=O(\rho_0^{\a-1}),\\
	&|z^1|^{2-2\b}\frac{\p g_{1\bar l}}{\p z^1}
	=O(\rho_0^{\a-1}),\quad
	|z^1|^{3-3\b}\frac{\p g_{1\bar 1}}{\p z^1} 
	=O(\rho_0^{\a-1}).
	\end{aligned}
	\right.
	\end{align}
	Moreover,	
	\begin{align}\label{bgdmetricbad}
	\left\{
	\begin{aligned}
	&|z^1|^{1-\b}\nabla_1^{cone}g_{k\bar l}
	=O(\rho_0^{\a-1}),\\
	&|z^1|^{2-2\b}\nabla_1^{cone}g_{k\bar 1}
	=O(\rho_0^{\a-1}),\\
	&|z^1|^{2-2\b}\nabla_1^{cone}g_{1\bar l}
	=O(\rho_0^{\a-1}),\\
	&|z^1|^{3-3\b}\nabla_1^{cone}g_{1\bar 1}
	=O(\rho_0^{\a-1}).
	\end{aligned}
	\right.
	\end{align}
\end{cor}	
\begin{proof}
	Using Proposition \ref{lem:goodmetric}, we have in $V$-coordinate, the metric $g$ satisfy $\norm{g_{i\bar{j}}}_{C^{1,\a}_{\tilde{V}}(B_{3c_\beta/4})} \leq C(1)$. We choose the normal coordinate such that $g$ is the flat cone metric and $\p_{\tilde v}g_{k\bar l}=0$ at $\tilde v=0$.
	

Then we have
	\begin{align*}
	\left\{
	\begin{aligned}
	&\p_{v^i}g_{k\bar l}=C^{0,\a,\b},\quad
	\p_{v^i} g_{k\bar 1}=C^{0,\a,\b},\quad\\
	&	\p_{v^i} g_{1\bar l}=C^{0,\a,\b},\quad
	\p_{v^i} g_{1\bar 1}=C^{0,\a,\b},\quad\\
	&	\p_{v^{1}} g_{k\bar l}
	=O(\rho_0^{\a-1}),\quad
	\p_{v^{1}} g_{k\bar 1}
	=O(\rho_0^{\a-1}),\\
	&\p_{v^1} g_{1\bar l}
	=O(\rho_0^{\a-1}),\quad
	\p_{v^1} g_{1\bar 1}
	=O(\rho_0^{\a-1}),
	\end{aligned}
	\right.
	\end{align*}
	then transforming back to the $z$-coordinate with $v^1=(z^1)^\b$ and $\tilde g_{1\bar l}=\b^{-1}(z^1)^{1-\b}g_{1\bar l}$, $\tilde g_{1\bar 1}=\b^{-2}|z^1|^{2-2\b}g_{1\bar 1}$, we obtain \eqref{bgdmetric}.

We use the standard formula,
	\begin{align*}
	\Gamma^{k}_{ij}(\om_{cone})=g_{cone}^{k\bar{l}}\frac{\partial (g_{cone})_{i\bar{l}}}{\partial
		z^{j}}
	\end{align*}
	and compute
	\begin{align*}
	\frac{\p}{\p z^1}|z^1|^{2\b-2}=(\b-1)|z^1|^{2\b-4} z^{\bar 1}
	\end{align*}
	to see that the Christoffel symbols of $\om_{cone}$ in a cone chart are, for all $2\leq i,j,k\leq n$,
	\begin{align*}
	\Gamma^1_{1k}(\om_{cone})=
	\Gamma^i_{11}(\om_{cone})=
	\Gamma^1_{jk}(\om_{cone})=
	\Gamma^i_{1k}(\om_{cone})=
	\Gamma^i_{jk}(\om_{cone})=0.
	\end{align*} 
	The only exception is
	\begin{align*}
	\Gamma^1_{11}(\om_{cone})=-\frac{1-\b}{z^1}.
	\end{align*} 
	
	Then we compute $\nabla_1^{cone}g_{1\bar l}$ and $\nabla_1^{cone}g_{1\bar 1}$. We use the formulas of the first covariant derivative w.r.t $\om_{cone}$ 
	\begin{align*}
	\nabla^{cone}_1 \vphi_1
	&=\frac{\p \vphi_1}{\p z^1}-\sum_{p=1}^n \Gamma_{11}^p \vphi_p
	=\frac{\p \vphi_1}{\p z^1}+\frac{1-\b}{z^1}\vphi_1.
	\end{align*}
	For a $(1,1)$-form $a_{i\bar j}dz^1\wedge dz^{\bar j}$, the first covariant derivative w.r.t $\om_{cone}$ is
	\begin{align*}
	\nabla^{cone}_1 a_{i\bar j}
	&=\frac{\p a_{i\bar j}}{\p z^1}-\sum_{p=1}^n \Gamma_{i1}^p a_{p\bar j}-\sum_{p=1}^n \Gamma_{1\bar j}^{\bar p} a_{i\bar p}
	=\frac{\p a_{i\bar j}}{\p z^1}+\frac{1-\b}{z^1}a_{1\bar j}.
	\end{align*}
	Thus we have
	\begin{align*}
	&|z^1|^{1-\b}\nabla_1^{cone}g_{k\bar l}
	=|z^1|^{1-\b}\frac{\p g_{k\bar l}}{\p z^1},\\
	&|z^1|^{2-2\b}\nabla_1^{cone}g_{k\bar 1}
	:=|z^1|^{2-2\b}[\frac{\p g_{k\bar 1}}{\p z^1}+\frac{1-\b}{z^1}g_{k\bar 1}],\\
	&|z^1|^{2-2\b}\nabla_1^{cone}g_{1\bar l}
	:=|z^1|^{2-2\b}[\frac{\p g_{1\bar l}}{\p z^1}+\frac{1-\b}{z^1}g_{1\bar l}].
	\end{align*}
	The first three identities in the conclusion \eqref{bgdmetricbad} follows from \eqref{bgdmetric} and the following estimates, 	
	\begin{align*}
	&|z^1|^{2-2\b}\frac{1-\b}{z^1}g_{k\bar 1}=O(\rho_0^{\a-1}),\\
	&|z^1|^{2-2\b}\frac{1-\b}{z^1}g_{1\bar l}=O(\rho_0^{\a-1}),
	\end{align*}
	since $g_{k\bar 1},g_{1\bar l}\rightarrow 0$, as points converges to the point on the divisor.
	The last identity in the conclusion is proved by using $\nabla_1^{cone}(g_{cone})_{1\bar 1}=0$, \eqref{bgdmetric} and 
	\begin{align*}
|z^1|^{3-3\b}\frac{1-\b}{z^1}(g_{1\bar 1}-(g_{cone})_{1\bar 1})=O(\rho_0^{\a-1}).
	\end{align*} Thus we have
	\begin{align*}
	&|z^1|^{3-3\b}\nabla_1^{cone}g_{1\bar 1}\\
	&:=|z^1|^{3-3\b}[\frac{\p g_{1\bar 1}-(g_{cone})_{1\bar 1}}{\p z^1}+\frac{1-\b}{z^1}(g_{1\bar 1}-(g_{cone})_{1\bar 1})]\\
	&=O(\rho_0^{\a-1})
	\end{align*}
	and complete the proof.

\end{proof}

We determine the Christoffel symbols of the cscK cone metric. 

\begin{cor}\label{connection}
	Let $\varphi$ be a $C^{2,\a,\b}$ solution of \eqref{eqn:main} of the cscK cone metric (or $J$-twisted cscK cone metric). Then the Christoffel symbols of the corresponding metric $g$ satisfy for any $2\leq i,j,k\leq n$,
	\begin{align}\label{bgconnection}
	\left\{
	\begin{aligned}
	&\Gamma^{i}_{jk}=C^{0,\a,\b}, \quad ,\quad 
	|z^1|^{1-\b}\Gamma^{1}_{11}=O(\rho_0^{\a-1}),\\
	&|z^1|^{\b-1}\Gamma^{1}_{jk}=C^{0,\a,\b},\quad
	|z^1|^{1-\b}\Gamma^{i}_{1k}(\om)=C^{0,\a,\b},\quad
	|z^1|^{1-\b}\Gamma^{i}_{j1}=O(\rho_0^{\a-1}), \\
	&\Gamma^{1}_{j1}=O(\rho_0^{\a-1}), \quad
	\Gamma^{1}_{1k}=O(\rho_0^{\a-1}),\quad
	|z^1|^{2-2\b}\Gamma^{i}_{11}=O(\rho_0^{\a-1}).
	\end{aligned}
	\right.
	\end{align} 
\end{cor}
\begin{proof} 
	The connection of $\om_\vphi$ is given by the following formulas and check one by one by using \corref{dmetric} as
	\begin{align*}
	&\Gamma^{i}_{jk}(\om)
	=\sum_{1\leq l\leq n}g^{i\bar{l}}\frac{\partial g_{j\bar{l}}}{\partial
		z^{k}}
	=\sum_{2\leq l\leq n}g^{i\bar{l}}\frac{\partial g_{j\bar{l}}}{\partial
		z^{k}}+g^{i\bar{1}}\frac{\partial g_{j\bar{1}}}{\partial
		z^{k}}.
	\end{align*} 
	Then
	\begin{align*}
		&|z^1|^{1-\b}\Gamma^{i}_{1k}(\om)
	=|z^1|^{1-\b}[\sum_{2\leq l\leq n}g^{i\bar{l}}\frac{\partial g_{1\bar{l}}}{\partial
		z^{k}}+g^{i\bar{1}}\frac{\partial g_{1\bar{1}}}{\partial
		z^{k}}],\\
	&|z^1|^{1-\b}\Gamma^{i}_{j1}(\om)
	=|z^1|^{1-\b}[\sum_{2\leq l\leq n}g^{i\bar{l}}\frac{\partial g_{j\bar{l}}}{\partial
		z^{1}}+g^{i\bar{1}}\frac{\partial g_{j\bar{1}}}{\partial
		z^{1}}].
	\end{align*} 
	Also
	\begin{align*}
	|z^1|^{\b-1}\Gamma^{1}_{jk}(\om)
	=|z^1|^{\b-1}[\sum_{2\leq l\leq n}g^{1\bar{l}}\frac{\partial g_{j\bar{l}}}{\partial
		z^{k}}+g^{1\bar{1}}\frac{\partial g_{j\bar{1}}}{\partial
		z^{k}}].
	\end{align*} 
	Furthermore, 
	\begin{align*}
	&\Gamma^{1}_{j1}(\om)
	=\sum_{2\leq l\leq n}g^{1\bar{l}}\frac{\partial g_{j\bar{l}}}{\partial
		z^{1}}+g^{1\bar{1}}\frac{\partial g_{j\bar{1}}}{\partial
		z^{1}},\\
	&\Gamma^{1}_{1k}(\om)
	=\sum_{2\leq l\leq n}g^{1\bar{l}}\frac{\partial g_{1\bar{l}}}{\partial
		z^{k}}+g^{1\bar{l}}\frac{\partial g_{1\bar{l}}}{\partial
		z^{k}}.
	\end{align*} 
	The next two terms involve $\frac{\partial g_{1\bar{l}}}{\partial
		z^{1}}$ and $\frac{\partial g_{1\bar{1}}}{\partial
		z^{1}}$,
	\begin{align*}
	|z^1|^{2-2\b}\Gamma^{i}_{11}(\om)=|z^1|^{2-2\b}[\sum_{2\leq l\leq n}g^{i\bar{l}}\frac{\partial g_{1\bar{l}}}{\partial
		z^{1}}
	+g^{i\bar{1}}\frac{\partial g_{1\bar{1}}}{\partial
		z^{1}}]
	\end{align*} 
	and
	\begin{align*}
	&|z^1|^{1-\b}\Gamma^{1}_{11}(\om)
	=|z^1|^{1-\b}[\sum_{2\leq l\leq n}g^{1\bar{l}}\frac{\partial g_{1\bar{l}}}{\partial
		z^{1}}+g^{1\bar{1}}\frac{\partial g_{1\bar{1}}}{\partial
		z^{1}}].
	\end{align*} 
	Thus the conclusion is true.

\end{proof}




\subsection{Higher order estimates}\label{Higher order estimates}
We use the same terminology from \cite{arxiv:1609.03111}.
We consider asymptotic for cscK cone metrics in the unit ball $B_1\subset B^\ast:=\mathbb C \times \mathbb C^{n-1}$ centred at the origin,
\begin{align}\label{eqn:main}
\left\{
\begin{aligned}
\det(\vphi_{i\bar j})&=\frac{e^{K}}{|z_1|^{2-2\b}},\\
\tri_\vphi K&=S.
\end{aligned}
\right.
\end{align}
Assume that  
\begin{itemize}
	\item $S$ is a smooth function;
	\item $\varphi$ is in $C^{2,\a,\b}$ space , then from the second equation, and $K$ is also in $C^{2,\a,\b}$ space;
	\item there is some constant $c>1$ such that
	\begin{align}\label{metricequivalent}
	\frac{1}{c}\omega_{cone}\leq \sqrt{-1}\partial\bar{\partial} \varphi \leq c\omega_{cone}.
	\end{align}
\end{itemize}

We also consider the twisted cscK cone metrics with $$S=\gamma(\vphi_{i\bar j})=\underline S_\b+(1-t)(\frac{\om_0^n}{\om_\vphi^n}-1).$$
We put the necessary change of the proof of this case in the remarks after each subsections for readers' convenience.

Given a point $$Z_0=(r_0,\theta_0,\xi_0)\in B_1\setminus \mathbf 0,$$ we consider a neighbourhood $\Omega$ of $Z_0$ which is contained in the regular part $M$. Note that $\Omega$ does not contain $\mathbf 0$.

\begin{defn}[Polar coordinates]
We use the polar coordinates $(\rho,\theta,\xi)$ in $B_1$, defined by
\begin{align*}
\rho=\abs{z_1}^\beta, \quad z_1= \abs{z_1}e^{i\theta}, \quad \xi=(z_2,\cdots,z_n).
\end{align*} 
\end{defn}

With this terminology in place, our result on the higher order estimates of the cscK cone metric and $J$-twisted cscK cone metric can be stated as follows.
\begin{thm}\label{thm:interior}
	Let $\varphi$ be a $C^{2,\a,\b}(B_1)$ solution of \eqref{eqn:main}. Then given any $k_1,k_2,k_3 \in \mathbb N\cup \set{0}$, there exists a constant $C(k_1,k_2,k_3)$ such that
	\begin{align}\label{eqn:interior}
	\abs{(\rho \partial_\rho)^{k_1} (\partial_\theta)^{k_2} (\nabla_\xi)^{k_3} \varphi}\leq C (k_1,k_2,k_3),, 
	\end{align}
	for any $\rho\in (0,1/2)$ and $\abs{\xi}<1/2$.
\end{thm}

We recall the scaling coordinate systems. 


\begin{defn}[Lifted holomorphic coordinates]\label{Lifted holomorphic coordinates}	For points in $\Omega$, we write
	\begin{align*}
 v_1:=z_1^\beta
	\end{align*}
and define for $z\in \Omega$,	\begin{align*}
	LH: z\mapsto v=(v_1,\xi)=(v_1,\cdots,v_n)\in \mathbb C^n.
	\end{align*} Then the centre of the domain $Z_0$ becomes $$LH(z_0)=v_0=(\rho_0 e^{i\beta\theta_0}, \xi_0)$$ and the imagine $LH(\Omega)$ is contained in $$B_{c_\beta\rho_0}(v_0)\subset \mathbb C^n\setminus\mathbf 0.$$

We denote the partial derivatives by $$\varphi_{V,i},\quad \varphi_{V,\bar{j}},\quad \varphi_{V, i\bar{j}},\quad etc$$ and tangent derivatives by \begin{align*}
\nabla_{V,\xi}\in \{ \partial_{v_i},\quad \partial_{\bar{v}_i},\quad \forall i=2,\cdots,n\}.
\end{align*}
	\end{defn}
	
	\begin{defn}[Rescaled lifted holomorphic coordinates]\label{Rescaled lifted holomorphic coordinates}
 Now we define the rescaled lifted holomorphic coordinates $\tilde{v}$ by 
	\begin{align*}
	SLH: v\mapsto \tilde{v}=\frac{v-v_0}{\rho_0}.
	\end{align*}
	Then after scaling by $SLH$, the domain $\Omega$ is contained in $$B_{c_\beta}:=B_{c_\beta}(0)\subset \mathbb C^n$$ for a universal constant $c_\b$, which is a ball in the $\tilde v$-coordinates.
	
	 We now denote the partial derivatives by $$\varphi_{\tilde{V},i},\quad \varphi_{\tilde{V},\bar{j}},\quad \varphi_{\tilde{V},i\bar{j}},\quad etc$$ and the tangent derivatives, as before, by $$\nabla_{\tilde{V},\xi}.$$
\end{defn}

The rescaled lifted holomorphic coordinates are the right coordinates to carry on the interior Schauder estimates and make sure the constants appearing from the estimates are independent of the position of the point $Z_0$. 

The H\"older functions are defined in the usual sense with respect to the distance given by the $v$-coordinates. The norm $$\norm{\cdot}_{C^{k,\alpha}_V(B_{c_\beta \rho_0}(v_0))}$$ is the usual H\"older norm. Here we use the lower subscript to emphasise the use of the $V$-coordinates around $v_0$. Similar convention hold for $C^{k,\alpha}_{\tilde{V}}$.

We always require the constant $C(k)$ increases on $k$ and may be different from line to line.

%
\subsubsection{Fundamental estimates: startup}
From Lemma 2.3 in \cite{arxiv:1609.03111}, the following estimates hold by coordinate transformation, since $\vphi$ and $K$ are both in $C^{2,\a,\b}(B_1)$.
\begin{lem} \label{2abinrealcord}
	Suppose $u\in C^{2,\a,\b}(B_1)$. 
There is a constant $C>0$ independent of $Z_0$ and $u$ such that
	\begin{align*}
	\norm{u_{V}}_{C^{0,\a}_V(B_{c_\beta \rho_0}(v_0))}+
	\norm{u_{V,i}}_{C^{0,\a}_V(B_{c_\beta \rho_0}(v_0))}+\norm{u_{V,i\bar{j}}}_{C^{0,\a}_V(B_{c_\beta \rho_0}(v_0))}\leq C \norm{u}_{C^{2,\a,\b}(B_1)}.
	\end{align*}
\end{lem}

All tangent derivatives of $\varphi$ and $K$ are bounded in $C^{2,\a,\b}$ space, because along tangent direction, no singularities occur.
\begin{lem} \label{Tangent derivatives}
	For $l=0,1,2,\cdots$, there exists a constant $C(l)$ such that
	\begin{align}\label{eqn:Xigc0}
	&\norm{ (\nabla_{V,\xi}^l \vphi)_{V,i\bar{j}}}_{C_V^\alpha(B_{c_\beta \rho_0}(v_0))}\leq C(l);\\
	&\norm{ (\nabla_{V,\xi}^l K)_{V,i\bar{j}}}_{C_V^\alpha(B_{c_\beta \rho_0}(v_0))}\leq C(l).
	\end{align}
\end{lem}
%

%
\subsubsection{Rescaled estimates on $g$ and $K$}
We rescale the K\"ahler potential, 
\begin{align}\label{eqn:setting}
\tilde{\varphi}(\tilde{v})=\rho_0^{-2} \varphi(v_0+\tilde{v}\rho_0). 
\end{align}
Then the rescaled metric $$g_{i\bar{j}}:=\tilde{\varphi}_{\tilde{V},i\bar{j}}, \quad \text{for} \, i,j=1,\cdots,n$$
is equivalent to the standard Euclidean metric, and also the $C_{\tilde{V}}^{0,\alpha}(B_{c_\beta})$ norm is bounded,
\begin{align}\label{eqn:gcalpha}
\norm{g_{i\bar{j}}}_{C_{\tilde{V}}^{0,\alpha}(B_{c_\beta})}\leq C
\end{align}

We will prove the higher order interior estimates of both $g_{i\bar j}$ and $K$. We denote the strictly increasing sequence $\eta_k$ for $k=0,1,2,\cdots$ such that $\lim_{k\rightarrow\infty}\eta_k=0.25 c_\beta$. This is a technical arrangement to shrink the radius of balls in each interior estimates.

Recall the first equation in \eqref{eqn:main}, i.e. $$\det(\vphi_{i\bar j})=\frac{e^{K}}{|z_1|^{2-2\b}}.$$ Rewriting it in the rescaled lifted holomorphic coordinates $\tilde{v}$, we have
	\begin{align}\label{1stequationscale}
	\log \det(\tilde \vphi_{\tilde V, i\bar j})=K.
	\end{align} 
	Taking $\partial_{\tilde{v}_i}\partial_{\bar{\tilde{v}}_j}$, we get
	\begin{align}\label{eqn:gij}
	\triangle_g g_{i\bar{j}} - g^{k\bar{m}} g^{n\bar{l}} \pfrac{g_{i\bar{m}}}{\tilde{v}_n} \pfrac{g_{k\bar{j}}}{\bar{\tilde{v}}_{{l}}} =  K_{\tilde{V},i\bar{j}},
	\end{align}
	where $\triangle_g = g^{i\bar{j}}\frac{\partial^2}{\partial \tilde{v}_i \partial \bar{\tilde{v}}_j}$.

\begin{prop}\label{lem:goodmetric}
	There are constants $C(k)$ such that for $k=0,1,2,\cdots$,
	\begin{align*}
	\norm{g_{i\bar{j}}}_{C^{k,\a}_{\tilde{V}}(B_{3c_\beta/4})} \leq C(k);\quad
	\norm{K}_{C^{k,\a}_{\tilde{V}}(B_{3c_\beta/4})}\leq C(k).
	\end{align*} 
\end{prop}
\begin{proof}
		From Lemma \ref{2abinrealcord}, the RHS of \eqref{eqn:gij} is bounded by
	\begin{align*}
	\norm{K_{\tilde{V},i\bar{j}}}_{C^{0,\a}_{\tilde V}(B_{c_\beta})} \leq C \norm{K}_{C^{2,\a,\b}(B_1)}\cdot \rho^{2+\a}_0.
	\end{align*}	
	
Then using the same technique in \cite{arxiv:1609.03111} (see the appendix therein) to obtain the interior estimate of \eqref{eqn:gij}, we obtain that there exists some $\alpha'>0$ such that
	\begin{align}\label{c1alpha}
	\norm{g_{i\bar{j}}}_{C_{\tilde{V}}^{1,\alpha'}(B_{c_\beta-\eta_1})}\leq C(1)
	\end{align}
	for some constant $C(1)$ depending on $\eta_1$, $\alpha'$, $c$ in \eqref{metricequivalent} and the constant in \eqref{eqn:gcalpha}.
	
	Applying the classical Schauder estimate to \eqref{eqn:gij}, we get that 
	there is a constant $C(2)$ independent of $Z_0$ such that
	\begin{align}\label{c2alpha}
	\norm{g_{i\bar{j}}}_{C^{2,\a}_{\tilde{V}}(B_{c_\beta-\eta_2})} \leq C(2).
	\end{align}
	
	Under the rescaled lifted holomorphic coordinates $\tilde{v}$, the second equation in \eqref{eqn:main} becomes 
	\begin{align}\label{2ndequationscale}
	S=\tri_g K=g^{i\bar{j}}K_{\tilde{V},i\bar{j}}.
	\end{align}
	Since $S$ is a smooth function, 
	\begin{align}\label{eqn:goodh}
	\norm{S}_{C_{\tilde{V}}^{k,\a}(B_{c_\beta})}\leq C(k), \quad \forall k\geq 0.
	\end{align}
	Then applying the Schauder estimate to \eqref{2ndequationscale} with \eqref{c2alpha}, we have the 4th order estimates of $K$, 
	\begin{align*}
	\norm{K}_{C^{4,\alpha}_{\tilde{V}}(B_{c_\beta-\eta_3})}\leq C(3).
	\end{align*}
	
	Returning to \eqref{eqn:gij}, the 3rd and also the 4th order estimates of $g_{i\bar{j}}$ are also obtained by bootstrapping method,
	\begin{align}\label{c4alpha}
	\norm{g_{i\bar{j}}}_{C^{4,\a}_{\tilde{V}}(B_{c_\beta-\eta_4})} \leq C(4).
	\end{align} 
	Repeating the argument above, we arrive at the estimates of $g_{i\bar{j}}$ and $K$.
\end{proof}

\begin{rem}(The twisted case $S=\gamma(\vphi_{i\bar j})$.)
	We only need to change \eqref{eqn:goodh} to be
	\begin{align*}
	\norm{S}_{C_{\tilde{V}}^{2,\a}(B_{c_\beta})}\leq C(2),
	\end{align*}
	that follows from \eqref{c2alpha}.
\end{rem}
%
\subsubsection{Refined rescaled tangent estimates on $g$ and $K$}
We rewrite Lemma \ref{Tangent derivatives} in the rescaled lifted holomorphic coordinates. The tangent estimates on metric $g_{i\bar j}$ are rescaled as 
\begin{align}\label{eqn:xigc0}
\norm{ \nabla_{\tilde{V},\xi}^l g_{i\bar{j}}}_{C_{\tilde{V}}^{0,\alpha}(B_{c_\beta})}\leq C(l) \cdot \rho_0^{l}, \forall l=0,1,2,\cdots.
\end{align}
We will further improve the tangent estimates from $C_{\tilde{V}}^{0,\alpha}$ to $C^{k,\a}_{\tilde{V}}$ in shrinking balls.

We take tangent derivatives $\nabla_{\tilde{V},\xi}$ of \eqref{eqn:gij} and \eqref{2ndequationscale} with respect to the rescaled coordinates, 
	\begin{align}
	\label{eqn:extra}
	& \triangle_g (\nabla_{\tilde{V},\xi} g_{i\bar{j}})-(\nabla_{\tilde{V},\xi} K)_{\tilde{V},i\bar{j}}\\ &= G_1(g,\nabla_{\tilde{V}} g, \nabla^2_{\tilde{V}} g) \# \nabla_{\tilde{V},\xi} g + G_2(g,\nabla_{\tilde{V}} g) \# \nabla_{\tilde{V}} (\nabla_{\tilde{V},\xi} g).\nonumber\\
	\label{2ndequationscaletangent}
	&g^{i\bar{j}}(\nabla_{\tilde{V},\xi} K)_{\tilde{V},i\bar{j}}=\nabla_{\tilde{V},\xi} S-\nabla_{\tilde{V},\xi} g^{i\bar{j}}K_{\tilde{V},i\bar{j}},
	\end{align}

In order to obtain the interior higher order estimates, it remains to examine the coefficients and the non-homogeneous terms. 

\begin{prop}\label{lem:goodtangent}For any $k,l =0,1,2,\cdots$, there are constants $C(k,l)$ such that
	\begin{align}\label{eqn:goodeta}
	&\norm{ (\nabla_{\tilde{V},\xi})^l g_{i\bar{j}}}_{C^{k,\a}_{\tilde{V}}(B_{3c_\beta/4-\eta_l})} \leq C(k,l)\rho_0^l,\\
	\label{eqn:goodetaK}&\norm{ (\nabla_{\tilde{V},\xi})^l K}_{C^{k,\a}_{\tilde{V}}(B_{3c_\beta/4-\eta_l})} \leq C(k,l)\rho_0^l.
	\end{align}
\end{prop}
\begin{proof}
	The proof is by induction. We first prove the lemma with $l=1$. 	Applying the Schauder estimate to \eqref{2ndequationscaletangent}, we derive the tangent estimates on $\nabla_{\tilde{V},\xi} K$, i.e.
	\begin{align}\label{eqn:schauderK}
	\norm{\nabla_{\tilde{V},\xi} K}_{C^{k+2,\alpha}_{\tilde{V}}(B_{(3c_\beta/4-\eta_1)})}\leq C \left(  \norm{\nabla_{\tilde{V},\xi} K}_{C^{0}(B_{c_\beta})} + \norm{\nabla_{\tilde{V},\xi} S}_{C^{k,\alpha}_{\tilde{V}}(B_{c_\beta})}\right).
	\end{align}
	Since $S$ is smooth, 
	\begin{align}\label{eqn:h1}
	\norm{\nabla_{\tilde{V},\xi} S}_{C^{k,\alpha}_{\tilde{V}}(B_{c_\beta})}\leq C\rho_0.
	\end{align}
	From \lemref{2abinrealcord}, 
	\begin{align}\label{eqn:K1}
	\norm{\nabla_{\tilde{V},\xi} K}_{C^{0}_{\tilde{V}}(B_{c_\beta})}\leq C\rho_0.
	\end{align}
	Putting them into the RHS. of \eqref{eqn:schauderK}, we obtain the second inequity \eqref{eqn:goodetaK} with $l=1$ in the conclusion,
	\begin{align}\label{eqn:schauderK}
	\norm{ \nabla_{\tilde{V},\xi} K}_{C^{k+2,\a}_{\tilde{V}}(B_{3c_\beta/4-\eta_1})} \leq C(k,1)\rho_0,
	\end{align}

	We apply the Schauder estimate  to \eqref{eqn:extra} with coefficients bounded from Proposition \ref{lem:goodmetric} to obtain 
	\begin{align}\label{eqn:schauder}
	\norm{\nabla_{\tilde{V},\xi} g}_{C^{k+2,\alpha}_{\tilde{V}}(B_{(3c_\beta/4-\eta_1)})}\leq C \left(  \norm{\nabla_{\tilde{V},\xi} g}_{C^{0}(B_{c_\beta})} + \norm{\nabla_{\tilde{V},\xi} K}_{C^{k+2,\alpha}_{\tilde{V}}(B_{c_\beta})}\right).
	\end{align}
	The RHS. of \eqref{eqn:schauder} is bounded, because of \eqref{eqn:xigc0} and \eqref{eqn:schauderK}, we prove the first inequality \eqref{eqn:goodeta} with $l=1$, i.e.
	\begin{align}
	\norm{ \nabla_{\tilde{V},\xi} g_{i\bar{j}}}_{C^{k+2,\a}_{\tilde{V}}(B_{3c_\beta/4-\eta_1})} \leq C(k,l)\rho_0.
	\end{align}
	
	If \eqref{eqn:goodeta} holds for $l=l_0$, we take $l_0$ times tangent derivatives on both $g$-equation \eqref{eqn:extra} and 
	$K$-equation \eqref{2ndequationscaletangent}, and the argument is similar as above. We leave it to interested readers.
\end{proof}

\begin{rem}
	We notice that \eqref{eqn:h1} follows from \eqref{eqn:xigc0}.
\end{rem}
\subsubsection{Final step of the proof of higher order estimates (\thmref{thm:interior})}
Now we are ready to the final step to prove the higher order interior estimates for the cscK cone potential $\vphi$. We will show the proof by induction on $k_3$ in the estimate \eqref{eqn:interior}.

Recall the definition of $\tilde\vphi$, i.e. $\tilde{\varphi}(\tilde{v})=\rho_0^{-2} \varphi(v_0+\tilde{v}\rho_0)$. Since $\vphi\in C^{2,\a,\b}$ (\lemref{2abinrealcord}), we see
\begin{align}
\norm{\tilde{\varphi}(\tilde{v})}_{C^0(B_{c_\beta})}\leq C \rho_0^{-2}
\label{eqn:tildeWC0}
\end{align}
and the $C^0$ norm of the first order derivative of $\tilde{\varphi}$ becomes
\begin{align}
\norm{\nabla_{\tilde{V}} \tilde{\varphi}(\tilde{v})}_{C^0(B_{c_\beta})}\leq C \rho_0^{-1}.
\label{eqn:tildeWC1}
\end{align}
%
\subsubsection{$k_3=0$}
We take derivative $\nabla_{\tilde V}$ on the rescaled equation \eqref{1stequationscale} of $\tilde\vphi$ to have, 
\begin{align}\label{eqn:dwpsi}
\triangle_g (\nabla_{\tilde{V}} \tilde{\varphi}) = \nabla_{\tilde{V}} K.
\end{align}
The coefficients and right hand side are estimated in Proposition \ref{lem:goodmetric}, as 
\begin{align*}
\norm{g_{i\bar{j}}}_{C^{k,\a}_{\tilde{V}}(B_{3c_\beta/4})} \leq C(k);\quad
\norm{K}_{C^{k,\a}_{\tilde{V}}(B_{3c_\beta/4})}\leq C(k).
\end{align*} 
So the Schauder estimate applied to \eqref{eqn:dwpsi} gives for any $k\geq 0$,
\begin{align}\label{eqn:smallk3}
\norm{\nabla_{\tilde{V}}\tilde{\varphi}}_{C^{k+2,\alpha}_{\tilde{V}}(B_{c_\beta/4})}
&\leq C \left( \norm{\nabla_{\tilde{V}}\tilde{\varphi}}_{C^0(B_{c_\beta/2})} + \norm{\nabla_{\tilde{V}} K}_{C^{k,\alpha}_{\tilde{V}}(B_{c_\beta/2})} \right) \\
&\leq C \rho_0^{-1},\nonumber
\end{align}
which implies that $$\norm{\tilde{\varphi}}_{C^{k+3,\alpha}_{\tilde{V}}(B_{c_\beta/4})} \leq C \rho_0^{-1}.$$ Rescaled back to $\vphi$, this inequality above becomes
\begin{align}\label{eqn:k30}
\norm{\varphi}_{C^{k}_{\tilde{V}}(B_{c_\beta/4})}\leq C , \quad \forall k\geq 3.
\end{align}
Thus \thmref{thm:interior} is proved with $k_3=0$, combined with the lower order estimates in \lemref{2abinrealcord}.

%

\subsubsection{$k_3=1$}
Rewrite \eqref{eqn:smallk3} in $\vphi$, it reads
\begin{align*}
\norm{\nabla_{V}\varphi}_{C^{k+2,\alpha}_{\tilde{V}}(B_{c_\beta/4})} \leq C , \quad \forall k\geq 0.
\end{align*} Particularly, 
\begin{align*}\norm{\nabla_{V,\xi}\varphi}_{C^{k+2,\alpha}_{\tilde{V}}(B_{c_\beta/4})} \leq C , \quad \forall k\geq 0.
\end{align*} The lower order estimates, $\norm{\nabla_{V,\xi}\varphi}_{C^{0,\alpha}_{\tilde{V}}(B_{c_\beta/4})} $ and $\norm{\nabla_{V,\xi}\varphi}_{C^{1,\alpha}_{\tilde{V}}(B_{c_\beta/4})} $, follow from \lemref{2abinrealcord}. So we have proved \thmref{thm:interior} with $k_3=1$.

%

\subsubsection{$k_3\geq2$}

From the boundedness of the tangential derivatives of $\varphi$ and the definition of $\tilde\vphi$:
\begin{align}\label{eqn:k32}
\norm{(\nabla_{\tilde{V},\xi})^{k_3} \tilde{\varphi}}_{C^0(B_{c_\beta})}\leq C \rho_0^{k_3-2}.
\end{align}
By Lemma 2.1 in \cite{arxiv:1609.03111}, the proof of Theorem \ref{thm:interior} for the case $k_3>1$ is reduced to the claim that
\begin{align}\label{eqn:weclaim}
\norm{(\nabla_{\tilde{V},\xi})^{k_3} \tilde{\varphi}}_{C^{l,\alpha}_{\tilde{V}}(B_{c_\beta/4})}\leq C(l) \rho_0^{k_3-2}, \quad \text{for} \quad \forall  l\geq0.
\end{align}


The rest of this section is devoted to the proof of \eqref{eqn:weclaim}, which is an induction on $k_3$. For $k_3=2$, we take one more $\nabla_{\tilde{V},\xi}$ of the equation from \eqref{eqn:dwpsi},
\begin{align*}
\triangle_g (\nabla_{\tilde{V},\xi} \tilde{\varphi}) = \nabla_{\tilde{V},\xi} K
\end{align*}
to get
\begin{align}\label{eqn:2k3}
\triangle_{g} (\nabla_{\tilde{V},\xi}^2 \tilde{\varphi}) = (\nabla_{\tilde{V},\xi} \tilde{\varphi})_{\tilde{V},i\bar{j}} \# \nabla_{\tilde{V},\xi} g^{i\bar{j}} + \nabla^2_{\tilde{V},\xi} K.
\end{align}

In order to apply the Schauder estimate, we check that: 
\begin{itemize}
	\item Proposition \ref {lem:goodmetric} tells us the coefficient $g$ and the term $K$ are both in $C^{l,\a}_{\tilde{V}}(B_{3c_\beta/4})$ for $l=1,2,\cdots$;
	\item \eqref{eqn:k32} gives us that the $C^0(B_{c_\beta})$ norm of $\nabla^2_{\tilde{V},\xi} \tilde{\varphi}$ is bounded by a constant independent of $\rho_0$;
	\item the non-homogeneous term is (by switching the order of derivatives) 
	\begin{align*}
	(\nabla_{\tilde{V},\xi} \tilde{\varphi})_{\tilde{V},i\bar{j}} \# \nabla_{\tilde{V},\xi} g^{i\bar{j}} + \nabla^2_{\tilde{V},\xi} K= \nabla_{\tilde{V},\xi} g_{i\bar{j}} \# \nabla_{\tilde{V},\xi} g^{i\bar{j}}+ \nabla^2_{\tilde{V},\xi} K,
	\end{align*}
	where the first term is estimated by Proposition \ref{lem:goodtangent},  for $l=1,2,\cdots$,
	\begin{align}\label{eqn:non1}
	\norm{\nabla_{\tilde{V},\xi} g_{i\bar{j}} \# \nabla_{\tilde{V},\xi} g^{i\bar{j}}}_{C^{l,\alpha}_{\tilde{V}}(B_{c_\beta/2})} \leq C \rho_0^2 \leq C
	\end{align}
	and the second term by Proposition \ref{lem:goodmetric} (applying to $\nabla^2_\xi K$)
	\begin{align}\label{eqn:non2}
	\norm{\nabla^2_{\tilde{V},\xi} K}_{C^{l,\alpha}_{\tilde{V}}(B_{c_\beta/2})} \leq C \rho_0^2 \leq C.
	\end{align}
	
\end{itemize}

Applying the Schauder estimate to \eqref{eqn:2k3}, we conclude the proof of \eqref{eqn:weclaim} for $k_3=2$. We repeat this argument to see that \eqref{eqn:weclaim} for any $k_3$ holds. Hence Theorem \ref{thm:interior} is completely proved.




\section{Linear theory of Lichnerowicz operators}\label{Linear theory for Lichnerowicz operator}
We now denote the cscK cone metric with $C^{2,\a,\b}$ cscK potential $\vphi$ by 
\begin{align*}
\om:=\om_{cscK}=\om_D+i\p\bar\p\vphi.
\end{align*}
Recall the cscK cone equations \eqref{P} and \eqref{2nd equ smooth} in Section \ref{cscKcone}, with $P \in C^{2,\a,\b}$, 
\begin{align*}
\left\{
\begin{aligned}
\tri_{\om} P&=\Tr_{\om}\theta-\underline S_\b,\\
\frac{\om^n}{\om_0^n}&=\frac{e^{P+h_0}}{|s|_h^{2-2\b}}.
\end{aligned}
\right.
\end{align*}

Let $Q$ be the variation of $P$ and $u$ be the variation of $\vphi$. 
The linearised equation of the cscK cone equation at the cscK cone metric $\om$ on functions $u$ and $Q$ reads
\begin{align}\label{linearizedcsck}
\left\{
\begin{aligned}
\tri_{\om} Q&=-u^{i\bar j}T_{i\bar j},\\
\tri_\om u&=Q,
\end{aligned}
\right.
\end{align}
or as a fourth order Lichnerowicz operator, 
\begin{align*}
{{\mathbb{L}\mathrm{ic}}}_{\om}(u)=\tri^2_\om u+u^{i\bar j}T_{i\bar j}.
\end{align*}

We introduce the following H\"older spaces $C_\mathbf{w}^{4,\a,\b}(\om)$ to be our solution space for \eqref{linearizedcsck}.
\begin{defn}[H\"older spaces $C_\mathbf{w}^{4,\a,\b}(\om)$] \label{cw4ab}
	We define a "weak" fourth order H\"older space with respect to a given K\"ahler cone metric $\om$,
	\begin{align}
	C_\mathbf{w}^{4,\a,\b}(\om)=\{u\in C^{2,\a,\b}\vert \tri_{\om} u \in C^{2,\a,\b}\}.
	\end{align}

\end{defn}
	We also use an appropriate normalisation condition later.
In this definition, not all 4th order derivatives of $u$ are $C^{0,\a,\b}$.
Apparently, if the reference K\"ahler cone $\om$ is classically $C^{2,\a}$ outside the divisor, on the regular part $M$, then the function in this space is $C^{4,\a}$ in the usual sense, according to the interior Schauder estimate.

We define the spaces $C^{4,\a,\b}(\om)$ in Section 2 in \cite{arxiv:1511.02410} and Section 2 in \cite{arxiv:1703.06312}, the function in which has all their 4th order derivatives in $C^{0,\a,\b}$. The idea is firstly to define a local 4th order H\"older space with respect to the flat cone metric $\om_{\b}$ near the cone points. Then we glue the local 4th order H\"older spaces together to the global spaces $C^{4,\a,\b}(\om)$ with respect to a background K\"ahler cone metric $\om$ as Proposition 2.2 in \cite{arxiv:1703.06312}. The background metric $\om$ is required to have nice geometric properties; and with such nice $\om$, we could prove $C_\mathbf{w}^{4,\a,\b}(\om)$ coincides with $C^{4,\a,\b}(\om)$.
To realise such construction, the background metric $\om$ could be chosen to be the model cone metric $\om_D$ as Proposition 5.5 in \cite{arxiv:1511.02410}, when the cone angle and the H\"older exponent satisfy the half angle condition
\begin{align}\label{anglerestriction}
0<\b<\frac{1}{2};\quad \a\b<1-2\b.
\end{align} 
The advantage of the Definition \ref{linearizedcsck} is that we don't need geometric conditions of the background metric $\om$.

In this section, we prove the Fredholm alternative theorem in $C_\mathbf{w}^{4,\a,\b}(\om)$ when $\frac{1}{2}<\b\leq1$.
\begin{thm}[Linear theory; $0<\b\leq1$]\label{Fredholm}
	Let $X$ be a compact K\"ahler manifold, $D\subset X$ a smooth divisor, $\omega$ a cscK cone metric with $C^{2,\a,\b}$ potential such that the cone angle $2\pi\beta$  and the H\"older exponent $\alpha$ satisfy 
	\begin{align*} 
	0<\b\leq1;\quad \a\b<1-\b.
	\end{align*} Assume that $f\in C^{0,\a,\b}$ with normalisation condition $\int_X f\om^n =0$. Then one of the following holds:
	\begin{itemize}
		\item Either the Lichnerowicz equation ${{\mathbb{L}\mathrm{ic}}}_{\om}(u)=f$ has a unique $C_\mathbf{w}^{4,\a,\b}(\om)$ solution.
		\item Or the kernel of ${{\mathbb{L}\mathrm{ic}}}_{\om}(u)$ generates a holomorphic vector field tangent to $D$.
	\end{itemize}
\end{thm}

With half angle restriction \eqref{anglerestriction}, 
the Fredholm alternative theorem above is proved in \cite{arxiv:1703.06312} in the better space $C^{4,\a,\b}(\om)$. 

In Section \ref{Asymptotics of functions}, we prove an asymptotic of functions in $C_\mathbf{w}^{4,\a,\b}(\om)$ for further use.

The idea of the proof was carried out in \cite{arxiv:1703.06312} when the cone angle is less than half. When angle is larger than half, we use the weaker space $C_\mathbf{w}^{4,\a,\b}(\om)$. The new observation is that we could prove appropriate a priori estimates for partial fourth order derivatives, nut not the full control of all 4th order derivatives.

We first consider the $K$-bi-Laplacian operator $$(\tri_{\om}-K) \tri_\om u$$ and prove the existence and uniqueness of the weak solution in the Sobolev space 
$H_{\mathbf{w},0}^{2,\b}(\om)$. We next improve the regularity of the weak solution to H\"older spaces $C_\mathbf{w}^{4,\a,\b}(\om)$, using the 2nd Schauder estimate. These two steps will be done in Section \ref{Bi-Laplacian equations: existence and regularity}. 

Then we define a continuity path 
\begin{align*}
L_t^Ku
:=(\tri_{\om}-K) \tri_\om u+t  u^{i\bar j}T_{i\bar j}(\om)
\end{align*} connecting the $K$-bi-Laplacian operator (t=0) and the $K$-Lichnerowicz operator (t=1) \begin{align*}
{{\mathbb{L}\mathrm{ic}}}_{\om}(u)-K\tri_\om u,
\end{align*} and prove uniform a priori $C_\mathbf{w}^{4,\a,\b}(\om)$ estimates along the path in Section \ref{Continuity method}. 

At last, the theorem follows from the functional analysis theorems for the usual continuity method.

The difficulties come from the H\"older space $C_\mathbf{w}^{4,\a,\b}(\om)$, which is the most natural one for the solution to stay, but the function $u$ in this space have no a priori boundedness on the pure 2nd derivatives, $\p_1\p_1 u$. That causes the failure of the Ricci identity of $u$, i.e. 
\begin{align*}
\int_M |\p\p u|^2\om^n=\int_M |\p\bar \p u|^2\om^n-\int_MRic(\p u, \p u)\om^n.
\end{align*} We overcome his problem by using the RHS. of this identity to define a bilinear form. Fortunately, it also leads to a weak Sobolev spaces $H_{\mathbf{w},0}^{2,\b}(\om)$ and a global $L^2$-estimates. Then the 2nd Schauder estimate and compactness of the 2nd H\"older space $C^{2,\a,\b}$ are applied to obtain the $C_\mathbf{w}^{4,\a,\b}(\om)$ estimate.

We close this section by showing that our Fredholm alternative for Lichnerowicz operator, i.e. Theorem \ref{Fredholm} is reduced to \thmref{K-Lichnerowicz equation}.
\begin{proof}[Proof of Theorem \ref{Fredholm}]
	Due to \thmref{K-Lichnerowicz equation}, the $K$-Lichnerowicz operator \begin{align*}{{\mathbb{L}\mathrm{ic}}}^K_{\omega}(u)={{\mathbb{L}\mathrm{ic}}}_{\omega}(u) -K \tri_\om u
	\end{align*} is invertible and the inverse map \begin{align*}({{\mathbb{L}\mathrm{ic}}}^K_{\om})^{-1}: C_\mathbf{w}^{4,\a,\b}(\om)\rightarrow C_\mathbf{w}^{4,\a,\b}(\om)
\end{align*} is compact. 
	Now, we consider the Lichnerowicz equation, i.e
	\begin{align*}
	{{\mathbb{L}\mathrm{ic}}}_{\omega}(u)={{\mathbb{L}\mathrm{ic}}}_\om^{K}(u)+K \tri_\om u=f.
	\end{align*}
	Actually, this is equivalent, after taking $({{\mathbb{L}\mathrm{ic}}}^K_{\om})^{-1}$, 
	\begin{align*}
	u+K  ({{\mathbb{L}\mathrm{ic}}}_\om^{K})^{-1} \tri_\om u= ({{\mathbb{L}\mathrm{ic}}}_\om^{K})^{-1}f.
	\end{align*}
	According to \thmref{K-Lichnerowicz equation}, the mapping  \begin{align*}\mathfrak{T}:=-K  ({{\mathbb{L}\mathrm{ic}}}^K_{\om})^{-1}\tri_\om : C_\mathbf{w}^{4,\a,\b}(\om)\rightarrow C_\mathbf{w}^{4,\a,\b}(\om)
	\end{align*} is compact, since \begin{align*}\tri_\om : C_\mathbf{w}^{4,\a,\b}(\om)\rightarrow C^{2,\a,\b}(\om)
	\end{align*} and the identity map $$C^{2,\a,\b}(\om)\rightarrow C^{0,\a,\b}(\om)$$ is compact. Thus Theorem \ref{Fredholm} follow from the Fredholm alternative from functional analysis theory.
\end{proof}
\subsection{Bi-Laplacian equations: existence and regularity}\label{Bi-Laplacian equations: existence and regularity}
We will prove solutions to the following $K$-bi-Laplacian equation 
in $C_\mathbf{w}^{4,\a,\b}(\om)$,
\begin{align}\label{KbiLaplacian equation}
(\tri_{\om}-K) \tri_\om u=f.
\end{align}
The positive constant $K$ will be determined later in \propref{weakersolution}.
\begin{thm}[$K$-bi-Laplacian equation]\label{K-bi-Laplacian}
	Assume that $\om=\om_D+i\p\bar\p\vphi$ is a K\"ahler cone metric with $\vphi\in C^{2,\a,\b}$, $f\in C^{0,\a,\b}$, $0<\b\leq 1$ and the H\"older exponent satisfies $\a\b<1-\b$. Suppose that $K> C_P+1$. Then there exists unique $C_\mathbf{w}^{4,\a,\b}(\om)$ solution to equation \eqref{KbiLaplacian equation}. 
\end{thm}
\begin{proof}
	The proof is divided into the existence part (Proposition \ref{weakersolution}) and the regularity part (Proposition \ref{Schauderweakomega}).
\end{proof}

We first need to define the proper notion of week solution for equation \eqref{KbiLaplacian equation}.
Given a K\"ahler cone metric, its volume element is an $L^p$ function (for some $p\geq 1$) with respect to a smooth metric and gives rise to a measure $\om^n$ on $M$. We will use the following Sobolev spaces and their embedding theorems with respect to $\om$.
\begin{defn}[Sobolev spaces $W^{1,p,\b}(\om)$] 
	For a K\"ahler cone metric $\omega$,  the Sobolev spaces $W^{1,p,\b}(\om)$ for $p\geq 1$
	are defined with respect to $\om$. The $W^{1,p,\b}(\om)$ norm is 
	\begin{align*}
	\Vert u \Vert_{W^{1,p,\b}(\om)}&=\left(\int_M \vert u \vert^p + |\p u|_\om^p\om^n\right)^{1/p}.
	\end{align*}
\end{defn}

\begin{defn}
	We define the Sobolev space $H^{1,\b}:=W^{1,2,\b}(\om)$ and $H_{0}^{1,\b}=\{u\in H^{1,\b}\vert \int_M u \om^n=0\}$. The Sobolev norm remains the same. 
\end{defn}

\begin{lem}[Sobolev inequality]\label{w1pbSobolev imbedding theorem}
	Assume that $u\in W^{1,p,\b}(\om)$. If $p< 2n$, then there exists a constant $C$ such that
	\begin{align*}
	\Vert u\Vert_{L^{q}(\om)}\leq C\Vert u\Vert_{W^{1,p,\b}(\om)},
	\end{align*}for any $q\leq \frac{2np}{2n-p}$.
\end{lem}

\begin{lem}[Poincar\'e inequality]\label{Poincareinequality}
	There is a constant $C_P$ such that for any $u\in H_{0}^{1,\b}=\{u\in H^{1,\b}\vert \int_M u \om^n=0\}$,
	\begin{align*}
	\Vert u\Vert_{L^2(\om)} \leq C_P \Vert \p u\Vert_{L^2(\om)}. 
	\end{align*}
\end{lem}

\begin{defn}[Sobolev spaces $H^{2,\b}_{\mathbf{w} }(\om)$] 
We define 2nd Sobolev space $H^{2,\b}_{\mathbf{w} }$ with semi-norm $$[ u ]_{H^{2,\b}_{\mathbf{w} }(\om)}=\sum_{1\leq a,b\leq n}\Vert \p_a\p_{\bar b}u\Vert_{L^2(\om)}$$ and norm $$|| u ||_{H^{2,\b}_{\mathbf{w} }(\om)}=|| u ||_{H^{1,\b}(\om)}+[ u ]_{H^{2,\b}_{\mathbf{w} }(\om)}.$$ 
\end{defn} 


We define the bilinear form on $H_{\mathbf{w} ,0}^{2,\b}(\om)$ by
\begin{align*}
\mathcal{B}^K(u,\eta):=\int_M[ (i\p\bar\p u ,i\p\bar\p\eta)_\om+K (\p u, \p \eta)_\om]\om^n
\end{align*}
for all $u,\eta\in H_{\mathbf{w} ,0}^{2,\b}(\om)$, 
and define $u$ to be the $H_{\mathbf{w} ,0}^{2,\b}(\om)$-weak solution of the $K$-bi-Laplacian equation \eqref{KbiLaplacian equation}, if it satisfies the following identity for all $\eta\in H_{\mathbf{w} ,0}^{2,\b}(\om)$,
\begin{align*}
\mathcal{B}^K(u,\eta)=\int_M f\eta\om^n.
\end{align*}

\begin{prop}[$H_{\mathbf{w},0}^{2,\b}(\om)$-weak solution]\label{weakersolution}
	Assume that $\om$ is a K\"ahler cone metric. Suppose that $K> C_P+1$ and $f$ is in the dual space $ (H_{\mathbf{w},0}^{2,\b}(\om))^\ast$.
	Then the $K$-bi-Laplacian equation \eqref{KbiLaplacian equation} has a unique weak solution $u\in H_{\mathbf{w},0}^{2,\b}(\om)$.
\end{prop}
\begin{proof}
	The boundedness of $\mathcal{B}^K$ follows from its definition by using Cauchy-Schwartz inequality,
	\begin{align*}
	\mathcal{B}^K(u,\eta)
	&\leq \Vert u\Vert_{H^{2,\b}_{\mathbf{w} }(\om)}\Vert \eta\Vert_{H^{2,\b}_{\mathbf{w} }(\om)}
	+K \Vert u\Vert_{H^{1,\b}(\om)}\Vert \eta\Vert_{H^{1,\b}(\om)}.
	\end{align*}
	Then we prove the coercivity. The definition of the bilinear form $\mathcal{B}^K$ implies that
	\begin{align*}
	\Vert u\Vert ^2_{H_{\mathbf{w} }^{2,\b}(\om)}
	&	= \int_M(|\p\bar\p u|^2+|\p u|_\om^2+|u|^2)\om^n\\
	&	=\mathcal{B}^K(u,u)+\int_M[(1-K)|\p u|_\om^2+|u|^2]\om^n.
	\end{align*}
	Then, with Poincar\'e inequality (\lemref{Poincareinequality}), the RHS above is controlled by
	\begin{align*}
	\mathcal{B}^K(u,u)+\int_M[(1-K+C_P)|\p u|_\om^2]\om^n.
	\end{align*}
	Thus choosing $K>C_P+1$, we have
	\begin{align*}
	\Vert u\Vert ^2_{H^{2,\b}_{\mathbf{w}}(\om)}\leq \mathcal{B}^K(u,u).
	\end{align*}
	Thus the Lax-Milgram theorem applies and there is a unique weak solution $u\in H_{\mathbf{w},0}^{2,\b}$ to equation \eqref{KbiLaplacian equation}. 
\end{proof}

We now improve the regularity of the $H_{\mathbf{w},0}^{2,\b}(\om)$-weak solution $u$.
\begin{prop}[Schauder estimate]\label{Schauderweakomega}
	Assume that $\om=\om_D+i\p\bar\p\vphi$ is a K\"ahler cone metric with $\vphi\in C^{2,\a,\b}$, $f\in C^{0,\a,\b}$ and the H\"older exponent satisfies $\a\b<1-\b$. Then the weak solution $u\in H^{2,\b}_{\mathbf{w},0}$ to equation \eqref{KbiLaplacian equation} is actually $C_\mathbf{w}^{4,\a,\b}(\om)$. 
\end{prop}
\begin{proof}
	Applied to equation $(\tri_{\om}-K ) v=f$, the 2nd order linear elliptic theory \cite{MR2975584} tells us that this equation has a unique $C^{2,\a,\b}$ solution $v$, and also
	\begin{align}\label{Schauderweakomegav}
	\int_Mv (\tri_\om -K )\eta\om^n=\int_Mf\eta \om^n.
	\end{align} with $\eta$ is a smooth function with vanishing average. We claim that: this unique solution $v$ has to be $\tri_\om u$ in Proposition \ref{weakersolution}.
	We then prove this claim. Choosing $\eta$ as above, we have from the definition of the $ H^{2,\b}_{\mathbf{w},0}$ weak solution,
	$$\int_M[ (i\p\bar\p u ,i\p\bar\p\eta)_\om+K (\p u, \p \eta)_\om]\om^n=\int_Mf\eta \om^n.$$
	Since $\eta$ is smooth, we are able to use the integration by parts to the left hand side
	\begin{align*}
	&\int_M (i\p\bar\p u ,i\p\bar\p\eta)_\om\om^n=-\int_M u_{i} \eta_{\bar i j \bar j}\om^n=\int_M \tri_\om u \tri_\om \eta\om^n;\\
	&\int_M (\p u, \p \eta)_\om\om^n=-\int_M \tri_\om u \eta\om^n,
	\end{align*} and then obtain 
	\begin{align*}
	\int_M\tri_\om u (\tri_\om -K )\eta \om^n=\int_Mf\eta \om^n.
	\end{align*}
	Comparing with \eqref{Schauderweakomegav} and using the uniqueness of $v$, we see that $\tri_\om u$ is the same to $v$, and thus in $C^{2,\a,\b}$. Applying linear theory again, we obtain that $u\in C^{2,\a,\b}$.
\end{proof}


\subsection{Fredholm alternative for Lichnerowicz operators}\label{Continuity method}
In order to prove the Fredholm alternative \thmref{Fredholm}, it suffices to consider the $K$-Lichnerowicz equation, \begin{align}\label{K-Lic}
{{\mathbb{L}\mathrm{ic}}}^K_{\omega}(u)={{\mathbb{L}\mathrm{ic}}}_{\omega}(u) -K \tri_\om u=f.
\end{align}
\begin{thm}[$K$-Lichnerowicz equation]\label{K-Lichnerowicz equation}
	Assume that $\om=\om_D+i\p\bar\p\vphi$ is a K\"ahler cone metric with $\vphi\in C^{2,\a,\b}$, $f\in C^{0,\a,\b}$, $0<\b\leq 1$ and the H\"older exponent satisfies $\a\b<1-\b$. Suppose that $K>1+ \Vert T\Vert_{L^\infty}+2C_P$. Then there exists unique $C_\mathbf{w}^{4,\a,\b}(\om)$ solution to equation \eqref{K-Lic}. 
\end{thm}
\begin{proof}
We define the continuity path
\begin{align}\label{continuitypath}
L_t^Ku
:=(\tri_{\om}-K) \tri_\om u+t  u^{i\bar j}T_{i\bar j}(\om)
\end{align}
with \begin{align*}L_t^K:  C_\mathbf{w}^{4,\a,\b}(\om)\rightarrow C^{0,\a,\b}
\end{align*} for any $0\leq t\leq 1$.

When $t=0$, we already solved 
\begin{align*}
(\tri_{\om}-K) \tri_\om u=f
\end{align*}
for any $f\in C^{0,\a,\b}$ and obtain a unique solution $u\in C_\mathbf{w}^{4,\a,\b}(\om)$ thanks to \thmref{K-bi-Laplacian}. 

In order to use the continuity method in the H\"older space $C_\mathbf{w}^{4,\a,\b}(\om)$ to prove the existence and uniqueness of $C_\mathbf{w}^{4,\a,\b}(\om)$ solution for the $K$-Lichnerowicz equation, it is sufficient to obtain the a priori estimates (\thmref{closedness}). We will prove these estimates in the rest of this section.
\end{proof}

We will need the following lemmas to do integration by parts.
\begin{lem}\label{IbPRic}
	Assume that $u\in C^{2,\a,\b}$, the tensor $T$ is bounded and $\nabla T=0$ on $M$. Then it holds
	\begin{align*}
	\int_M u^{i\bar j}T_{i\bar j}(\om)u \om^n
	=-\int_M u^{i}T_{i\bar j}(\om)u^{\bar j} \om^n.
	\end{align*}
\end{lem}
\begin{proof}
	We apply the cutoff function $\chi_\eps$ which has been fully discussed in \cite{arxiv:1511.02410}. Then the argument is essentially Lemma 4.10 in \cite{arxiv:1603.01743}. By dominated convergence theorem, we have
	\begin{align*}
	\lim_{\eps\rightarrow 0}\int_M u^{i\bar j}T_{i\bar j}(\om)u\chi_\eps \om^n
	=\int_M u^{i\bar j}T_{i\bar j}(\om)u \om^n.
	\end{align*}
	On the other hand, using $\nabla T=0$ on $M$,
	\begin{align*}
	\int_M u^{i\bar j}T_{i\bar j}(\om)u\chi_\eps \om^n
	&=-\int_M u^{i}T_{i\bar j}(\om)(u \chi_\eps)^{\bar j} \om^n\\
	&=-\int_M u^{i}T_{i\bar j}(\om)u^{\bar j} \chi_\eps \om^n
	-\int_M u^{i}T_{i\bar j}(\om)\chi_\eps^{\bar j} u \om^n
	\end{align*}
	The first term converges under the assumption on $u$ and $T$.
	The second term also converges, since for $2\leq i,j\leq n$,  
	\begin{align*}
	&u^{1}T_{1\bar 1}(\om)\chi_\eps^{\bar 1}=\eps. O(r^{-\b}),\quad u^{1}T_{1\bar j}(\om)\chi_\eps^{\bar j}=\eps. O(1),\\
	&u^{i}T_{i\bar 1}(\om)\chi_\eps^{\bar 1}=\eps. O(r^{-\b}),\quad u^{i}T_{i\bar j}(\om)\chi_\eps^{\bar j}=\eps. O(1).
	\end{align*}
\end{proof}

\begin{lem}\label{BKt}
Assume that $u\in C_\mathbf{w}^{4,\a,\b}(\om)$. Then the identity holds
\begin{align*}
\mathcal{B}^K_t(u,u):=\int_M u L_t^K  u \om^n=\int_M[|\tri_\om u|^2- t u^{i}T_{i\bar j}(\om)u^{\bar j} +K|\p u|_\om^2]\om^n.
\end{align*}
\end{lem}
\begin{proof}
We multiply $L_t^K  u $ with $u$ and integrate over the manifold $M$, 
\begin{align*}
\int_M u L_t^K  u \om^n
&=\int_Mu [(\tri_{\om}-K) \tri_\om u+t  u^{i\bar j}T_{i\bar j}(\om)]\om^n.
\end{align*}Since both $u$ and $\tri_\om u$ are $C^{2,\a,\b}(\om)$, we have
$$\int_Mu\tri^2_{\om}u\om^n=\int_M|\tri_{\om}u|^2\om^n$$ and $$\int_Mu \tri_{\om} u\om^n=-\int_M|\p  u|_\om^2\om^n.$$ 
Thus the identity follows from \lemref{IbPRic}.
\end{proof}

We need a global $L^2$-estimate for the linear equation with K\"ahler cone coefficients.
\begin{lem}[Global $L^2$-estimate]\label{l2pureestimate} 
There exists a constant $C>0$ depending on $n,\om,M$ such that,
\begin{align*}
\Vert u\Vert ^2_{H_{\mathbf{w}}^{2,\b}(\om)}
\leq C (\Vert\tri_\om u\Vert_{L^2(\om)}+\Vert u\Vert ^2_{H_{\mathbf{w}}^{1,\b}(\om)}).
\end{align*}
\end{lem}
\begin{proof}
	This is a similar to Section 3.2 in \cite{arxiv:1703.06312}, by patching local estimates together. The difference is that we did not need to control $\p_1\p_j u$, so we do no require bound on the Christoffel symbols of $\om$. 
	
	We let $\{U_i,\psi_i;1\leq i\leq N\}$ be the finite cove of $M$ and $\rho_i$ be the corresponding partition of unity, and  supported on $\mathtt{B}_1\subset \mathtt{B}_{3}\subset U_i$ for each $i$. In the cone charts, where the charts have nonempty intersection with the divisor $D$, we choose $\psi_i$ such that $|\om-\om_{cone}|_{L^\infty(\mathtt{B}_3)}$ is sufficiently small, which could be done since $\om$ is $C^{0,\a,\b}$. Then freezing the leading coefficients and applying the local $L^2$-estimate to the cut-off equation 
	\begin{align*}
	&\tri_{\om_{cone}} (\rho_i u)\\
	&=(\tri_{\om_{cone}} (\rho_i u)-\tri_{\om} (\rho_i u))+\tri_\om \rho_i u
	+ \rho_i \tri_\om u
	+2(\p \rho_i,\p u)_{\om}:=f
	\end{align*} to obtain the estimate on $\rho_i u$,
	\begin{align*}
	[\rho_i u]_{H_{\mathbf{w}}^{2,\b}(\mathtt{B}_1;\om_{cone})}
	&\leq C\Vert f \Vert_{L^{2}(\mathtt{B}_2;\om_{cone})}
	\end{align*}
	Then the conclusion follows from adding $\sum_i\rho_i u$ and Cauchy-Schwartz inequality.
\end{proof}

We also need a H\"older interpolation inequality, which follows from the compactness of $C^{2,\a,\b}$ to $C^{0,\a,\b}$.
\begin{lem}[H\"older interpolation inequality]\label{Holderinterpolation} 
	For all $u\in C^{2,\a,\b}$ and any $\eps>0$, there exists a constant $C(\eps)$ such that
\begin{align*}
|u|_{C^{0,\a,\b}}\leq \eps |u|_{C^{2,\a,\b}}+C(\eps) ||u||_{L^2(\om)}.
\end{align*}
\end{lem}
\begin{proof}
	We prove by contradiction method. Assume that there is $\eps_0$ and $u_k$ such that for any $k\rightarrow\infty$, we have 
	\begin{align*}
|u_k|_{C^{0,\a,\b}}>\eps_0 |u_k|_{C^{2,\a,\b}}+k ||u_k||_{L^2(\om)}.
\end{align*}
Normalise $u_k$ such that $|u_k|_{C^{2,\a,\b}}=1$. The compactness of $C^{2,\a,\b}$ to $C^{0,\a,\b}$ gives limit of $u_k$ to $u_\infty$ in $C^{0,\a,\b}$. the inequality above implies that $ ||u_\infty||_{L^2(\om)}=0$ which contradicts to the normalisation condition $|u_\infty|_{C^{2,\a,\b}}=1$.
\end{proof}

Now we are ready to prove the key a priori estimates.

\begin{thm}\label{closedness}
	Assume $\omega$ is a cscK cone metric, and the H\"older exponent satisfy $\a\b<1-\b$. Assume that $K>1+ \Vert T\Vert_{L^\infty}+2C_P$.
	There is constant $C_1$ such that for any $u\in C_\mathbf{w}^{4,\a,\b}(\om)$ along the continuity path \eqref{continuitypath} with $0\leq t\leq 1$, we have
	\begin{align*}
	|u|_{C_\mathbf{w}^{4,\a,\b}(\om)}\leq C_1 |L^K_t u|_{C^{0,\a,\b}}.
	\end{align*}
\end{thm}
\begin{proof}
	Applying Schauder estimate to the equation
	\begin{align*}
	\tri^2_\om u=L_t^Ku- t  u^{i\bar j}T_{i\bar j}(\om)+K\tri_\om u,
	\end{align*}
	we have
	\begin{align*}
	|\tri_\om u|_{C^{2,\a,\b}} 
	\leq C_2 \left(|L^K_t u-t  u^{i\bar j}T_{i\bar j}(\om)+K\tri_\om u|_{C^{0,\a,\b}}+||\tri_\om u||_{L^{2}(\om)}\right).
	\end{align*}
	Note that $T$ is $C^{0,\a,\b}$, the RHS. is bounded by
	\begin{align*}
	C_3 \left(|L^K_t u|_{C^{0,\a,\b}}+|u|_{C^{2,\a,\b}}+|\tri_\om u|_{C^{0,\a,\b}}\right).
	\end{align*}
	Applying the Schauder estimate to the middle term, we get
	\begin{align*}
	|u|_{C^{2,\a,\b}} 
	\leq C_4 \left(|\tri_\om u|_{C^{0,\a,\b}}+||u||_{L^{2}(\om)}\right).
	\end{align*}
	Combine them together, we have
	\begin{align*}
	|u|_{C_\mathbf{w}^{4,\a,\b}(\om)}\leq C_5 \left(|L^K_t u|_{C^{0,\a,\b}}+|\tri_\om u|_{C^{0,\a,\b}}+||u||_{L^{2}(\om)}\right).
	\end{align*}
	
	We use the $\eps$-interpolation inequality of the H\"older spaces,  Lemma \ref{Holderinterpolation} with sufficiently small $\eps$, to the 2nd term on the right hand side,
	\begin{align*}
	|\tri_\om u|_{C^{0,\a,\b}}\leq \eps |\tri_\om u|_{C^{2,\a,\b}}+C(\eps) ||u||_{L^2(\om)}.
	\end{align*}
	Thus we have
	\begin{align}\label{path4ab}
	|u|_{C_\mathbf{w}^{4,\a,\b}(\om)}\leq C_6 \left(|L^K_t u|_{C^{0,\a,\b}}+||u||_{L^{2}(\om)}\right).
	\end{align}

We now estimate $\Vert u\Vert ^2_{H_{\mathbf{w}}^{2,\b}(\om)}$ to control $||u||_{L^{2}(\om)}$.
	We first use the $L^2$-estimate of the cone metrics to the standard linear operator $\tri_\om u$ (Lemma \ref{l2pureestimate}), i.e. there exists a constant $C_7>0$ such that,
	\begin{align*}
	\Vert u\Vert ^2_{H_{\mathbf{w}}^{2,\b}(\om)}
	\leq C_7 \int_M(|\tri_\om u|^2+|\nabla u|_\om^2+|u|^2)\om^n.
	\end{align*}
	With \lemref{BKt}, the RHS of previous inequality becomes,
	\begin{align*}
	&C_8\left(\mathcal{B}^K_t(u,u)+\int_M[ t u^{i}T_{i\bar j}(\om)u^{\bar j} +(1-K)|\nabla u|_\om^2+|u|^2]\om^n\right).
	\end{align*}
	Then we use the Cauchy-Schwartz inequality to $$\mathcal{B}^K_t(u,u)=\int_M u L_t^K  u \om^n\leq 	\Vert u\Vert^2_{L^2(\om)}+ \Vert L^K_t u\Vert_{L^2(\om)}^2,$$ apply the Poincar\'e inequality (Lemma \ref{Poincareinequality}) to the 4th term $\int_M|u|^2\om^n$,
	\begin{align*}
	\Vert u\Vert^2_{L^2(\om)} \leq C_P^2 \Vert \p u\Vert^2_{L^2(\om)}. 
	\end{align*} We set $$K_0=1-K+\Vert T(\om)\Vert_{L^\infty}+2C_P^2,$$ and obtain that
	\begin{align}\label{pureL2}
	\Vert u\Vert ^2_{H_{\mathbf{w}}^{2,\b}(\om))} &\leq C_9\left( \Vert L^K_t u\Vert_{L^2(\om)}^2+K_0\Vert\nabla u\Vert_{L^2(\om)}^2\right).
	\end{align}
	
	We further choose appropriate $K$ such that $K_0<0$, then
	\begin{align*}
	\Vert u\Vert^2_{L^2(\om)}\leq \Vert u\Vert ^2_{H_{\mathbf{w}}^{2,\b}(\om))} \leq C_9 \Vert L^K_t u\Vert_{L^2(\om)}^2\leq C_{10}|L^K_t u|_{C^{0,\a,\b}}^2.
	\end{align*}
	We put the previous inequality to the right hand side of \eqref{path4ab}, to conclude the required estimate, i.e.
	\begin{align*}
	|u|_{C_\mathbf{w}^{4,\a,\b}(\om)}\leq C_1 |L^K_t u|_{C^{0,\a,\b}}.
	\end{align*}
\end{proof}


\subsection{Asymptotic of functions in $C_\mathbf{w}^{4,\a,\b}(\om)$}\label{Asymptotics of functions}

We prove the growth rate of 
\begin{align}\label{nnuformula}
\nabla_k\nabla_i u=\frac{\partial^2}{\p z^k\p z^i} u-\Gamma^{j}_{ki}(g)\p_j u
\end{align} near the divisor. Here $\Gamma^{j}_{ki}(g)$ are the Christoffel symbols of the cscK cone metric $g$.

We use the notations in Section \ref{Asymptotic behaviours} under the rescaled lifted holomorphic coordinates $\tilde{v}$, see Definition \ref{Lifted holomorphic coordinates} and \ref{Rescaled lifted holomorphic coordinates}. Given a point $v_0$ and small neighbourhood $B_{c_\beta\rho_0}(v_0)$ outside the divisor. After rescaling $v$ to $\tilde v$, the neighbourhood becomes a fixed set $B_{c_\beta}(v_0)$.

\begin{prop}\label{nnu}
	Suppose $u\in C_\mathbf{w}^{4,\a,\b}(\om)$ and $\om$ is a cscK cone metric.
	Then
	\begin{align*}
	|\nabla^g\nabla^g u|_g=O(r_0^{\a\b-\b}).
	\end{align*}
	In which, $\nabla^g$ is the covariant differentiation w.r.t. the cone metric $g$. 
\end{prop}
\begin{proof}

We denote $$\tri_{\om,\tilde V}:=g^{i\bar j}\frac{\p^2}{\p \tilde v^i\p \tilde v^{\bar j}}=\rho_0^2g^{i\bar j}\frac{\p^2}{\p  v^i\p v^{\bar j}}$$ the Laplacian with respect to $\om$ in the coordinate $\tilde v$.
We let $$f:=\tri_{\om,\tilde V} u,$$ and then modify $u$ by a polynomial function such that the modified functions and their gradient vanish at the point $\tilde v=0$, i.e. let 
\begin{align*}
u_0=\sum_{2\leq i\leq n}\frac{\p u(0)}{\p \tilde v^i} (\tilde v^i)
+\sum_{2\leq i\leq n}\frac{\p u(0)}{\p \tilde{v}^{\bar i}} (\tilde{v}^{\bar i})
\end{align*} and 
\begin{align*}
w:=u-u_0.
\end{align*}
Then with the definition $\rho_0=r_0^{\b}$, we have
\begin{align*}
w
=O(\rho_0^{\a+1}).
\end{align*}
Since $u_0$ is a linear function, we get $\tri_{\om,\tilde V}  u_0=0$. Then we have
\begin{align*}
\tri_{\om,\tilde V} w=f.
\end{align*}
The metric $\om$ is identified to the flat cone metric and $||g_{i\bar j}||_{C_{\tilde V}^{0,\a}(B_{c_\beta})}$ is bounded by Proposition \ref{lem:goodmetric}. 
We apply the interior Schauder estimate to this equation in $B_{c_\beta}$ and rescaling back (or see Theorem 4.6 in \cite{MR1814364}), we have for any $1\leq i,j\leq n$,
\begin{align*}
\frac{\p^2 w}{\p v^i\p  v^j}
=O(\rho_0^{\a-1}).
\end{align*}
So is $u$, because $u_0$ is a linear function.
Since $u\in C^{1,\a,\b}$ and $\om$ is a cscK cone metric, from \corref{connection},
\begin{align*}
|\Gamma^{j}_{ki}(g)\p_j u|_g=O(\rho_0^{\a-1}).
\end{align*}
The conclusion follows from $\rho_0=r_0^{\b}$.

\end{proof}


\section{The automorphism group is reductive, if a cscK cone metric exists}\label{Reductivity of automorphism group}


Let $Aut(X;D)$ be identity component of the group of holomorphic automorphisms of $X$ which fix the divisor $D$, $\mathfrak{h}(X;D)$ be the space of all holomorphic vector fields tangential to the divisor and $\mathfrak{h}'(X;D)$ be the complexification of a Lie algebra consisting of Killing vector fields of $X$ tangential to $D$.
\begin{thm}\label{preceisereductivity}Suppose $\om$ is a cscK cone metric. 
	Then there exists a one-to-one correspondence between $\mathfrak{h}'(X;D)$ and the kernel of ${\mathbb{L}\mathrm{ic}}_{\om}$. 
	
Precisely speaking, 
the Lie algebra $\mathfrak{h} (X;D)$ has a semidirect sum decomposition:
\begin{equation}
\mathfrak{h}(X;D) = \mathfrak{a}(X;D) \oplus \mathfrak{h'} (X;D),
\end{equation}
where $\mathfrak{a}(X;D)$ is the complex Lie subalgebra of $\mathfrak{h}(X;D)$ consisting of all parallel holomorphic vector fields tangential to $D$,
and $\mathfrak{h'}(X;D)$ is an ideal of $\mathfrak{h}(X;D)$ consisting of the image under $grad_g$ of the kernel of $\cD$ operator. The operator $grad_g$ is defined to be $grad_g(u)= \uparrow^{\omega}\bar\p u=g^{i\bar j}\frac{\p u}{\p z^{\bar j}}\frac{\p}{\p z^i}$.

Furthermore $\mathfrak{h}'(X;D)$ is the complexification of a Lie algebra consisting of Killing vector fields of $X$ tangential to $D$. 
In particular $\mathfrak{h}'(X;D)$ is reductive. 
Moreover, $\mathfrak{h}(X;D)$ is reductive.
\end{thm}
\begin{proof}
	The proof follows the same line of Theorem 4.1 in \cite{arxiv:1603.01743}, with the help of the following lemmas. The new gradients are the use of the space $C_\mathbf{w}^{4,\a,\b}(\om)$, and obtaining the growth rate of $\nabla^g\nabla^g u$ which follows from the asymptotic of the cscK cone metrics in Section \ref{Asymptotic behaviours} and also the asymptotic of the functions in the space $C_\mathbf{w}^{4,\a,\b}(\om)$ in Section \ref{Asymptotics of functions}.
\end{proof}
\begin{lem}\label{13term}Suppose $u\in C_\mathbf{w}^{4,\a,\b}(\om)$ and $\om$ is a cscK cone metric. Then we have
\begin{align*}
\int_M u\cdot  {{\mathbb{L}\mathrm{ic}}}_{\om}(u) \om^n
=- \int_M g^{i\bar j}g^{k\bar l}\cdot \nabla_{\bar j}u \cdot \nabla_{\bar l}\nabla_{ k}\nabla_{ i}  u\cdot \om^n.
\end{align*} 
\end{lem}
\begin{proof}

We will apply the integration by parts directly to the Lichnerowicz operator
\begin{align*}
	{{\mathbb{L}\mathrm{ic}}}_{\om}(u)=\tri^2_\om u+u^{i\bar j}T_{i\bar j}.
\end{align*}  

Firstly, since  $u, \tri_{\om} u\in C^{2,\a,\b}$, we have
\begin{align*}
\int_M u \cdot \tri^2_\om u\cdot \om^n=-\int_M g^{i\bar j}\nabla_{\bar j}u \nabla_{i}(\tri_{\om} u) \om^n.
\end{align*}
In the local normal coordinate on the regular part $M$, we use the K\"ahler condition of $g$ and the Ricci formula to prove
\begin{align*}
\nabla_{ i}(\tri_{\om} u)
&=g^{k\bar l}\nabla_{ i}\nabla_{k}\nabla_{ \bar l}  u
=g^{k\bar l}\nabla_{ k}\nabla_{i}\nabla_{ \bar l}  u
=g^{k\bar l}\nabla_{ k}\nabla_{\bar l}\nabla_{ i}  u\\
&=g^{k\bar l}[\nabla_{\bar l}\nabla_{ k}\nabla_{ i}  u-R^{p}_{ k \bar l i} \nabla_{p} u]
=g^{k\bar l}\nabla_{\bar l}\nabla_{ k}\nabla_{ i}  u-g^{p\bar q}R_{ i \bar q} \nabla_{p} u.
\end{align*}
Thus we have,
\begin{align*}
\int_M u \tri^2_\om u\om^n
=- \int_M[ g^{i\bar j}g^{k\bar l}\nabla_{\bar j}u \nabla_{\bar l}\nabla_{ k}\nabla_{ i}  u-R^{i\bar j}\nabla_{\bar j}u  \nabla_{i} u ]\om^n. 
\end{align*}
From \lemref{Ttensor} and \lemref{IbPRic}, we have
\begin{align*}
\int_M u^{i\bar j}T_{i\bar j}(\om)u \om^n
=-\int_M T_{i\bar j}(\om)\nabla_{\bar j}u  \nabla_{i} u \om^n
=-\int_MR_{i\bar j}(\om)\nabla_{\bar j}u  \nabla_{i} u \om^n.
\end{align*}
Combining these two identities together, we prove the lemma.
\end{proof}

We then continue from the previous lemma and applying the integration by parts. with the help of the asymptotic behaviour of functions in Proposition \ref{nnu}. 

\begin{lem}\label{hvfonm}Suppose $u\in C_\mathbf{w}^{4,\a,\b}(\om)$, $\om$ is a K\"ahler cone metric and $|\nabla_k\nabla_i u|_g=O(r^{-\kappa})$ with $-\kappa+\b>0$. Then we have
	\begin{align*}
	- \int_M g^{i\bar j}g^{k\bar l}\cdot \nabla_{\bar j}u \cdot \nabla_{\bar l}\nabla_{ k}\nabla_{ i}  u\cdot \om^n=\int_M |\nabla^g\nabla^g u|_g^2 \om^n.
	\end{align*}
	In which, $\nabla^g$ is the covariant differentiation w.r.t. the cone metric $g$. 
\end{lem}
\begin{proof}

We let $\chi_{\ep}$ be the smooth cut off function supported outside the $\rho$-tubular neighbourhood of the divisor with the properties that such that 
\begin{align*}
|\nabla\chi_{\ep}|=\eps \cdot O(r^{-1}).
\end{align*}
Multiplying the integrand in \eqref{13term} with the cut-off function $\chi_{\ep}$, we compute by integration by parts,
\begin{align*}
&\int_M g^{i\bar j}g^{k\bar l}\cdot \nabla_{\bar j}u \cdot \nabla_{\bar l}\nabla_{ k}\nabla_{ i}  u\cdot \chi_{\ep} \cdot \om^n\\
&=\int_M |\nabla^g\nabla^g u|_g^2  u\cdot \chi_{\ep} \cdot \om^n
+\int_M g^{i\bar j}g^{k\bar l}\cdot  \nabla_{\bar j}u \cdot\nabla_{ k}\nabla_{ i}  u\cdot \nabla_{\bar l} \chi_{\ep} \cdot \om^n.
\end{align*}

We denote 
\begin{align*}
I_\eps=\int_M g^{i\bar j}g^{k\bar l}\cdot  \nabla_{\bar j}u \cdot\nabla_{ k}\nabla_{ i}  u\cdot \nabla_{\bar l} \chi_{\ep} \cdot \om^n.
\end{align*}
We know from $u\in C_0^{1,\a,\b}$ and $g$ is $C^{0,\a,\b}$, that the growth rate of the following terms near the divisor,
\begin{align*}
&|\p_{z^1} u|_g=O(r^{\a\b}),\quad |\p_{z^i} u|_g=O(1),\quad\forall 2\leq i\leq n,\\
&| \nabla \chi_{\ep} |_g=\eps \cdot O(r^{-\b}),\quad \om^n=O(r^{2\b-2}).
\end{align*}
Then from the assumption the growth order the the integrand is $-\kappa+\b-2>-2$ and thus $I_\eps\rightarrow 0$, as $\eps\rightarrow 0$.
\end{proof}

\begin{cor}\label{hvfonmpt}Suppose $u\in C_\mathbf{w}^{4,\a,\b}(\om)$ is in the kernel of ${\mathbb{L}\mathrm{ic}}_{\om}$ and $\om$ is a cscK cone metric. Then we have
	\begin{align*}
	 |\nabla^g\nabla^g u|_g=0 \text{ on }M.
	\end{align*} 
\end{cor}
\begin{proof}
According to Proposition \ref{nnu} we have the growth rate of 
$|\nabla_k\nabla_i u|_g$ is 	\begin{align*}
-\kappa=\a\b-\b>-\a\b-\b.
\end{align*} 
Thus the conclusion follows from \lemref{13term} and \lemref{hvfonm}.
\end{proof}

Next we are going to prove that the kernel of ${\mathbb{L}\mathrm{ic}}_{\om}$ generate a holomorphic vector field on $X$ tangential to the divisor. 
\begin{lem}
Suppose $u\in C_\mathbf{w}^{4,\a,\b}(\om)$ is in the kernel of ${\mathbb{L}\mathrm{ic}}_{\om}$ and $\om$ is a cscK cone metric. Then the lifting of $u$ by the metric $\omega$ is a holomorphic vector field on $X$ tangential to $D$.
\end{lem}
\begin{proof}
According to \corref{hvfonmpt}, when $\nabla^g\nabla^g u=0$ on the regular part $M$. Then $X=g^{i\bar j}\frac{\p u}{\p z^{\bar j}}\frac{\p }{\p z^{i}}$ is a holomorphic vector field on $M$.
Since $u\in C^{2,\a,\b}$, $X$ vanishes along $D$ and the tangent directions could be extended to $D$, which extends $X$ to a holomorphic vector field on the whole $X$.
 \end{proof}


\section{Bifurcation of the $J$-twisted path}\label{bifurcation}
We now introduce a new continuity path, providing a path towards a cscK metric,
\begin{align}\label{path}
\Phi(t,\vphi)=t(S-\underline S)- (1-t)(\frac{\om^n}{\om_\vphi^n}-1).
\end{align} 
The twisted term also appears in \cite{MR3412393,MR3085100}.
It is clear that this path is the critical point of the following perturbed functional
\begin{align}\label{perturbedenergy}
E(\vphi(t))=t \nu(\vphi(t))+(1-t)J(\vphi(t)).
\end{align}


We would consider the end-point at $t=1$.
When $\vphi,P\in C^{2,\a,\b}$, we consider the continuity path with different parameter, which is written as the fully non-linear operator,
\begin{align}\label{vphippath}
\Phi(t,\vphi(t))=-\tri_{\vphi(t)} P(t)+g_{\vphi(t)}^{i\bar j} R_{i\bar j}(\om_\theta)-\underline S_\b-(1-t)(\frac{\om^n}{\om_\vphi^n}-1).
\end{align}
and
\begin{align}
\log\frac{\om_{\vphi(t)}^n}{\om_\theta^n}=P.
\end{align}
Or as a formally fourth order operator outside the divisor, it is 
\begin{align*}
\Phi(t,\vphi(t))=S(\om_\vphi)-\underline S_\b-(1-t)(\frac{\om^n}{\om_\vphi^n}-1).
\end{align*}


We denote by $\mathcal O$ an orbit of the cscK cone metrics. We minimise the functional $J$ over $\mathcal O$.
\begin{prop}\label{minimiseJ}
	The functional $J$ has a unique minimiser on $\mathcal O$.
\end{prop}
\begin{proof}
	Since the orbit $\mathcal O$ is finite dimensional, we could find the critical points of $J$ over $\mathcal O$. Let $\theta$ be such critical point.
	It follows from \lemref{Jconvex} and again the reductivity of the automorphism group, i.e. any cscK cone metric on $\mathcal O$ could be connected to $\theta$ by the cone geodesic
	$exp(t \Re X)$
	for some holomorphic vector field $X$. From Proposition \ref{Jconvex}, i.e. $J$ is convex along the $C_{\mathbf w}^{1,1,\b}$ cone geodesic, we see that the critical point $\om_\theta$ is actually the unique minimiser of $J$ over $\mathcal O$.
\end{proof}
Now the minimiser $\\theta$ is a cscK cone metric and $\la_\theta$ is the K\"ahler cone potential such that
$$\theta:=\om_{\theta}=\om+i\p\bar\p\la_\theta.$$
We also denote by $H_\theta$ the kernel space of the Lichnerowicz operator, 
\begin{align}
Lic_{\theta}(u)=\Delta_\theta^2 u +  u^{i\bar j}T_{i\bar j}(\theta).
\end{align}
Thus at $\theta$, for any $u\in H_\theta$,
\begin{align}\label{propeetyofgauge}
\frac{1}{V}\int_{M}u \cdot (\frac{\om^n}{\om_\vphi^n}-1) \om_\vphi^n=0.
\end{align}




Now we prove \thmref{Bifurcation} at the cscK cone metric $\theta$.
\begin{thm}\label{bifrucationattheta} Assume the notations above. There exists a parameter $\tau>0$ such that $\vphi(t)$ is extended uniquely to a smooth one-parameter family of solutions of the $J$-twisted path \eqref{vphippath} on $(1-\tau, 1]$.
\end{thm}

We consider the continuity path \eqref{vphippath} with $\vphi(1)=\lambda_\theta$ and construct the bifurcation at $\vphi(1)$.

\textbf{Step 1.}
The linearised operator at $t=1$ is the Lichnerowicz operator. Then we decompose the whole space $C_\mathbf{w}^{4,\a,\b}(\theta)$ into the direct sum of $H_\theta$ and its orthogonal space $H^\bot_\theta$, i.e.
$$C_\mathbf{w}^{4,\a,\b}(\theta)=H_\theta\oplus H^\bot_\theta.$$ The path $\vphi(t)-\la_\theta$ is then decomposed into
\begin{align*}
\vphi(t)-\la_\theta=\vphi^\parallel+\vphi^\bot.
\end{align*}
$\Phi(t,\vphi(t)$ vanishes at $$(t,\vphi^\parallel,\vphi^\bot)=(1,0,0),$$ since $\theta$ is a cscK cone metric.
While, we let $P$ denote the projection from $C_\mathbf{w}^{4,\a,\b}(\theta)$ to $H_\theta$ and decompose the linear operator $\Phi$ into two parts. 

\textbf{Step 2.}
We first consider the vertical part,
\begin{align}\label{phibot}
\Phi^\bot(t,\vphi^\parallel,\vphi^\bot)=(1-P)[S(\vphi(t))-\underline S_\b]-(1-t)(\frac{\om^n}{\om_\vphi^n}-1).
\end{align}
Meanwhiles, its derivative on $\vphi^\bot$ at $(1,0,0)$ is for any $\xi\in H^\bot_\theta$,
\begin{align*}
\delta_{\vphi^\bot}\Phi^\bot\vert_{(1,0,0)}(\xi)=(1-P)Lic_{\theta}\xi 
\end{align*}
which is invertible according to \thmref{Fredholm}. Therefore, we are able to use the implicit function theorem on $C_\mathbf{w}^{4,\a,\b}(\theta)$ space to conclude that there is small neighbourhood $U$ near $(t,\vphi^{\parallel})=(1,0)$ such that when $(t,\vphi^{\parallel})\in U$,
\begin{align*}
\vphi^\bot: U \subset(1-\tau,1]\times H_\theta &\rightarrow H_\theta^\bot,\\
(t, \vphi^\parallel)&\mapsto\vphi^\bot(t, \vphi^\parallel)
\end{align*} solves
\begin{align}\label{phibot}
\Phi^\bot(t,\vphi^\parallel,\vphi^\bot(t,\vphi^{\parallel}))=0,
\end{align} with $\vphi^\bot(1,0)=0.$

\textbf{Step 3.}
We differentiate \eqref{phibot} along the kernel direction,
at $(t,\vphi^{\parallel})=(1,0)$,  is for any $u\in H^\bot_\theta$,
\begin{align*}
0=\delta_{\vphi^{\parallel}} \Phi^{\bot}\vert_{(1,0)}(u)=(1-P)[-Lic_\theta(\delta_{\vphi^\parallel}\vphi^\bot\vert_{(1,0)}(u))].
\end{align*}
Since both the imagine of $1-P$ and $Lic_\theta$ are in $H_\theta^\bot$, we conclude that
\begin{align}\label{parallelbot}
\delta_{\vphi^{\parallel}} \vphi^{\bot}\vert_{(1,0)}(u) &=0, \forall u\in H_\theta.
\end{align}
Furthermore, taking $t$ derivative of the path $\Phi^\bot=0$, we have
\begin{align}
0=\frac{\p \Phi^\bot}{\p t}&=(1-P)\{-\tri^2_\vphi[\frac{\p \vphi^{\parallel}}{\p t}+ \frac{\p \vphi^{\bot}}{\p t}+\delta_{\vphi^{\parallel}} \vphi^{\bot}(\frac{\p \vphi^{\parallel}}{\p t})]\nonumber\\
&-<Ric(\vphi),\p\bar\p [\frac{\p \vphi^{\parallel}}{\p t}+ \frac{\p \vphi^{\bot}}{\p t}+\delta_{\vphi^{\parallel}} \vphi^{\bot}(\frac{\p}{\p t}\vphi^{\parallel})]>_\vphi\nonumber\\
&+(\frac{\om^n}{\om_\vphi^n}-1)+(1-t)\frac{\om^n}{\om_\vphi^n}\tri_\vphi[\frac{\p \vphi^{\parallel}}{\p t}+ \frac{\p \vphi^{\bot}}{\p t}+\delta_{\vphi^{\parallel}} \vphi^{\bot}(\frac{\p \vphi^{\parallel}}{\p t})]\}\nonumber.
\end{align}
At $(t,\vphi^{\parallel})=(1,0)$, it reads
\begin{align*}
0=\frac{\p \Phi^{\bot}}{\p t}\vert_{(1,0)} &=(1-P)[- Lic_\theta( \frac{\p \vphi^{\bot}}{\p t}\vert_{(1,0)} )+\frac{\om^n}{\theta^n}-1].
\end{align*}
From \lemref{propeetyofgauge}, $\frac{\om^n}{\om_\vphi^n}-1 \in H^\bot_\theta$, we have
\begin{align}\label{vphibottheta}
- Lic_\theta( \frac{\p \vphi^{\bot}}{\p t}\vert_{(1,0)} )+\frac{\om^n}{\theta^n}-1=0.
\end{align}

\textbf{Step 4.}
We next consider the horizontal operator on the finite dimensional space $H_\theta$,
\begin{align*}
\Phi^\parallel(t,\vphi^\parallel)&=P[S(\vphi(t))-\underline S_\b-(1-t)(\frac{\om^n}{\theta^n}-1)].
\end{align*} 
Note that $\Phi^\parallel(t,\vphi^\parallel)=0$ at $(1,0)$.
From \eqref{vphibottheta}, we see that $\frac{\p \Phi^{\parallel}}{\p t}\vert_{(1,0)}=0$.
Then we consider the modified functional
\begin{align*}
\tilde \Phi^\parallel(t,\vphi^\parallel)=\frac{\Phi^\parallel(t,\vphi^\parallel)}{t-1}.
\end{align*}
Again, we let $\xi=\frac{\p\vphi^\bot}{\p t}\vert_{(t=1,\vphi^\parallel=0)},$ then compute for all $u\in H_\theta$,
\begin{align*}
\delta_{\vphi^\parallel}\tilde \Phi^\parallel\vert_{(1,0)}(u)
&=\delta_{\vphi^\parallel}\frac{\p}{\p t} \Phi^\parallel\big\vert_{(1,0)}(u)\\
&=P[<\p\bar\p u,\p\bar\p\tri_\theta \xi>_\theta
+\tri_\theta<\p\bar\p u,\p\bar\p \xi>_\theta
+<\p\bar\p \tri_\theta u,\p\bar\p \xi>_\theta\\
&+u^{i\bar l}\theta^{k \bar j}(Ric_\theta)_{i\bar j} \xi_{k\bar l}+\theta^{i\bar l}u^{k\bar l}(Ric_\theta)_{i\bar j} \xi_{k\bar l}-\frac{\om^n}{\theta^n}\tri_\theta u].
\end{align*}
So we obtain from \lemref{longcomputation}, where we carry the detailed computation, that
\begin{align*}(\delta_{\vphi^\parallel}\tilde \Phi^\parallel \vert_{(1,0)}(u), v)_{L^2(\theta)}
\end{align*}
is positive definite, and the equality holds if and only if $u=0$.

\textbf{Step 5.}
Thus we are able to apply the implicit function theorem to $\tilde \Phi^\parallel(t,\vphi^\parallel)$ over $C_\mathbf{w}^{4,\a,\b}(\theta)$ to construct a solution $\vphi^\parallel(t)\in C_\mathbf{w}^{4,\a,\b}(\theta)$ with $t\in (1-\tau,1]$ such that  $\tilde\Phi^\parallel(t,\vphi^\parallel(t))=0,$ and $\vphi(1)=0$.
Then the original nonlinear equation is solved as
$
\Phi(t,\vphi^\parallel(t),\vphi^\bot(t,\vphi^{\parallel}(t)))=0.
$
And moreover, 
\begin{align*}
\vphi(t)=\la_\theta+\vphi^\parallel(t)+\vphi^\bot(t,\vphi^{\parallel}(t))
\end{align*} is the solution to the continuity path on $t\in (1-\tau,1]$ with $\vphi(1)=\la_\theta.$

\begin{lem}\label{longcomputation}The following identity holds,
	\begin{align*}
	\int_M \delta_{\vphi^\parallel}\tilde \Phi^\parallel\vert_{(1,0)}(u)\cdot v\cdot \theta^n
	=\int_{M_\eps} <\p u, \bar\p v>_\theta\om^n.
	\end{align*}
\end{lem}
\begin{proof}
	On the regular part $M$, we use the direct computation (see Lemma 2.4 in \cite{arxiv:1506.01290}) to see that
	\begin{align*}
	&\delta_{\vphi^\parallel}\tilde \Phi^\parallel\vert_{(1,0)}(u)\\
	&=P[Lic_\theta<\p u, \bar\p \xi>_\theta-<\p u, \bar\p Lic_\theta( \xi )>_\theta-\frac{\om^n}{\theta^n}\tri_\theta u].
	\end{align*}
	From \eqref{vphibottheta}, we have
	\begin{align*}
	RHS&=-P[<\p u, \bar\p (\frac{\om^n}{\theta^n}-1)>_\theta+\frac{\om^n}{\theta^n}\tri_\theta u].
	\end{align*}
	We multiply with $v$ and thc cut-off function (see \lemref{hvfonm}) and integrate over $X$, we have
	\begin{align*}
	\int_{X} \delta_{\vphi^\parallel}\tilde \Phi^\parallel\vert_{(1,0)}(u)\cdot v\cdot\chi_{\eps} \cdot\theta^n
	=-\int_M [<\p u, \bar\p (\frac{\om^n}{\theta^n})>_\theta+\frac{\om^n}{\theta^n}\tri_\theta u]\cdot v\cdot\chi_{\eps} \cdot\theta^n.
	\end{align*}
	The second term is
	\begin{align*}
	&-\int_M \frac{\om^n}{\theta^n}\tri_\theta u \cdot v\cdot \chi_{\eps}\cdot \theta^n\\
	&=\int_M <\p u, \bar\p (\frac{\om^n}{\theta^n})>_\theta\cdot v\cdot \chi_{\eps}\cdot\theta^n+\int_M  <\p u, \bar\p v>_\theta\cdot\chi_{\eps}\cdot\om^n\\
	&+\int_M  <\p u, \bar\p \chi_{\eps}>_\theta\cdot v\cdot\om^n.
	\end{align*}
	Then we have
		\begin{align*}
	&\int_{X} \delta_{\vphi^\parallel}\tilde \Phi^\parallel\vert_{(1,0)}(u)\cdot v\cdot\chi_{\eps} \cdot\theta^n\\
	&=\int_M  <\p u, \bar\p v>_\theta\cdot\chi_{\eps}\cdot\om^n
	+\int_M  <\p u, \bar\p \chi_{\eps}>_\theta\cdot v\cdot\om^n.
	\end{align*}
	 Thus the lemma is proved as $\eps\rightarrow 0$, since $u,v,\frac{\om^n}{\theta^n}\in C^{2,\a,\b}$.
\end{proof}












\end{document}